\documentclass[11pt, reqno]{amsart}
\usepackage{amsfonts,amsmath,amsthm,amssymb,stmaryrd}
\usepackage{graphicx}\usepackage{bm}
\usepackage{esint}
\usepackage{hyperref}
\usepackage[RGB]{xcolor}
\usepackage{epstopdf}
\newtheorem{Lemma}{Lemma}[section]
\newtheorem{Theorem}{Theorem}[section]
\newtheorem{Definition}{Definition}[section]
\newtheorem{Proposition}{Proposition}[section]
\newtheorem{Remark}{Remark}[section]

\numberwithin{equation}{section} \allowdisplaybreaks

\usepackage[left=30 mm, right=30 mm,top=30 mm, bottom=30 mm]{geometry}

\def\bega{\begin{array}}
\def\enda{\end{array}}
\def\begi{\begin{itemize}}
\def\endi{\end{itemize}}
\def\be{\begin{equation}}
\def\beq{\begin{equation}}
\def\bel{\begin{equation}\label}
\def\eeq{\end{equation}}

\newcommand{\LC}{\left(}
\newcommand{\RC}{\right)}

\newcommand{\bea}{\begin{eqnarray}}
\newcommand{\eea}{\end{eqnarray}}
\newcommand{\beann}{\begin{eqnarray*}}
\newcommand{\eeann}{\end{eqnarray*}}

\begin{document}
\vskip 0.2cm

\title[Finsler type Lipschitz metric for a wave equation]{\bf  A Finsler type Lipschitz optimal transport metric for a quasilinear wave equation}

\author[H. Cai]{Hong Cai}
\address{Hong Cai \newline
Department of Mathematics and Research Institute for Mathematics and Interdisciplinary Sciences, Qingdao University of Science and Technology, Qingdao, Shandong, P.R. China, 266061.}
\email{caihong19890418@163.com}

\author[G. Chen]{Geng Chen}
\address{Geng Chen \newline
Department of Mathematics, University of Kansas, Lawrence, KS 66045, USA.}
\email{gengchen@ku.edu}

\author[Y. Shen]{Yannan Shen}
\address{Yannan Shen \newline
Department of Mathematics, University of Kansas, Lawrence, KS 66045, USA.}
\email{yshen@ku.edu}


\maketitle

\begin{abstract}
{\small We consider the global well-posedness of weak energy conservative solution to a general quasilinear wave equation through variational principle, where the solution may form finite time cusp singularity, when energy concentrates. As a main result in this paper, we construct a Finsler type optimal transport metric, then prove that the solution flow is Lipschitz under this metric.
We also prove a generic regularity result by applying Thom's transversality theorem, then find piecewise smooth transportation paths among a dense set of solutions. The results in this paper are for large data solutions, without restriction on the size of solutions.
}
 \bigbreak
\noindent


\end{abstract}

\tableofcontents

\section{Introduction}
Consider a class of quasilinear wave equations derived from a variational principle whose action is a quadratic function of derivatives of the field with coefficients depending on both the field and independent variables
\begin{equation}\label{field}
\delta\int A_{\mu\nu}^{ij}(\mathbf{x},u)\frac{\partial u^\mu}{\partial x_i}\frac{\partial u^\nu}{\partial x_j}d\mathbf{x}=0,
\end{equation}
where we use the summation convention, see \cite{AH2007}.
Here ${\bf x}\in \mathbb{R}^{d+1}$ are the space-time variables and $u:\mathbb{R}^{d+1}\to R^n$ are the dependent variables. We assume the coefficients $A_{\mu\nu}^{ij}:\mathbb{R}^{d+1}\times \mathbb{R}^n\to \mathbb{R}$ are smooth and satisfy $A_{\mu\nu}^{ij}=A_{\nu\mu}^{ij}=A_{\mu\nu}^{ji}$.
The Euler-Lagrange equations associated with \eqref{field} are
\begin{equation}\label{EulerL}
\frac{\partial}{\partial x_i}\Big(A_{k\mu}^{ij}\frac{\partial u^\mu}{\partial x_j}\Big)=\frac{1}{2}\frac{\partial A_{\mu\nu}^{ij}}{\partial u^k}\frac{\partial u^\mu}{\partial x_i}\frac{\partial u^\nu}{\partial x_j}.
\end{equation}

In this paper, we consider the special case of \eqref{field} when  $n=1$ and $d=1$, where the Euler-Lagrange equation \eqref{EulerL} reads that
\begin{equation}\label{EL}
(A^{11}u_t+A^{12}u_x)_t+(A^{12}u_t+A^{22}u_x)_x=\frac{1}{2}
\big(\frac{\partial A^{11}}{\partial u}u_t^2+2\frac{\partial A^{12}}{\partial u}u_t u_x+\frac{\partial A^{22}}{\partial u}u_x^2\big).
\end{equation}
Moreover, assume the coefficients satisfy
\begin{equation*}
(A^{ij})_{2\times 2}=\left(
  \begin{array}{cc}
   \alpha^2 & \beta \\
   \beta & -\gamma^2 \\
\end{array}
\right)(x,u),
\end{equation*}
then equation \eqref{EL} exactly gives the following nonlinear variational wave equation,
\begin{equation}\label{vwl}
(\alpha^2 u_t+\beta u_x)_t+(\beta u_t-\gamma^2 u_x)_x=\alpha\alpha_u u_t^2+\beta_u u_tu_x-\gamma\gamma_u u_x^2,
\end{equation}
with initial data
\begin{equation}\label{ID}
u(x,0)=u_0(x)\in H^1,\quad
u_t(x,0)=u_1(x)\in L^2.
\end{equation}
Here the variable $t\geq 0$ is time, and $x$ is the spatial coordinate. The coefficients $\alpha=\alpha(x,u),\beta=\beta(x,u),\gamma=\gamma(x,u)$ are smooth functions on $x$ and $u$, satisfying that, there exist positive constants $\alpha_1,\alpha_2,\beta_2,\gamma_1$ and $\gamma_2$, such that for any $z=(x,u)$,
\begin{equation}\label{con}
\begin{cases}
0<\alpha_1\leq \alpha(z)\leq\alpha_2,~~|\beta(x,u)|\leq \beta_2,~~0<\gamma_1\leq\gamma(z)\leq\gamma_2,\\
\displaystyle\sup_z\{|\nabla\alpha(z)|,|\nabla\beta(z)|,|\nabla\gamma(z)|\}<\infty,~~\forall z\in\mathbb{R}^2.
\end{cases}\end{equation}
Moreover, in this paper we always assume that the following generic condition is satisfied
\begin{equation}\label{gencon}
\partial_u \lambda_\pm(x,u)=0\Rightarrow\partial_{uu}\lambda_\pm(x,u)\neq0 \quad{\rm or}\quad \partial_{ux}\lambda_\pm(x,u)\neq0.
\end{equation}
Then system \eqref{vwl} is strictly hyperbolic with two eigenvalues
\begin{equation}\label{lambda}
\lambda_{-}:=\frac{\beta-\sqrt{\beta^2+\alpha^2\gamma^2} }{\alpha^2}<0,\qquad
\lambda_{+}:=\frac{\beta+\sqrt{\beta^2+\alpha^2\gamma^2} }{\alpha^2}>0.
\end{equation}
In this paper, we will always call waves in the families of $\lambda_-$ and $\lambda_+$ as backward and forward waves, respectively.
By \eqref{con}, $-\lambda_{-}(x,u)$ and $\lambda_{+}(x,u)$ are both smooth on $x$ and $u$, bounded and uniformly positive.

Solutions of \eqref{vwl}-\eqref{ID} may form finite time cusp singularity, see examples in \cite{BC,BHY,GHZ}. The existence and uniqueness of global-in-time energy conservative H\"older continuous (weak) solution have been established by Hu in \cite{H} and the authors in \cite{CCDS}, respectively. In this paper, we address the Lipschitz continuous dependence and generic regularity of solution.

\subsection{Physical background}

Let's first introduce various physical models related to equations  \eqref{field} and \eqref{vwl}.

\paragraph{\bf Variational wave equation}
A particular physical example leading to \eqref{field} and \eqref{vwl} is the motion of a massive director field in a nematic liquid crystal. A nematic liquid crystal can be described by a director field of unit vectors $\mathbf{n}\in\mathbb{S}^2$ describing the orientation of rod-like molecules.
In the regime in which inertia effects
dominate viscosity, the propagation of orientation waves in the director field is modeled by the least action principle (see \cite{AH2009})
\begin{equation}\label{ofdelta}
\delta\int\Big(\partial_t \mathbf{n}\cdot \partial_t \mathbf{n}-W(\mathbf{n},\nabla\mathbf{n}) \Big)\,d\mathbf{x}dt=0,\quad \mathbf{n}\cdot\mathbf{n}=1,
\end{equation}
where
\begin{equation}\label{O_F}
W(\mathbf{n},\nabla\mathbf{n})=K_1|\mathbf{n}\times(\nabla\times\mathbf{n})|^2
+K_2(\nabla \cdot \mathbf{n})^2+K_3(\mathbf{n}\cdot\nabla\times\mathbf{n})^2
\end{equation}
is the well-known Oseen-Franck potential energy density.
Here $K_1, K_2$ and $K_3$ are positive elastic constants. This variational principle is in the form of \eqref{field}.

The study of \eqref{ofdelta} starts from a simplest case consisting of planar deformations depending on a single space variable $x\in \mathbb R$, i.e. when
$
\mathbf{n}=(\cos u(x,t),\sin u(x,t),0).
$ In this case,
the functional $W(\mathbf{n},\nabla\mathbf{n})$ vastly simplifies to $W(\mathbf{n},\nabla\mathbf{n})=(K_1\cos^2u+K_2 \sin^2u)u_x^2$.
Then the dynamics are described by the variational principle
\begin{equation*}
\delta \int(u_t^2-c^2(u)u_x^2)\,dx\,dt=0,
\end{equation*}
with the wave speed $c$ given by $c^2(u)=K_1\cos^2 u+K_2\sin^2 u$.
Thus, the Euler-Lagrange equation for this variational principle results the variational wave equation
\begin{equation}\label{vwe}
u_{tt}-c(u)(c(u)u_x)_x=0.
\end{equation}
It is known that solutions for the initial value problem of \eqref{vwe} generically have finite time cusp singularity \cite{BC,BHY,GHZ}.
The global existence of H\"older continuous energy conservative solution was established by Bressan and Zheng in \cite{BZ}. To select a unique solution after singularity formation, one needs to add an additional admissible condition, such as the energy conservative condition. In \cite{BCZ}, uniqueness of energy conservative solution has been established by Bressan, Chen and Zhang. Later, these results have been extended to \eqref{ofdelta} in \cite{CCD,CZZ,ZZ10,ZZ11}. Also see other existence results for \eqref{vwe}: existence of conservative solution for more general initial data in \cite{HR}; and existence of dissipative solution with monotonic $c(\cdot)$ in \cite{BH,ZZ03}. 

The breakthrough on the Lipschitz continuous dependence happened later in \cite{BC2015} by Bressan and Chen, where the solution flow was proved to be Lipschitz continuous on a new Finsler type optimal transport metric. In fact, the solution flow fails to be Lipschitz continuous under existing metrics, such as the Sobolev metric or Wasserstein metric.

The main target of this paper is to extend this Lipschitz continuous dependence result to a much more general equation \eqref{vwl}. Here, the equation \eqref{vwl} is a general quasilinear wave equation with various physical backgrounds (see \cite{AH2007}), including \eqref{vwe} as an example, when $\alpha=1, \beta=0$ and $\gamma=c(u)$.

\bigskip

\paragraph{\bf The $O(3)~\sigma$-Model}
Another background model we like to introduce is the $O(3)~\sigma$-model.
In fact, one motivation to consider \eqref{vwl} in this paper is for the future study of multi-d solutions of \eqref{field} with radial symmetry, which is related to the $O(3)~\sigma$-model.

One of the simplest nontrivial models of quantum field theory is based on the $(2+1)$-dimensional Lorentz invariant $O(3)~\sigma$-model. This model is also known as the wave map flow from the $(2+1)$-dimensional Minkowski space to the sphere $\mathbb{S}^2\subset\mathbb{R}^3$: for a wave map $\Phi:\mathbb{R}^{2+1}\mapsto \mathbb{S}^2$ with the Lagrangian density
$$\mathcal{L}[\Phi]=\frac{1}{2}\partial_\alpha\Phi\cdot\partial_\alpha\Phi ~m^{\alpha \beta},$$
where $m^{\alpha \beta}$ is the Minkowski metric, see \cite{RS}. The Euler-Lagrange equations are given by
\begin{equation}\label{ele}
\Box\Phi=-\Phi(\partial^\alpha\cdot\partial_\alpha\Phi).
\end{equation}
Here, we are particularly interested in solutions with the following symmetry.
Consider solutions with $\lambda$-equivariant symmetry, or $\lambda$-corotational, which correspond to equivariant maps that in local coordinates take the form
$$\Phi(r,t,\theta)=(u(r,t),\lambda\theta)\hookrightarrow(\sin u\cos \lambda\theta, \sin u\sin \lambda\theta, \cos u)\in\mathbb{S}^2,$$
where $u$ is the colatitude measured from the north pole of the sphere and the metric on $\mathbb{S}^2$ is given by  $ds^2=du^2+\sin^2 u\, d\theta^2$. In this case \eqref{ele} is reduced to the following scalar equation
\beq\label{o3s}
u_{tt}=u_{rr}+\frac{1}{r}u_r-\lambda^2\frac{\sin(2u)}{2r^2}
\eeq
with $\lambda\in \mathbb{N}^+$. This equation has various physical backgrounds including general relativity and Yang-Mills field \cite{RS}.
There are many intensive deep  studies on \eqref{o3s}, for example the singularity formaiton example in \cite{RS} with $\lambda\geq 4$ and some global existence result in \cite{JL} with $\lambda\geq 2$. For more references, we refer readers to the review in the introduction of \cite{DG} and references of
\cite{JL,RS}.

The general equation \eqref{field} includes a quasilinear version of $O(3)~\sigma$-model.
Since we only consider the 1-d case of \eqref{field}, it seems unrelated to the $O(3)~\sigma$-model. Now we explain our motivation.

First, in general, in order to study the quasilinear wave equations \eqref{field} in multiple space dimension with symmetry, it is reasonable to start building mathematical tools for the 1-d case. We believe that this preparation will be very helpful later for the multi-d case with symmetry.

Secondly, in the next part, we will introduce a new observation found by the second author showing the relation between full Poiseuille flow of nematic liquid crystals via Erickson-Leslie model and the $O(3)~\sigma$-model \eqref{o3s}.
\bigskip

\paragraph{\bf Poiseuille flow of nematic liquid crystal via Erickson-Leslie model and \eqref{o3s}}
To get a more physical relevant model for nematic liquid crystals than \eqref{ofdelta}, one needs to consider a coupled system describing both ``liquid'' and ``crystal'' properties, such as the  Erickson-Leslie model, see \cite{L}.
We consider the flow of a liquid-crystal due to a pressure gradient in a stationary capillary tube of radius $R$. The Poiseuille laminar flow concerns  solutions of the form, in terms of cylindrical coordinates $(r,\phi,z)$ of the tube,
 \[{\bf v}=(v_{r}, v_{\phi}, v_z)=(0,0,v(r,t)),\quad {\bf n}=(n_r,n_{\phi},n_z)=(\sin u(r,t), 0, \cos u(r,t)),\]
where $u$ is the angle between ${\bf n}$ and the positive $z$-axis.
Now this model  via Erickson-Leslie model is
\begin{align}\label{time_poi}\begin{split}
\rho  v_t=&a+\frac{1}{r}\Big(rg(u)v_r+rh(u) u_t\Big)_r,\\
\sigma  u_{tt}+\gamma_1 u_t=&c( u)(c(u) u_r)_r+\frac{c^2( u)}{r} u_r
-\frac{ K_1\sin(2 u)}{2r^2}-\frac{\gamma_1+\gamma_2\cos(2 u) }{2}v_r,
\end{split}
\end{align}
where the constant $a$ is the gradient of pressure along the $z$-axis, and
 \begin{align}\label{fgh}\begin{split}
 g( u)=&\alpha_1\sin^2 u\cos^2 u+\frac{\alpha_5-\alpha_2}{2}\sin^2\theta+\frac{\alpha_3+\alpha_6}{2}\cos^2 u+\frac{\alpha_4}{2},\\
c^2( u)=&K_1\cos^2 u+K_3\sin^2 u,\quad h( u)=\alpha_3\cos^2 u-\alpha_2\sin^2 u.
 \end{split}
 \end{align}
Here $K_1,K_3$ are defined in \eqref{O_F}; $\alpha_j$'s are the Lesile's parameters satisfying some  empirical relations (p.13, \cite{L}). Here $\rho,\sigma$ are positive constants, and $\gamma_1$, $\gamma_2$ are nonnegative constants. See \cite{L} for more details.

For example, when $K_1=K_3=1$, it is easy to see that $c(u)=1$. So $
\eqref{time_poi}_2$ becomes semilinear on $u$.
Without loss of generality, we can choose $\sigma=1$. One exactly gets \eqref{o3s} with $\lambda=1$, by
further assuming $\gamma_1=\gamma_2=0$ on $
\eqref{time_poi}_2$. Other models for nematic Liquid crystals with
$\lambda>1$ and $\lambda\in\mathbb{N}^+$ can also been established.

In general, the main part of $
\eqref{time_poi}_2$ on $u$ can be considered as a quasilinear version of
\eqref{o3s}. The understanding of the wave part in $
\eqref{time_poi}_2$  is vital in studying \eqref{time_poi}.

In a recent paper \cite{CHL}, the cusp singularity formation and global existence of H\"older continuous solution for the 1-d model
of \eqref{time_poi} has been established. By 1-d model we mean
\begin{align}\label{time_poi2}
\begin{split}
 \rho v_t=&a+\Big(g(u)v_x+h(u)u_t\Big)_x,\\
\nu u_{tt}+\gamma_1 u_t=&c(u)(c(u)u_{x})_x
 -h(u)v_x,
\end{split}
\end{align}
whose solution is the solution of  Erickson-Leslie model satisfying the following symmetry:
\[{\bf v}=(0,0,v(x,t))\quad {\bf n}=(\sin u(x,t), 0, \cos u(x,t)),\]
where $(x,t)\in \mathbb R\times \mathbb R^+$.
The existence framework in \cite{BZ} for the variational wave equation \eqref{vwe} serves as the basis for the existence result in \cite{CHL}  for \eqref{time_poi2}, because the major wave parts on $u$ in these two equations are same. Furthermore, the recent study shows that this 1-d framework will very likely direct to some interesting results for the exterior problem of \eqref{time_poi} out of a small cylinder including the center line.

Using a similar strategy, we like to first study the 1-d equation \eqref{vwl} in this paper, priori to the future study on quasilinear wave equations \eqref{field} in the multi-d case with symmetry.

\subsection{Main results of this paper}
The solution of  \eqref{vwl}--\eqref{ID} generically forms finite time cusp singularities, due to the quasilinear structure in the equation \cite{BC, BHY,GHZ}. In general one can only consider weak solutions.

\begin{Definition}[Weak solution]\label{weakdef}
The function $u=u(x,t)$, defined for all $(x,t)\in\mathbb{R}\times\mathbb{R}^+$, is a {\bf weak solution} to the Cauchy problem \eqref{vwl}--\eqref{ID} if it satisfies following conditions.
\begin{itemize}
\item[(i)] In the $x$-$t$ plane, the function $u(x,t)$ is locally H\"older continuous
with exponent $1/2$.  The function $t\mapsto u(\cdot,t)$ is
continuously differentiable as a map with values in $L^p_{\rm
loc}$, for all $1\leq p<2$. Moreover, it is Lipschitz continuous with respect to
(w.r.t.)~the $L^2$ distance, that is, there exists a constant $L$ such that
\begin{equation*}
 \big\|u(\cdot,t)-u(\cdot,s)\big\|_{L^2}
 \leq L\,|t-s|,
\end{equation*}
for all $t,s\in\mathbb R^+$.
\item[(ii)] The function $u(x,t)$ takes on the initial conditions in (\ref{ID})
pointwise, while their temporal derivatives hold in  $L^p_{\rm
loc}\,$ for $p\in [1,2)\,$.
\item[(iii)] The equations (\ref{vwl}) hold in distributional sense, that is
\begin{equation*}
\int\int\big[\varphi_t(\alpha^2 u_t+\beta u_x)+\varphi_x(\beta u_t-\gamma^2 u_x)+\varphi(\alpha \alpha_u u_t^2+\beta_u u_t u_x-\gamma\gamma_u u_x^2)\big]\,dx\,dt=0
\end{equation*}
for any test function $\varphi\in
C^1_c(\mathbb R\times \mathbb R^+)$.
\end{itemize}
\end{Definition}

The existence of weak solution, defined above, has been established by \cite{H}, where the main idea is to study a semi-linear system under characteristic coordinates. This method was first used in \cite{BZ} for the variational wave equation \eqref{vwe}.
To select a unique solution after singularity formation, one needs to add an additional admissible condition, such as the energy conservative condition whose definition will be given in Theorem \ref{thm_ec}.
The uniqueness of energy weak solution satisfying the energy conservative condition, i.e. the uniqueness of conservative solution, was recently proved by the authors in \cite{CCDS}. We will review the existence and uniqueness results in the next section.

In this paper, we consider the stability of conservative solution. Due to finite time energy concentration, occurring when gradient blowup forms, the solution flow is not Lipschitz in the energy space, i.e. $H^1$ space. We will construct a Finsler-type distance that renders the conservative solution flows of \eqref{vwl}--\eqref{ID} Lipschitz continuous. The new distance will be determined by the minimum cost to transport  from one solution to the other. We consider a double optimal transportation problem which can equip the metric with information on the quasilinear  structure of the wave equation. More precisely, we consider the propagation of forward and backward waves, respectively, and their interactions. To control the energy transfer between two directions, we add a wave potential capturing future wave interactions in the metric.

The main result of this paper is

\begin{Theorem}\label{thm_metric}
We consider the unique conservative solution given in Theorems \eqref{thm_ex} and \eqref{thm_ec} for \eqref{vwl}-\eqref{ID}.
 Let the conditions \eqref{con}--\eqref{gencon}  be satisfied,
 then the geodesic distance $d$, defined in Definition \ref{def_weak}, provides solution flow the following Lipschitz continuous property. Consider two
initial data $(u_0,u_1)(x)$ and $(\hat{u}_0,\hat{u}_1)(x)$ in \eqref{ID}, then for any $T>0$, the corresponding solutions $u(x,t)$ and $\hat{u}(x,t)$ satisfy
\begin{equation*}
d\Big((u,u_t)(t),(\hat{u},\hat{u}_t)(t)\Big)\leq C d\Big((u_0,u_1),(\hat{u}_0,\hat{u}_1)\Big),
\end{equation*}
when $t\in[0,T]$, where the constant $C$ depends only on T and the total energy.
\end{Theorem}

This result together with the existence and uniqueness results give a fairly complete picture for the global well-posedness of H\"older continuous energy conservative solution to \eqref{vwl}--\eqref{ID}.

One crucial obstruction in establishing the new distance is how to prove the existence of regular enough transportation planes between two solutions. We overcome this issue by proving a generic regularity result.
\begin{Theorem}[Generic regularity]\label{thm_reg}
 Let the condition \eqref{con}--\eqref{gencon} be satisfied and let $T>0$ be given, then there exists an open dense set of initial data
\[\mathcal{M}\subset \Big(\mathcal{C}^3(\mathbb{R})\cap H^1(\mathbb{R})\Big)\times \Big(\mathcal{C}^2(\mathbb{R})\cap L^2(\mathbb{R})\Big),\]
such that, for $(u_{0}, u_{1})\in \mathcal{M}$, the conservative solution $u=u(x,t)$ of \eqref{vwl}--\eqref{ID} is twice continuously differentiable in the complement of finitely many characteristic curves, within the domain $ \mathbb{R}\times[0,T].$
\end{Theorem}
Roughly speaking, we prove that, for generic smooth initial data $({u}_0,{u}_1)$, the corresponding solution is piecewise smooth in the $x$--$t$ plane, with  singularities occurring along a finite set of smooth curves. We also show the existence of regular enough paths between generic solutions, so the transport metric can be well defined on generic solutions. The generic regularity result itself is a very interesting result. It clarifies generic properties of solutions and singularities.
\bigskip

The idea of using double optimal transport metric to prove Lipschitz continuous dependence of energy conservative weak solution, including cusp singularity, was first used for the variational wave equation \eqref{vwe} by \cite{BC2015}. In this paper, we consider a much more general and complicated system, so there are many variations in the construction of the new Lipschitz metric and proofs of Lipschitz property, comparing to the one for variational wave equation.  In fact, since the metric needs to be designed according to the equation itself, for \eqref{vwl}, we need to adjust many terms used in \cite{BC2015} for \eqref{vwe} and also add some new terms, especially on the subtle relative shift terms. A slight change in the metric may ruin the Lipschitz property.

The generic regularity result has been established for the variational wave equations in \cite{BC,CCD}. In this paper, based on the method used in \cite{BC}, we prove Theorem \ref{thm_reg} for the more general system \eqref{vwl}. The calculations on both Lipschitz metric and generic regularity for this general model are considerably more complicated than those for the variational wave equation. The results in this paper are for large solutions, without restriction on the size of solutions.

This paper will be divided into  seven sections. Section 2 is a short review on the existence and uniqueness of conservative solution to \eqref{vwl}--\eqref{ID}. In Section 3,  we will introduce main ideas used in this paper, and also the structure of Sections 4 to 7, in which we construct the metric and prove main theorems.


\section{Previous existence and uniqueness results}\label{sec_rew}
\setcounter{equation}{0}
We begin, in this paper, by reviewing the existence and uniqueness of conservative weak solution to the Cauchy problem \eqref{vwl}--\eqref{ID} in \cite{CCDS,H}.

\begin{Theorem} [Existence \cite{H}] \label{thm_ex}
Let the condition \eqref{con} be satisfied, then the Cauchy problem \eqref{vwl}--\eqref{ID} admits a global weak solution $u=u(x,t)$ defined for all $(x,t)\in\mathbb{R}\times\mathbb{R}^+$.

\end{Theorem}

To introduce the uniqueness result, let's first introduce some notations.
Denote wave speeds as
\[c_1:=\alpha\lambda_-=\frac{\beta-\sqrt{\beta^2+\alpha^2\gamma^2} }{\alpha}<0,\qquad c_2:=\alpha\lambda_+=\frac{\beta+\sqrt{\beta^2+\alpha^2\gamma^2} }{\alpha}>0,\]
and Riemann variables as
\begin{equation}\label{R-S}
 R:=\alpha u_t+c_2 u_x,\quad S:=\alpha u_t+c_1 u_x.
 \end{equation}
By \eqref{con}, the wave speeds $-c_1$ and $c_2$ are smooth, bounded and uniformly positive.

For a smooth solution of  (\ref{vwl}), the variables $R$ and $S$ satisfy
\begin{equation} \label{R-S-eqn}
\begin{cases}
\alpha(x,u) R_t+c_1(x,u)R_x=a_1 R^2-(a_1+a_2)RS+a_2 S^2+c_2bS-d_1 R,\\
\alpha(x,u)  S_t+c_2(x,u)S_x=-a_1 R^2+(a_1+a_2)RS-a_2 S^2+c_1bR-d_2 S,\\
\displaystyle u_t=\frac{c_2S-c_1R}{\alpha(c_2-c_1)} \quad{\text or} \quad u_x=\frac{R-S}{c_2-c_1},
 \end{cases}
 \end{equation}
 where
\begin{equation*}
\begin{split}
&a_i=\frac{c_i\partial_u\alpha-\alpha\partial_u c_i}{2\alpha(c_2-c_1)},\quad b=\frac{\alpha\partial_x(c_1-c_2)+(c_1-c_2)\partial_x\alpha}{2\alpha(c_2-c_1)},\\
&d_i=\frac{c_2\partial_x c_1-c_1\partial_x c_2}{2(c_2-c_1)}+\frac{\alpha\partial_x c_i-c_i\partial_x\alpha}{2\alpha},\quad(i=1,2).
\end{split}
\end{equation*}
Here $\partial_x$ and $\partial_u$ denote partial derivatives with respect to $x$ and $u$, respectively.

Multiplying the first equation in \eqref{R-S-eqn} by $2R$ and the second one by $2S$, one has the balance laws for energy densities in two directions, namely
\begin{equation}\label{balance1}
\begin{cases}
\displaystyle (R^2)_t+(\frac{c_1}{\alpha}R^2)_x=\frac{2a_2}{\alpha}(RS^2
-R^2S)+\frac{2c_2b}{\alpha}RS-\frac{c_2\partial_x c_1-c_1\partial_x c_2}{\alpha(c_2-c_1)}R^2,\\
\displaystyle (S^2)_t+(\frac{c_2}{\alpha}S^2)_x=\frac{2a_1}{\alpha}(RS^2
-R^2S)-\frac{2c_1b}{\alpha}RS-\frac{c_2\partial_x c_1-c_1\partial_x c_2}{\alpha(c_2-c_1)}S^2.
\end{cases}
\end{equation}
Moreover, we have
\begin{equation}\label{balance}
\begin{cases}
(\Tilde{R}^2)_t+(\frac{c_1}{\alpha}\Tilde{R}^2)_x=G,\\
(\Tilde{S}^2)_t+(\frac{c_2}{\alpha}\Tilde{S}^2)_x=-G,
\end{cases}
\end{equation}
where \begin{equation*}
\begin{split}
&\Tilde{R}^2=\frac{-c_1}{c_2-c_1}R^2,\quad \Tilde{S}^2=\frac{c_2}{c_2-c_1}S^2,\quad{\rm and}\\
&G=\frac{2c_2a_1}{\alpha(c_2-c_1)}R^2 S-\frac{2c_1a_2}{\alpha(c_2-c_1)}RS^2-\frac{2c_1c_2b}{\alpha(c_2-c_1)}RS,
\end{split}\end{equation*}
which indicates the following conserved quantities
\begin{equation*}
\alpha^2 u_t^2+\gamma^2u_x^2=\Tilde{R}^2+\Tilde{S}^2,
\end{equation*}
and the corresponding energy conservation law
\[
(\Tilde{R}^2+ \Tilde{S}^2)_t+(\frac{c_1}{\alpha}\Tilde{R}^2+\frac{c_2}{\alpha}\Tilde{S}^2)_x=0.
\]

Now, we state the uniqueness result in \cite{CCDS}, which together  with the energy conservation proved in \cite{H} show that the problem (\ref{vwl})--(\ref{ID}) has a unique weak solution which conserves the total energy.

\begin{Theorem}[Uniqueness \cite{CCDS} and energy conservation \cite{H}]\label{thm_ec}
 Let the condition \eqref{con} be satisfied, then there exists a unique  conservative weak solution $u(x,t)$ for \eqref{vwl}--\eqref{ID}.

Here a weak solution  $u(x,t)$ defined in Definition \ref{weakdef} is said to be {\em  (energy) conservative} if one can find two  families of positive Radon measures
on the real line: $\{\mu_-^t\}$ and $\{\mu_+^t\}$, depending continuously
on $t$ in the weak topology of measures, with the following properties.
\begi
\item[(i)] At every time $t$ one has
\begin{equation*}
\mu_-^t(\mathbb {R})+\mu_+^t(\mathbb {R})~=~\mathcal{E}_0~:=~
\int_{-\infty}^\infty \Big[\alpha^2\big(x,u_0(x)\big) u_1^2(x)+ \gamma^2\big(x,u_0(x)\big) u_{0,x}^2(x)) \Big]\, dx \,.\end{equation*}

\item[(ii)] For each $t$, the absolutely continuous parts of $\mu_-^t$ and
$\mu_+^t$ with respect to the Lebesgue measure
have densities  respectively given  by
\begin{equation*}
\Tilde{R}^2=\frac{-c_1}{c_2-c_1}(\alpha u_t+c_2 u_x)^2,\qquad
\Tilde{S}^2=\frac{c_2}{c_2-c_1}(\alpha u_t+c_1 u_x)^2.
\end{equation*}
\item[(iii)]  For almost every $t\in\mathbb {R}^+$, the singular parts of $\mu^t_-$ and $\mu^t_+$
are concentrated on the set where $\partial_u \lambda_-=0$ or $\partial_u \lambda_+=0$.

\item[(iv)] The measures $\mu_-^t$ and $\mu_+^t$ provide measure-valued solutions
respectively to the balance laws
\bel{mbl}
\begin{cases}
\displaystyle \xi_t + (\frac{c_1}{\alpha}\xi)_x  = \frac{2c_2a_1}{\alpha(c_2-c_1)}R^2 S-\frac{2c_1a_2}{\alpha(c_2-c_1)}RS^2-\frac{2c_1c_2b}{\alpha(c_2-c_1)}RS, \\
\displaystyle \eta_t + (\frac{c_2}{\alpha}\eta)_x  =  - \frac{2c_2a_1}{\alpha(c_2-c_1)}R^2 S+\frac{2c_1a_2}{\alpha(c_2-c_1)}RS^2+\frac{2c_1c_2b}{\alpha(c_2-c_1)}RS.
\end{cases}
\eeq
\endi
\end{Theorem}

Furthermore, for above conservative weak solution, the total energy represented by the sum $\mu_-+\mu_+$ is showed to be
conserved in time. This energy may only be concentrated
on a set of zero measure or at points where $\partial_u\lambda_-$ or $\partial_u\lambda+$ vanishes. In particular, if $\partial_u \lambda_\pm\not= 0$
for any $(x,u)$, then the set
$$\Big\{\tau;~\mathcal{E}(\tau)~:=~\int_{-\infty}^\infty\Big[|
\alpha^2\big(x,u(x,\tau)\big) u_t^2(x,\tau)+
\gamma^2\big(x,u(x,\tau)\big) {u}_x^2(x,\tau)
\Big]\, dx  ~<~\mathcal{E}_0\Big\}$$
has measure zero.

Finally, we remark that the Lipschitz continuous dependence result in this paper does not direct to a uniqueness result for the original equation \eqref{vwl}, because in the current paper we only consider the solution constructed in \cite{H}. The uniqueness result in \cite{CCDS} rules our the possibility to construct a different solution using a method different from \cite{H}. It also makes the  Lipschitz continuous dependence result obtained in this paper working for all solutions, due to the uniqueness.
\section{Main ideas}
In this section, we introduce why we need a double optimal transport metric for the Lipschitz continuous dependence of conservative weak solutions of \eqref{vwl}, and how we construct it.

\subsection{Why consider a double optimal transport problem?}
The finite time energy concentration at the gradient blowup, makes the stability problem of \eqref{vwl} a very challenging one. As a direct consequence, solution flow is not Lipschitz in the energy space, i.e. $H^1$ space.
For examples showing this instability, one see \cite{CCCS} on some unitary direction wave models, or \cite{CS,GHZ} on the variational wave equation.

So it is very natural to use an optimal transport metric. An existing metric is the Kantorovich-Rubinstein metric
\[
d^*(u,\hat {u})=\sup_{\|f\|_{{\footnotesize\hbox{Lip}}}\leq 1}\,\bigg|\int f\,d\mu - \int f d\hat{\mu}\bigg|
,\]
where $\mu,\hat{\mu}$  are the
measures with densities
$\frac{1}{2}({\Tilde R}^2+{\Tilde S}^2)(x,u)$ and
$\frac{1}{2}({\Tilde R}^2+{\Tilde S}^2)(x,\hat {u})$
w.r.t.~Lebesgue measure, corresponding to two solutions $u$ and $\hat u$. Here $\|\cdot\|_{{\footnotesize\hbox{Lip}}}$ means the Lipschitz norm. This metric is equivalent to a  $1$-Wasserstein metric by a duel theorem \cite{V}.

Unfortunately, this norm also fails to work. In fact, to capture the quasilinear structure of solutions, we need to consider a double transportation problem, which means that we shall study the propagation of waves in backward and forward directions, respectively, and their interactions. This directs to our new metric.

However, to consider a double transport problem, the metric will change dramatically. For example: first, the metric is a Finsler type metric; secondly, we need to control the growth of norm caused by the energy transfer between two families during nonlinear wave interactions.

As a payback, now we can use both balance laws on forward and backward energy densities along characteristics, such as \eqref{balance1} and \eqref{balance}, to analyze the propagation of waves and the impact of wave interactions. 
Now we give more details of our metric.

\subsection{How to construct the Finsler type Lipschitz metric?}
We construct the metric $d$ and prove that the solution map is Lipschitz continuous under this metric in several steps.

\begin{itemize}
\item[1.] We construct a Lipschitz metric for smooth solution (Sections \ref{sec_smooth}).
\item[2.] We prove by an application of Thom's transversality theorem that the piecewise smooth solutions with only generic singularities are dense in $H^1\times L^2$ space. This proves Theorem \ref{thm_reg}  (Section \ref{gen_sec}).
\item[3.] We extend the Lipschitz metric to piecewise smooth solutions with generic singularities. (Section \ref{sec_piecewise}).
\item[4.] We finally define the metric $d$ for weak solutions with general $H^1\times L^2$ initial data, and complete the proof of Theorems \ref{thm_metric}. We also compare the metric $d$ with some Sobolev metrics and Kantorovich-Rubinstein and Wasserstein metrics (Section \ref{sec_weak}).
\end{itemize}

A more detailed introduction is given below.

\medskip

\paragraph{\bf Step 1: Metric for smooth solutions}
To keep track of the cost in the transportation, we are led to construct the geodesic distance. That is, for two given solution profiles $u(t)$ and $u^\epsilon(t)$, we consider all possible smooth deformations/paths $\gamma^t: \theta \mapsto \big(u^\theta(t),u_t^\theta(t)\big)$ for $\theta\in [0,1]$ with $\gamma^t(0) =\big( u(t),u_t(t)\big)$ and $\gamma^t(1) =\big( u^\epsilon(t), u_t^\epsilon(t)\big)$, and then measure the length of these paths through integrating the norm of the tangent vector $d\gamma^t/d\theta$; see Figure \ref{homotopy figure} (a).

\begin{figure}[h]
  \includegraphics[page=5, scale=0.8]{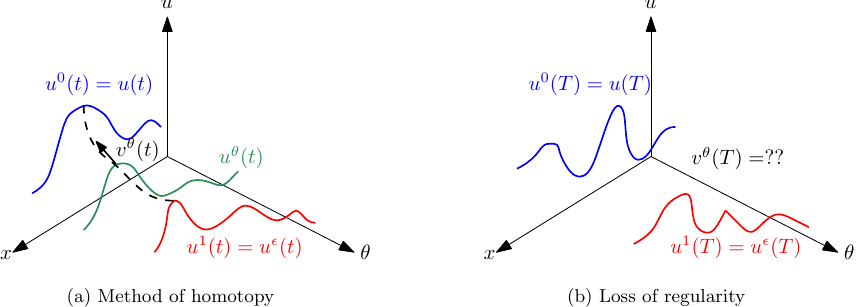}
  \caption{Compare two solutions $u(x)$ and $u^\epsilon(x)$ at a given time $t$.}
  \label{homotopy figure}
\end{figure}

Roughly speaking, the distance between $u$ and $u^\epsilon$ will be calculated by the optimal path length
\begin{equation*}
d\LC u(t), u^\epsilon(t) \RC = \inf_{\gamma^t}\|\gamma^t\| := \inf_{\gamma^t}\int^1_0 \| v^\theta(t) \|_{u^\theta(t)} \ d\theta, \quad \text{where } v^\theta(t) = {d\gamma^t\over d\theta}.
\end{equation*}
The subscript $u^\theta(t)$ emphasizes the dependence of the norm on the flow $u$.
The most important element is how to define the Finsler norm $\| v^\theta(t) \|_{u^\theta(t)}$ by capturing behaviors of the quasilinear wave equation, such that, for regular solutions,
\begin{equation}\label{rough est}
\|\gamma^t\| \leq C \|\gamma^0\| \qquad \forall\ t \in [0, T].
\end{equation}
Here $C$ only depends on the total initial energy and $T$, and is uniformly bounded when the solution approaches a singularity.

More precisely, we define the Finsler norm, by measuring the cost in shifting from one solution to the other one.
We will measure the cost in forward and backward directions, with energy densities $\mu^+,\mu^-$ defined in Theorem \ref{thm_ec}, respectively.

\begin{figure}[h]
  \includegraphics[scale=0.5]{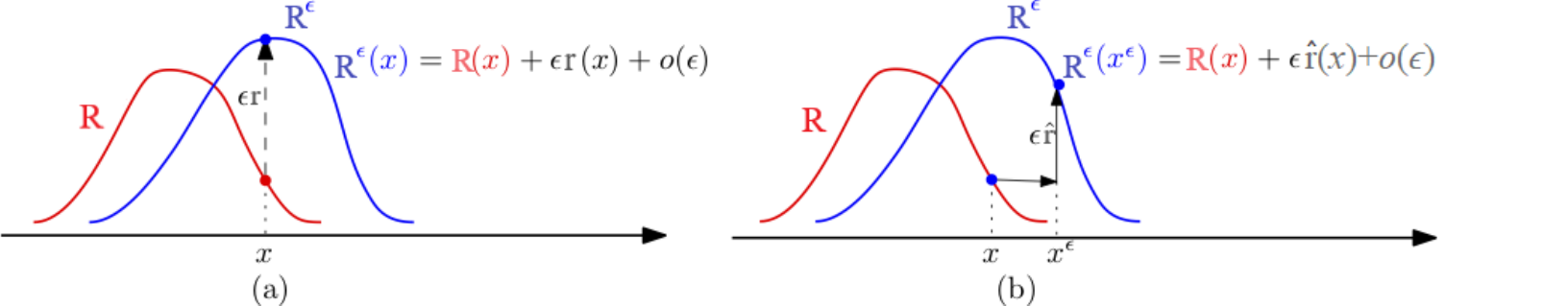}
  \caption{  \label{2storder figure}A sketch of how to deform from $R$ to $R^\epsilon$: (a) a vertical shift $\epsilon r$; (b) a horizontal shift $\epsilon R_x w$ followed by a vertical displacement $\epsilon r$. Here $x^\epsilon:=x+\epsilon w(x)$. We denote the total shift as $\epsilon \hat r=\epsilon (r+R_x w)$. }
  \end{figure}

First, from Figure \ref{2storder figure} (b), we illustrate how to measure the cost of transporting $R$ to $R^\epsilon$ on the $x$-$u$ plane by
\beq\label{change}
[\hbox{change in } R] = o(\epsilon) \hbox{ order of } \Big(R^\epsilon(x^\epsilon) - R(x)\Big) =\hat r :=  r(x) + R_x(x)\, w(x),
\eeq
at any time $t$. Recall $R$ is defined in \eqref{R-S} for backward waves. Here $r$ and $R_x w$ measures the vertical and horizontal changes in the leading order, respectively. Similarly, we can find variations of other functions, like $u$, and the cost of transportation in the forward direction.

Now the basic structure of the Finsler norm, or the cost function for double transportation, can be  very roughly introduced as
\begin{equation}\label{1.6}
\begin{split}
&\int_{-\infty}^{\infty}\Big\{[\text{change\ in\ }x]+[\text{change\ in\ }u]+
[\text{change\ in\ }\arctan R ]\Big\}(1+R^2)
dx\\[1mm]
+& \int_{-\infty}^{\infty}[\text{change\ of\ base\  measure}
\text{\ with\ density}\ R^2]\,dx+\hbox{terms in forward direction}
\end{split}
\end{equation}
Here we use $R^2$, $S^2$ defined in \eqref{R-S} instead of $\Tilde R^2,\ \Tilde S^2$ for  $\mu^+,\mu^-$, for convenience. The forward and backward energies might transfer to each other during wave interactions, so we need to add the change of base measure (of energy) in each direction.

The structure of cost function shown in \eqref{1.6} only gives very rough idea of the real norm.
Two other major groups of terms need to be added,  such that the metric can satisfy the desired uniform Lipschitz property.

\paragraph{\em (i). Interaction potentials: $\mathcal W^-,\ \mathcal W^+$}
By \eqref{balance1} or \eqref{balance}, the forward or backward energy might increase in the wave interaction, although the total energy is conserved. This causes a major challenge. A pair of interaction potentials $\mathcal W^-$ and  $\mathcal W^+$ will be added. These potentials produce time decay which controls the energy increase, in backward or  forward direction, respectively. We also need to very carefully weight each term in the cost function by a different weight.

\paragraph{\em (ii). Relative shifts}
From Figure \ref{Pic_drawing4}, we see that, when a backward wave shifts horizontally by $w$ and a forward wave shifts horizontally by $z$, then these two waves are relatively shifted by $w-z$. And new wave interactions happen due to the relative shift.
So the cost function needs to be adjusted accordingly in order to taking account of the relative shift.

Roughly speaking, the  relative shift terms in the cost function can be evaluated by the source terms of balance laws such as \eqref{balance1}. While there is in fact no routine way to find the relative shift terms. One needs to adjust each term very carefully, corresponding to the balance laws, to obtain the desired Lipschitz continuous property. A slight change in the delicate metric might ruin the Lipschitz property. This part is the most subtle and difficult part in the construction of the metric.

\begin{figure}[htbp]
   \centering
\includegraphics[width=0.35\textwidth]{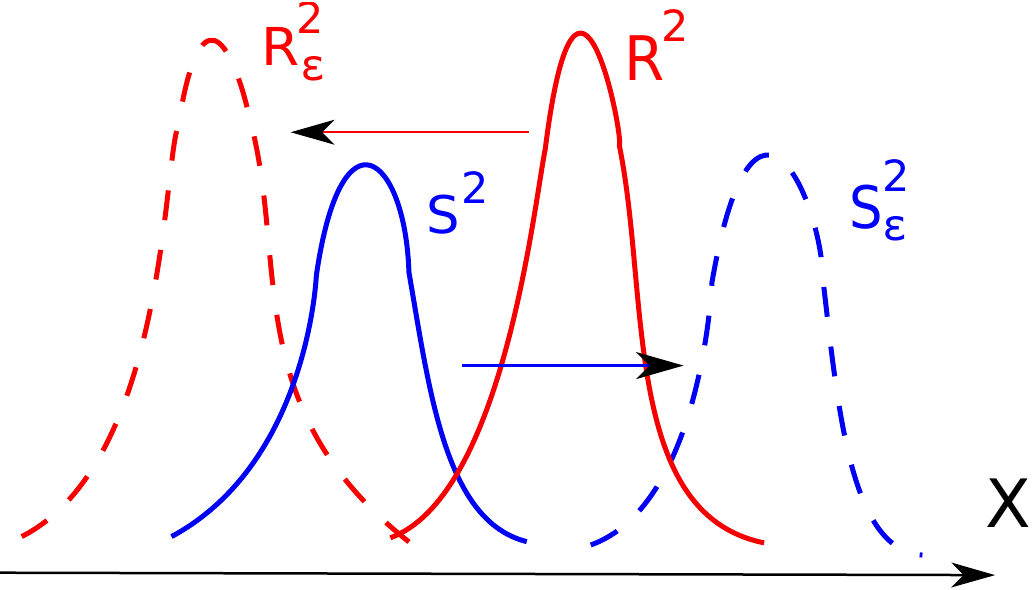}
   \caption{Waves propagate in two directions $\frac{dx}{dt}=\lambda_\pm$. Transportation will cause relative shifts between forward and backward waves, and new wave interactions. This needs to be counted in the cost function.
   }\label{Pic_drawing4}
\end{figure}



\vspace{.2cm}

\paragraph{\bf Step 2: Generic regularity and metric for generic solutions}
However, smooth solutions do not always remain smooth for all time. Therefore a smooth path of initial data $\gamma^0$ may lose regularity at a later time $T$ so that the tangent vector $d\gamma^T/dt$ may not be well-defined (see Figure \ref{homotopy figure} (b)); even if it does exist, it is not obvious that the estimate \eqref{rough est} should remain valid after a singularity.

To overcome the first problem, the idea is to show that there are sufficiently many paths which are initially regular (piecewise smooth), and remain regular later in time. So the norm is well-defined for almost all solutions, named as generic solutions. The generic regularity in Theorem \ref{thm_reg} is itself of great interest because it shows very detailed structures of singularities, and the generic properties of solutions.
The proof relies on an application of Thom's Transversality Theorem.

Then we extend the metric to generic (piecewise smooth) solutions. We consider the semilinear system under characteristic coordinates used in the existence result, and show the metric is continuous in time. So the Lipschitz property still holds.

\vspace{.2cm}
\paragraph{\bf Step 3: Metric for general weak solutions}
Finally, we extend the metric to general weak solutions, by taking limit from generic solutions. This completes the proof of Theorem \ref{thm_metric}. Moreover, we will compare the new metric with other metrics.

\begin{Remark}
There are a lot of works on the optimal transport Lipschitz metrics for the unitary direction wave models, such as Camassa-Holm and Hunter-Saxton equations \cite{BF, BHR,CCCS,CGH, GHR}. These equations can be written as scalar first order equations with nonlocal source terms. So one does not have to consider the double transport problem, and interactions of waves from different families. As a consequence, the construction of metric for these models is much less complex than for the wave equation model.
\end{Remark}

%

\section{The norm of tangent vectors for smooth solutions}\label{sec_smooth}

Our first goal is to define a Finsler norm on tangent vectors measuring the cost of transport. Then by elaborate estimates, we show this norm satisfies the desired Lipschitz property \eqref{rough est} for any smooth solutions.


Let $(u,R,S)$ be any smooth solution to \eqref{vwl}, \eqref{R-S-eqn}, and then take a family of perturbed solutions $(u^\epsilon,R^\epsilon,S^\epsilon)(x)$ to \eqref{vwl}, \eqref{R-S-eqn}, which can be written as
\begin{equation}\label{perb}
u^\epsilon(x)=u(x)+\epsilon v(x)+o(\epsilon), \quad{\rm and }\quad
\begin{cases}
R^\epsilon(x)=R(x)+\epsilon r(x)+o(\epsilon),\\
S^\epsilon(x)=S(x)+\epsilon s(x)+o(\epsilon).
\end{cases}
\end{equation}
Here both $u$ and $u^\epsilon$ satisfy the Definition  \ref{weakdef} of weak solution. Because of finite speed of propagation, for any time $T>0$, there exists a compact subset on the $x$-$t$ plane with $t\in[0,T]$, out of which the solution  is smooth.

Here and in the sequel, we will omit the variables $t, x$ when we use any functions if it does not cause any confusion.

Let the tangent vectors $r,s$ be given, in terms of \eqref{R-S-eqn} (equation on $u_x$) and \eqref{perb}, the perturbation $v$ can be uniquely determined by
\begin{equation}\label{vx}
v_x=\frac{r-s}{c_2-c_1}-\frac{\partial_u c_2-\partial_u c_1}{(c_2-c_1)^2}(R-S)v,\qquad v(0,t)=0.
\end{equation}
Moreover, it holds that
\begin{equation}\label{vt}
v_t=\frac{c_2s-c_1r}{\alpha(c_2-c_1)}-
\frac{c_2S-c_1R}{\alpha^2(c_2-c_1)}v\partial_u \alpha+ \frac{c_2\partial_u c_1-c_1\partial_u c_2}{\alpha(c_2-c_1)^2}(S-R)v.
\end{equation}
Furthermore, by a straightforward calculation, the first order perturbations $v,s,r$ must satisfy the equations
\begin{equation}\label{vtt}
\begin{split}
&\alpha^2 v_{tt}-\gamma^2 v_{xx}+2\beta v_{xt}=[2\gamma\gamma_{u}u_x+2\gamma\gamma_{x}-\beta_u u_t]v_x-[2\alpha\alpha_u u_t+\beta_u u_x+\beta_x]v_t\\
&\qquad\qquad-[\alpha^2_u u_t^2+\alpha_{uu}\alpha u_t^2+2\alpha\alpha_u u_{tt}+2\beta_u u_{xt}+\beta_{xu}u_t+\beta_{uu}u_tu_x]v\\
&\qquad\qquad+[\gamma_u^2 u_x^2+\gamma\gamma_{uu}u_x^2+2\gamma\gamma_uu_{xx}
+2\gamma_u\gamma_xu_x+2\gamma\gamma_{xu}u_x]v,
\end{split}\end{equation}
and
\begin{equation}\label{rt}
\begin{cases}
\alpha r_t+c_1r_x=2a_1Rr-(a_1+a_1)(Rs+Sr)+2a_2Ss+c_2bs-d_1r
+2a_1(c_2-c_1)R_xv\\
\qquad\qquad\qquad+\alpha B_1R^2v-\alpha(B_1+B_2)RSv+\alpha B_2 S^2v+\alpha B_4 Sv-\alpha B_5 Rv,\\
\alpha s_t+c_2s_x=-2a_1Rr+(a_1+a_1)(Rs+Sr)-2a_2Ss+c_1br-d_2s
+2a_2(c_2-c_1)S_xv\\
\qquad\qquad\qquad-\alpha B_1R^2v+\alpha(B_1+B_2)RSv-\alpha B_2 S^2v-\alpha B_6 Sv+\alpha B_3 Rv,\\
\end{cases}\end{equation}
where $B_i=\frac{\alpha \partial_u a_i-a_i \partial_u \alpha}{\alpha^2},i=1,2$, $B_3=\frac{\alpha \partial_u (c_1 b)-c_1b \partial_u \alpha}{\alpha^2},$ $B_4=\frac{\alpha \partial_u (c_2 b)-c_2b \partial_u \alpha}{\alpha^2},$
$B_5=\frac{\alpha \partial_u d_1-d_1\partial_u \alpha}{\alpha^2},$ $B_6=\frac{\alpha \partial_u d_2-d_2\partial_u \alpha}{\alpha^2}.$

To measure the cost in shifting from one solution to the other one, it is nature to consider both vertical and horizontal shifts in the energy space. With this in mind, since the tangent flows $v,r,s$ only measure the vertical shifts between two solutions, we also need to add quantities $w(x,t), z(x,t)$ measuring the horizontal shifts, corresponding to backward and forward directions, respectively.
And it is very tentative to embed some important information of waves into $w(x,t), z(x,t)$ in order to focus only on reasonable transports between two solutions.
Here we require $w(x,t)$ to  satisfy
\[
\epsilon w(x,t)=x^\epsilon(t)-x(t),
\]
where $x^\epsilon(t)$ and $x(t)$ are two backward characteristics starting from initial points $x^\epsilon(0)$ and $x(0)$.
Symmetrically, the function $\epsilon z(x,t)$ measures the difference of two forward characteristics.
Then it is easy to see that $w, z$ satisfy the following system
\begin{equation}\label{wz}
\begin{cases}
\displaystyle \alpha w_t+c_1w_x=\frac{\alpha\partial_x c_1-c_1\partial_x\alpha}{\alpha}w-2a_1(c_2-c_1)(v+u_x w),\\
\displaystyle \alpha z_t+c_2z_x=\frac{\alpha\partial_x c_2-c_2\partial_x\alpha}{\alpha}z-2a_2(c_2-c_1)(v+u_x z),\\
w(x,0)=w_0(x),\qquad z(x,0)=z_0(x).
\end{cases}\end{equation}

Next, we define {\em  interaction potentials} $\mathcal{W}^+/\mathcal{W}^-$ for forward/backward directions as follows.
\begin{equation}\label{W}
\mathcal{W}^-:=1+\int_{-\infty}^x S^2(y)\,dy,\quad
\mathcal{W}^+:=1+\int^{+\infty}_x R^2(y)\,dy.
\end{equation}
Essentially, when tracking a backward wave, $\mathcal{W}^-$ measures the total forward energy, that this backward wave will meet in the future. This interaction potential
keeps decaying as $t$ increases. We will use this decay to balance the possible increase of backward energy during wave interactions. The forward potential $\mathcal{W}^+$ can be explained similarly. To understand these potentials, one can compare them with Glimm potential for hyperbolic conservation laws, while we use the potentials for large data solutions.

In view of \eqref{balance1}, it holds that
\begin{equation*}
\begin{cases}
\displaystyle\mathcal{W}^-_t+\frac{c_1}{\alpha}\mathcal{W}^-_x=
-\frac{c_2-c_1}{\alpha}S^2+\int_{-\infty}^x \big[\frac{2a_1}{\alpha}(RS^2-R^2S)-\frac{2c_1 b}{\alpha}RS-\frac{c_2\partial_x c_1-c_1\partial_xc_2}{\alpha(c_2-c_1)}S^2\big]\,dy,\\
\displaystyle\mathcal{W}^+_t+\frac{c_2}{\alpha}\mathcal{W}^+_x=
-\frac{c_2-c_1}{\alpha}R^2+\int^{+\infty}_x \big[\frac{2a_2}{\alpha}(RS^2-R^2S)+\frac{2c_2 b}{\alpha}RS-\frac{c_2\partial_x c_1-c_1\partial_xc_2}{\alpha(c_2-c_1)}R^2\big]\,dy.\\
\end{cases}
\end{equation*}
This together with condition \eqref{con} implies that
\begin{equation}\label{Westimate}
\begin{cases}
\displaystyle\mathcal{W}^-_t+\frac{c_1}{\alpha}\mathcal{W}^-_x\leq
-\frac{2\gamma_1}{\alpha_2}S^2+G_1(t),\\
\displaystyle\mathcal{W}^+_t+\frac{c_2}{\alpha}\mathcal{W}^+_x\leq
-\frac{2\gamma_1}{\alpha_2}R^2+G_2(t),\\
\end{cases}
\end{equation}
where the functions
$$G_1(t):=\int_{-\infty}^{+\infty} \Big|\frac{2a_1}{\alpha}(RS^2-R^2S)+\frac{2c_1 b}{\alpha}RS-\frac{c_2\partial_x c_1-c_1\partial_xc_2}{\alpha(c_2-c_1)}S^2\Big|\,dy,$$
$$G_2(t):=\int^{+\infty}_{-\infty} \Big|\frac{2a_2}{\alpha}(RS^2-R^2S)+\frac{2c_2 b}{\alpha}RS-\frac{c_2\partial_x c_1-c_1\partial_xc_2}{\alpha(c_2-c_1)}R^2\Big|\,dy.$$
As proved in \cite{CCDS} (see equation (3.16) in \cite{CCDS}), we obtain
\begin{equation}\label{4.9}
\int_0^T G_i(t)\leq C_T,\quad i=1,2,
\end{equation}
for some constant $C_T$ depending only on $T$ and the total energy.

Up to now, we are ready to define a Finsler norm for the tangent vectors $v,r,s$ as
\begin{equation}\label{Finsler v}
\|(v,r,s)\|_{(u,R,S)}: = \inf_{v, w, \hat {r},z, \hat{s}} \|(v, w, \hat r,z, \hat{s})\|_{(u,R,S)},
\end{equation}
where the infimum is taken over the set of vertical displacements $v, \hat{r}, \hat{s}$ and horizontal shifts $w,z$ which satisfy equations \eqref{vx}, \eqref{vt}, \eqref{wz} and relations
\begin{equation}\label{rseq}
\begin{cases}
\displaystyle\hat{r}=r+wR_x+\frac{a_2(w-z)}{c_2-c_1}S^2
-\frac{a_1+a_2}{c_2-c_1}(w-z)RS+\frac{c_2b(w-z)}{c_2-c_1}S,\\
\displaystyle\hat{s}=s+zS_x-\frac{a_1(w-z)}{c_2-c_1}R^2
+\frac{a_1+a_2}{c_2-c_1}(w-z)RS+\frac{c_1b(w-z)}{c_2-c_1}R.
\end{cases}
\end{equation}
Here, to motivate the explicit construction of $\|(v,r,s)\|_{(u,R,S)}$, we consider a reference solution $R$ together with a perturbation $R^\epsilon$. As shown in Figure \ref{2storder figure}, the tangent vector $r$ can be expressed as a horizontal part $\epsilon w$ and a vertical part $\epsilon\hat{r}$, that is $r=\hat{r}-R_x w$. 
The other terms in \eqref{rseq} take account of relative shift. We will give more details on the relative shift terms later.
Now, we define the following norm.
\begin{equation}\label{norm1}
\begin{split}
&\ \|(v, w, \hat r,z, \hat{s})\|_{(u,R,S)}\\
&:=~\kappa_0\int_\mathbb{R} \big[|w|\, \mathcal{W}^-+|z|\, \mathcal{W}^+\big]\,dx+\kappa_1\int_\mathbb{R}  \big[|w|(1+R^2)\, \mathcal{W}^-
+|z|(1+S^2)\, \mathcal{W}^+\big]\,dx\\
&\quad+\kappa_2\int_\mathbb{R}  \Big|v+\frac{Rw-Sz}{c_2-c_1}\Big|\big[(1+R^2)\, \mathcal{W}^-
+(1+S^2)\, \mathcal{W}^+\big]\,dx+\kappa_3\int_\mathbb{R}  \big[|\hat{r}|\, \mathcal{W}^-
+|\hat{s}|\, \mathcal{W}^+\big]\,dx \\
&\quad+\kappa_4\int_\mathbb{R} \Big[ \Big|w_x+\frac{2a_1(w-z)}{c_2-c_1}S\Big|\, \mathcal{W}^-+
\Big|z_x-\frac{2a_2(w-z)}{c_2-c_1}R\Big|\, \mathcal{W}^+\Big]\,dx \\
&\quad+\kappa_5\int_\mathbb{R} \Big[ \Big|Rw_x+\frac{2a_1(w-z)}{c_2-c_1}RS\Big|\, \mathcal{W}^-+
\Big|Sz_x-\frac{2a_2(w-z)}{c_2-c_1}RS\Big|\, \mathcal{W}^+\Big]\,dx \\
&\quad+\kappa_6\int_\mathbb{R} \Big[ \Big|2R\hat{r}+R^2w_x+\frac{2a_1(w-z)}{c_2-c_1}R^2S\Big|\, \mathcal{W}^-+
\Big|2S\hat{s}+S^2z_x-\frac{2a_2(w-z)}{c_2-c_1}RS^2\Big|\, \mathcal{W}^+\Big]\,dx \\
&=: \sum_{i=0}^6\kappa_i\big(\int_\mathbb{R}  J_i^-\, \mathcal{W}^-\,dx+\int_\mathbb{R}  J_i^+\, \mathcal{W}^+\,dx\big)=: \sum_{i=0}^6\kappa_i I_i,
\end{split}
\end{equation}
where $\kappa_i$ with $i=0,1,2\cdots,6$ are the constants to be determined later, and $I_i, J_i^-, J_i^+$ are the corresponding terms in the above equation.
To help readers understand the metric, roughly speaking, $J_i^-$ mean that
\[
\left\{
\begin{array}{rcl}
J_0^-&\approx& [\hbox{change in }x],\vspace{.2cm}\\
J_1^-&\approx& [\hbox{change in }x]\cdot (1+R^2),\vspace{.2cm}\\
J_2^-&\approx& [\hbox{change in } u]\cdot (1+R^2),\vspace{.2cm}\\
J_3^-&\approx&[\hbox{change in }\arctan{R}]\cdot (1+R^2),\vspace{.2cm}\\
{{J_4^-}}&\approx&{[\hbox{change of base  measure with density } 1],}\vspace{.2cm}\\
{J_5^-}&\approx&{[\hbox{change of base  measure with density } R],}\vspace{.2cm}\\
J_6^-&\approx&[\hbox{change of base  measure with density }R^2].
\end{array}\right.\]
And $J_i^+$ is symmetric for forward waves. Here we use $R^2$, $S^2$ defined in \eqref{R-S} instead of $\Tilde R^2,\ \Tilde S^2$ for  $\mu^+,\mu^-$, for convenience.

Next we give more details on how to obtain \eqref{norm1}.

\vspace{.2cm}

\paragraph{\bf [I]} For $I_1$, the integrand $|w|(1+R^2)$ accounts for the cost of transporting the base measure with density
$1+R^2$ from the point $x$ to the point $x+\epsilon w(x)$.

The integrand $|z|(1+S^2)$ accounts for the cost of transporting the base measure with density
$1+S^2$ from the point $x$ to the point $x+\epsilon z(x)$. There are no relative shift terms.

The terms in $I_0$ are corresponding to the variation of $|x|$ with base measure  with density $1$, which are added for a technical purpose.

\vspace{.2cm}
\paragraph{\bf [II]} $I_2$ can be interpreted as: [change in $u$]. Indeed, the change in $u$ can be estimated as
\begin{equation*}
\begin{split}
\frac{u^\epsilon\big(x+\epsilon w(x)\big) - u(x)}{ \epsilon} &= v(x) + u_x(x) w(x)+ o(\epsilon)\\
&= v(x) + \frac{R(x)-S(x)}{c_2-c_1}w(x)+ o(\epsilon)\\
&= v(x) + \frac{Rw-Sz}{c_2-c_1}+\frac{z-w}{c_2-c_1}S+ o(\epsilon).
\end{split}\end{equation*}
Here the last term  $\frac{z-w}{c_2-c_1}S$
on the right hand side of the above equality is just balanced with the relative shift term.

Here we use this term to introduce how to calculate the relative shift term. Recall that
\[
\alpha u_t+c_2 u_x=R,\quad \alpha u_t+c_1 u_x=S.
\]
So the difference of these two equations give
\[
u_x=\frac{1}{c_2-c_1}(R-S).
\]
Roughly speaking,
\beq\label{deltau}
\Delta u\approx\frac{\Delta x}{c_2-c_1}(R-S)=\frac{z-w}{c_2-c_1}(R-S).
\eeq
Here the $S$ term balances $\frac{z-w}{c_2-c_1}S$. We omit the $R$ term since it is a lower order term.

As we can see from \eqref{deltau} that there is only a general philosophy on how to choose terms taking account of the relative shift. One needs to adjust them very carefully to meet the demand. Comparing to the variational wave equation \eqref{vwe} in which two wave speeds have same magnitude but different signs, it is quite different and much harder in finding the relative shift terms for the general equation \eqref{vwl}.

The integrand in $J_2^{+}$ can be estimated in a similar way.

\vspace{.2cm}
\paragraph{\bf [III]} $I_3$ accounts for the vertical displacements in the graphs of $R$ and $S$. More precisely, the integrand $|\hat{r}|$ as the change in $arctan R$ times the density $1+R^2$ of the base measure. Notice that, for $x^\epsilon=x+\epsilon w(x)+o(\epsilon)$,
\begin{equation*}
\begin{split}
\arctan R^\epsilon(x^\epsilon)& = \arctan \Big( R(x^\epsilon) + \epsilon r(x^\epsilon) + o(\epsilon) \Big) \\
& = \arctan \Big( R(x) + \epsilon w(x) R_{x}(x) + \epsilon r(x) + o(\epsilon)\Big) \\
& = \arctan R(x) + \epsilon {r(x) + w(x)R_{x}(x) \over 1 + R^2(x)}+ o(\epsilon),
\end{split}
\end{equation*}
which together with the relative shift term gives $J^-_3$. Here we add some subtle adjustments in the relative shift terms to take account of interactions between forward and backward waves using \eqref{R-S-eqn}.

The change in $\arctan S$ times the density $1+S^2$ of the base measure can be explained similarly.

\vspace{.2cm}
\paragraph{\bf [IV]} $I_6$ can be interpreted as the change in the base measure with densities $R^2$ and $S^2$, produced by the shifts $w,z$. Indeed,
\begin{equation*}
\begin{split}
\big(R^\epsilon(x^\epsilon) \big)^2 & = R^2(x^\epsilon) + 2\epsilon R(x^\epsilon) r(x^\epsilon) + o(\epsilon) \\
& = R^2(x) + 2\epsilon w(x) R(x) R_{x}(x) + 2\epsilon R(x) r(x) + o(\epsilon),
\end{split}
\end{equation*}
we obtain that
\begin{equation}\label{exp1}
\begin{split}
\big( R^\epsilon_x(x^\epsilon) \big)^2 dx^\epsilon - R^2(x) dx
=\Big( 2\epsilon R(x) R_{x}(x)w(x)  + 2\epsilon R(x) r(x) +\epsilon R^2(x)w_x(x)+o(\epsilon)\Big)\,dx.
\end{split}
\end{equation}
Moreover, as in \eqref{deltau}, in view of \eqref{balance1}, if the mass with density $S^2$ is transported from $x$ to $x+\epsilon z(x)$, the relative shift between forward and backward waves will contribute
\begin{equation}\label{exp2}
[2a_2(RS^2-R^2S)+2c_2bRS]\cdot \epsilon\frac{z-w}{c_2-c_1}.
\end{equation}

Hence subtracting \eqref{exp2} from  \eqref{exp1} yields the term $J_6^-$. Symmetrically, we have $J_6^{+}$.

\vspace{.2cm}

\paragraph{\bf [V]} In order to close the time derivatives of $I_2$ and $I_6$, we have to add two additional terms $I_4$ and $I_5$. Here $I_4$ accounts for the change in the Lebesgue measure produced by the shifts $w,z$, while $I_5$ account for  the change in the base measure with densities $R$ and $S$, produced by the shifts $w,z$. These two terms are in some sense lower order terms of $I_6$.

\vspace{.2cm}

The main goal of this section is to prove  the following lemma by showing that the norm of tangent vectors defined in \eqref{Finsler v} satisfies a Gr\"{o}wnwall type inequality.
\begin{Lemma}\label{lem_est} Let $T>0$ be given,
and $(u,R,S)(x,t)$ be a smooth solution to \eqref{vwl} and \eqref{R-S-eqn} when $t\in[0,T]$. Assume that the first order perturbations $(v,r,s)$ satisfy the corresponding equations \eqref{vtt}--\eqref{rt}. Then it follows that
\begin{equation}\label{normest}
\|(v,r,s)(t)\|_{(u,R,S)(t)}\leq \mathcal{C}(T)\|(v,r,s)(0)\|_{(u,R,S)(0)},
\end{equation}
for some constant $\mathcal{C}(T)$ depending only on the initial total energy and $T$.
\end{Lemma}

\begin{proof}
To achieve \eqref{normest}, it suffices to show that
\begin{equation}\label{est on w and z}
{d \over dt}\|(v, w, \hat{r},z,\hat{s})(t) \|_{(u,R,S)(t)} \leq a(t) \|(v, w, \hat{r},z,\hat{s})(t)\|_{(u,R,S)(t)},
\end{equation}
for  any $w, z$ and $\hat{r}, \hat{s}$ satisfying (\ref{wz}) and (\ref{rseq}), with a local integrable function $a(t)$.
Here and in the sequel, unless specified, we will use $C>0$ to denote a constant depending on the initial total energy and $T$, where $C$ may vary in different estimates. Now we prove \eqref{est on w and z} by seven steps.

\bigskip
\paragraph{\bf 0} We first treat the time derivative of $I_0$. By \eqref{R-S-eqn} and \eqref{wz}, a straightforward calculation yields that
\begin{equation*}\label{wt}
\begin{split}
&\displaystyle w_t+(\frac{c_1}{\alpha}w)_x\\
&=2\frac{\alpha\partial_x c_1-c_1\partial_x\alpha}{\alpha^2}w-\frac{2a_1}{\alpha}(c_2-c_1)(v+
\frac{Rw-Sz}{c_2-c_1})+\frac{2a_1S}{\alpha}(w-z)
-\frac{2a_1w}{\alpha}(R-S).
\end{split}\end{equation*}
This together with \eqref{Westimate} implies that
\begin{equation*}
\begin{split}
&\frac{d}{dt}\int_\mathbb{R}J_0^-\mathcal{W}^-\,dx = \frac{d}{dt}\int_\mathbb{R}|w|\mathcal{W}^-\,dx \\
&\leq C\int_\mathbb{R} |w|(1+|R|+|S|)\mathcal{W}^-\,dx+C\int_\mathbb{R} |z||S|\mathcal{W}^+\,dx+ C\int_\mathbb{R}  \Big|v+\frac{Rw-Sz}{c_2-c_1}\Big|\mathcal{W}^-\,dx\\
&\quad+G_1(t)\int_\mathbb{R} |w|\mathcal{W}^-\,dx-\frac{2\gamma_1}{\alpha_2}\int_\mathbb{R} |w|S^2\mathcal{W}^-\,dx.
\end{split}
\end{equation*}
Repeating the above process for the time derivative of $\int_\mathbb{R}J_0^+\mathcal{W}^+\,dx$ yields
\begin{equation}\label{I0est}
\begin{split}
&\frac{d}{dt}I_0 =\frac{d}{dt}\int_\mathbb{R}
\Big(J_0^-\mathcal{W}^-+J_0^+\mathcal{W}^+\Big)\,dx\\
&\leq C\sum_{k=1,2}\int_\mathbb{R}\Big((1+|S|)J_k^-\mathcal{W}^-
+(1+|R|)J_k^+\mathcal{W}^+\Big)\,dx\\
&\quad+\int_\mathbb{R} \Big(G_1(t)J_0^-\mathcal{W}^-+G_2(t)J_0^+\mathcal{W}^+\Big)\,dx-\frac{2\gamma_1}{\alpha_2}\int_\mathbb{R} \Big(S^2J_0^-\mathcal{W}^-+R^2J_0^+\mathcal{W}^+\Big)\,dx.
\end{split}
\end{equation}

\bigskip
\paragraph{\bf 1}  For $I_1$, using \eqref{R-S-eqn} and \eqref{balance1}, by a direct computation, one has
\begin{equation}\label{1+R}
\begin{split}
(1+R^2)_t+[\frac{c_1}{\alpha}(1+R^2)]_x=&\frac{2a_2}{\alpha}
(RS^2-R^2S)+\frac{2c_2b}{\alpha}RS-\frac{c_2\partial_x c_1-c_1\partial_xc_2}{\alpha(c_2-c_1)}R^2\\
&-\frac{2a_1}{\alpha}(R-S)+\frac{\alpha\partial_x c_1-c_1\partial_x\alpha}{\alpha^2}.
\end{split}\end{equation}
With this help and by \eqref{wz}, we obtain
\begin{equation*}\label{wt1}
\begin{split}
&\big[w(1+R^2)\big]_t+\big[\frac{c_1}{\alpha}w(1+R^2)\big]_x\\
&=(w_t+\frac{c_1}{\alpha}w_x)(1+R^2)
+w[(1+R^2)_t+\big(\frac{c_1}{\alpha}(1+R^2)\big)_x]\\
&=\big[\frac{\alpha\partial_x c_1-c_1\partial_x\alpha}{\alpha^2}w-\frac{2a_1}{\alpha}(c_2-c_1)(v+
\frac{R-S}{c_2-c_1}w)\big](1+R^2)+w\big[\frac{2a_2}{\alpha}
(RS^2-R^2S)\\
&\quad+\frac{2c_2b}{\alpha}RS-\frac{c_2\partial_x c_1-c_1\partial_xc_2}{\alpha(c_2-c_1)}R^2-\frac{2a_1}{\alpha}(R-S)+\frac{\alpha\partial_x c_1-c_1\partial_x\alpha}{\alpha^2}\big]\\
&=-\frac{2a_1}{\alpha}(v+\frac{Rw-Sz}{c_2-c_1})(1+R^2)(c_2-c_1)
-\frac{2a_1Sz}{\alpha}(1+R^2)+\frac{w}{\alpha}\big[2(a_1-a_2)R^2S\\
&\quad+2a_2RS^2
+2c_2bRS+\frac{\partial_x(c_2-c_1) }{c_2-c_1}c_1R^2-\frac{\partial_x\alpha}{\alpha}c_1R^2-2a_1R+4a_1S
+2\frac{\alpha\partial_x c_1-c_1\partial_x\alpha}{\alpha}\big],
\end{split}
\end{equation*}
which, together with the uniform bounds \eqref{Westimate} on the weights, shows that
\begin{equation*}
\begin{split}
&\frac{d}{dt}\int_\mathbb{R}J_1^-\mathcal{W}^-\,dx = \frac{d}{dt}\int_\mathbb{R}|w|(1+R^2)\mathcal{W}^-\,dx \\
&\leq C\int_\mathbb{R} |w|(1+|RS^2|+|R^2S|+|RS|+|R^2|+|R|+|S|)\mathcal{W}^-\,dx\\
&\quad+C\int_\mathbb{R} |z|(|S|+|R^2S|)\mathcal{W}^+\,dx+C\int_\mathbb{R}  \Big|v+\frac{Rw-Sz}{c_2-c_1}\Big|(1+R^2)\mathcal{W}^-\,dx\\
&\quad+G_1(t)\int_\mathbb{R} |w|(1+R^2)\mathcal{W}^-\,dx-\frac{2\gamma_1}{\alpha_2}\int_\mathbb{R} |w|(1+R^2)S^2\mathcal{W}^-\,dx.
\end{split}
\end{equation*}
Similarly, we can obtain the estimate for the other terms of $I_1$.  Thus, it holds that
\begin{equation}\label{I1est}
\begin{split}
\frac{d}{dt}I_1
\leq& C\sum_{k=1,2}\int_\mathbb{R}\Big((1+|S|)J_k^-\mathcal{W}^-
+(1+|R|)J_k^+\mathcal{W}^+\Big)\,dx\\
&+ C\int_\mathbb{R}\Big(S^2J_0^-\mathcal{W}^-+R^2J_0^+\mathcal{W}^+\Big)\,dx
\\
&+\int_\mathbb{R} \Big(G_1(t)J_1^-\mathcal{W}^-+G_2(t)J_1^+\mathcal{W}^+\Big)\,dx-\frac{\gamma_1}{\alpha_2}\int_\mathbb{R} \Big(S^2J_1^-\mathcal{W}^-+R^2J_1^+\mathcal{W}^+\Big)\,dx.
\end{split}
\end{equation}
Here we have used the fact that $|RS^2|\leq \frac{1}{2} (\frac{\gamma_1}{\alpha_2} R^2S^2+\frac{\alpha_2}{\gamma_1}S^2)$ and $|R^2S|\leq \frac{1}{2} (\frac{\gamma_1}{\alpha_2} R^2S^2+\frac{\alpha_2}{\gamma_1}R^2)$.

\bigskip
\paragraph{\bf 2}  To estimate the time derivative of $I_2$, recalling
\eqref{vx} and \eqref{vt}, we get the equation for the first order perturbation $v$:
\begin{equation}\label{vequ}
v_t+\frac{c_1}{\alpha}v_x=\frac{s}{\alpha}+
\frac{2a_1R-2a_2S}{\alpha}v+\frac{\partial_u c_1-\partial_u c_2}{\alpha(c_2-c_1)}Sv.
\end{equation}
Next, by \eqref{R-S-eqn} and \eqref{wz}, it holds that
\begin{equation}\label{Rwt}
\begin{split}
&\Big[\frac{Rw-Sz}{c_2-c_1}\Big]_t
+\frac{c_1}{\alpha}\Big[\frac{Rw-Sz}{c_2-c_1}\Big]_x\\
&=\frac{w}{c_2-c_1}(R_t+\frac{c_1}{\alpha}R_x)
-\frac{z}{c_2-c_1}(S_t+\frac{c_2}{\alpha}S_x)+\frac{z}{\alpha}S_x
+\frac{R}{c_2-c_1}(w_t+\frac{c_1}{\alpha}w_x)\\
&\quad-\frac{S}{c_2-c_1}(z_t+\frac{c_2}{\alpha}z_x)+\frac{S}{\alpha}z_x
+(Rw-Sz)\big[(\frac{1}{c_2-c_1})_t+\frac{c_1}{\alpha}(\frac{1}{c_2-c_1})_x\big]\\
&=\frac{w}{\alpha(c_2-c_1)}\Big[a_1R^2-\big(a_1+a_2+\frac{\partial_u c_2-\partial_u c_1}{c_2-c_1}\big)RS+a_2S^2+c_2bS+(d_1-\partial_x c_1)R\Big]\\
&\quad-\frac{z}{\alpha(c_2-c_1)}\Big[-a_1R^2+(a_1+a_2)RS-(a_2+\frac{\partial_u c_2-\partial_u c_1}{c_2-c_1})S^2+c_1bR+(d_2-\partial_x c_1)S\Big]\\
&\quad +\frac{z}{\alpha}S_x+\frac{S}{\alpha}z_x
-\frac{2a_1R}{\alpha}\big(v+\frac{R-S}{c_2-c_1}w\big)
+\frac{2a_2S}{\alpha}\big(v+\frac{R-S}{c_2-c_1}z\big).
\end{split}
\end{equation}
Consequently, in accordance with \eqref{1+R}, \eqref{vequ} and \eqref{Rwt}, we have
\begin{equation*}\label{Rwt1}
\begin{split}
&\Big[\big(v+\frac{Rw-Sz}{c_2-c_1}\big)(1+R^2)\Big]_t
+\Big[\frac{c_1}{\alpha}\big(v+\frac{Rw-Sz}{c_2-c_1}\big)(1+R^2)\Big]_x\\
&=\Big[(v_t+\frac{c_1}{\alpha}v_x)
+\big(\frac{Rw-Sz}{c_2-c_1}\big)_t
+\frac{c_1}{\alpha}\big(\frac{Rw-Sz}{c_2-c_1}\big)_x\Big](1+R^2)\\
&\quad+\big(v+\frac{Rw-Sz}{c_2-c_1}\big)
\big[(1+R^2)_t+\big(\frac{c_1}{\alpha}(1+R^2)\big)_x\big]\\
&=\frac{1+R^2}{\alpha}\Big[s+zS_x-\frac{a_1(w-z)}{c_2-c_1}R^2
+\frac{a_1+a_2}{c_2-c_1}(w-z)RS+\frac{c_1b(w-z)}{c_2-c_1}R\Big]
\\
&\quad+\frac{w(1+R^2)}{\alpha(c_2-c_1)}\big[a_2S^2+c_2bS\big]-\frac{z(1+R^2)}{\alpha(c_2-c_1)}\big[a_2S^2+(d_2-\partial_x c_1)S\big]\\
&\quad
+\frac{1+R^2}{\alpha}(Sz_x-\frac{2a_2(w-z)}{c_2-c_1}RS)+\big(v+\frac{Rw-Sz}{c_2-c_1}\big)\Big[
\frac{2a_2}{\alpha}(RS^2-R^2S)+\frac{2c_2b}{\alpha}RS
\\
&\quad -\frac{c_2\partial_xc_1-c_1\partial_xc_2}{\alpha(c_2-c_1)}R^2+\frac{\partial_uc_1-\partial_uc_2}{\alpha(c_2-c_1)}(1+R^2)S
+\frac{\alpha\partial_xc_1-c_1\partial_x\alpha}{\alpha^2}-\frac{2a_1}{\alpha}(R-S)
\Big].
\end{split}
\end{equation*}
This immediately gives that
\begin{equation*}
\begin{split}
&\frac{d}{dt}\int_\mathbb{R}J_2^-\mathcal{W}^-\,dx = \frac{d}{dt}\int_\mathbb{R}\Big|v+\frac{Rw-Sz}{c_2-c_1}\Big|(1+R^2)\mathcal{W}^-\,dx \\
&\leq C\int_\mathbb{R} |w|(1+R^2)(S^2+|S|)\mathcal{W}^-\,dx+C\int_\mathbb{R} |z|(1+R^2)(S^2+|S|)\mathcal{W}^+\,dx\\
&\quad+C\int_\mathbb{R}  \Big|v+\frac{Rw-Sz}{c_2-c_1}\Big|(1+|R|+|S|+R^2+|RS|+|R^2S|+|RS^2|)\mathcal{W}^-\,dx\\
&\quad+C\int_\mathbb{R} |\hat{s}|(1+R^2)\mathcal{W}^+\,dx+C\int_\mathbb{R} \Big|Sz_x-\frac{2a_2(w-z)}{c_2-c_1}RS\Big|(1+R^2)\mathcal{W}^+\,dx\\
&\quad+G_1(t)\int_\mathbb{R} \Big|v+\frac{Rw-Sz}{c_2-c_1}\Big|(1+R^2)\mathcal{W}^-\,dx
-\frac{2\gamma_1}{\alpha_2}\int_\mathbb{R} \Big|v+\frac{Rw-Sz}{c_2-c_1}\Big|(1+R^2)S^2\mathcal{W}^-\,dx.
\end{split}
\end{equation*}
In a similar way, we have
\begin{equation}\label{I2est}
\begin{split}
\frac{d}{dt}I_2
\leq&  C\int_\mathbb{R}\Big((1+|S|)J_2^-\mathcal{W}^-
+(1+|R|)J_2^+\mathcal{W}^+\Big)\,dx\\
&+
C\sum_{i=1,3,5}\int_\mathbb{R}\Big((1+S^2)J_i^-\mathcal{W}^-
+(1+R^2)J_i^+\mathcal{W}^+\Big)\,dx\\
&+\int_\mathbb{R} \Big(G_1(t)J_2^-\mathcal{W}^-+G_2(t)J_2^+\mathcal{W}^+\Big)\,dx-\frac{2\gamma_1}{\alpha_2}\int_\mathbb{R} \Big(S^2J_2^-\mathcal{W}^-+R^2J_2^+\mathcal{W}^+\Big)\,dx.
\end{split}
\end{equation}

\bigskip
\paragraph{\bf 3}  We now turn to the time derivative of $I_3$, which is much more delicate than the other terms.
Differentiating $\eqref{R-S-eqn}_1$ with respect to $x$, we have
\begin{equation}\label{Rxt}
\begin{split}
R_{tx}&+\big(\frac{c_1}{\alpha}R_x\big)_x=\frac{2a_1}{\alpha}RR_x
-\frac{a_1+a_2}{\alpha}(RS_x+SR_x)+\frac{2a_2}{\alpha}SS_x
+\frac{c_2b}{\alpha}S_x-\frac{d_1}{\alpha}R_x\\
&+\frac{B_1R^3-B_2S^3}{c_2-c_1}+\frac{B_1+2B_2}{c_2-c_1}RS^2
-\frac{2B_1+B_2}{c_2-c_1}R^2S+(A_1-\frac{B_5}{c_2-c_1})R^2\\
&+(A_2-\frac{B_4}{c_2-c_1})S^2-(A_1+A_2-\frac{B_4+B_5}{c_2-c_1})RS
-A_4R+A_3S,
\end{split}\end{equation}
where we have used the notations in \eqref{rt} and \[A_i=\frac{\alpha \partial_x a_i-a_i \partial_x \alpha}{\alpha^2},i=1,2,\quad A_3=\frac{\alpha \partial_x (c_2 b)-c_2b \partial_x \alpha}{\alpha^2},
\]
and
 \[
A_4=\frac{\alpha \partial_x d_1-d_1\partial_x \alpha}{\alpha^2},\quad A_5=\frac{\alpha \partial_x d_2-d_2\partial_x \alpha}{\alpha^2}.\]
Then it follows from \eqref{rt}, \eqref{wz} and \eqref{Rxt} that
\begin{equation}\label{rtest1}
\begin{split}
&\big[r+wR_x\big]_t+\big[\frac{c_1}{\alpha}(r+wR_x)\big]_x\\
&=(r_t+\frac{c_1}{\alpha}r_x)+\big(\frac{c_1}{\alpha}\big)_x r
+R_x(w_t+\frac{c_1}{\alpha}w_x)
+w[R_{tx}+\big(\frac{c_1}{\alpha}R_x\big)_x]\\
&=(v+\frac{R-S}{c_2-c_1}w)\big(B_1R^2-(B_1+B_2)RS+B_2S^2+B_3S-B_4R\big)
-\frac{a_1+a_2}{\alpha}R(s+S_xw)\\
&\quad+\frac{2a_2}{\alpha}S(s+S_xw)+\frac{c_2b}{\alpha}(s+S_xw)
+\frac{a_1-a_2}{\alpha}S(r+R_xw)-\frac{d_1}{\alpha}(r+R_xw)\\
&\quad+\frac{\alpha\partial_xc_1-c_1\partial_x\alpha}{\alpha^2}(r+R_xw)
+w\big[A_1R^2-(A_1+A_2)RS+A_2S^2+A_3S-A_4R\big].\\
\end{split}
\end{equation}
Furthermore, for the third term of $J_3^-$, we derive from \eqref{R-S-eqn} \eqref{balance1} and \eqref{wz} that
\begin{equation}\label{rtest2}
\begin{split}
&\Big[\frac{a_2(w-z)}{c_2-c_1}S^2\Big]_t
+\Big[\frac{c_1}{\alpha}\frac{a_2(w-z)}{c_2-c_1}S^2\Big]_x\\
&=\frac{a_2S^2}{c_2-c_1}(w_t+\frac{c_1}{\alpha}w_x)
-\frac{a_2S^2}{c_2-c_1}(z_t+\frac{c_2}{\alpha}z_x)+\frac{a_2S^2}{\alpha}z_x
+\frac{a_2(w-z)}{c_2-c_1}\big[(S^2)_t+\big(\frac{c_2}{\alpha}S^2\big)_x\big]\\
&\quad-\frac{a_2(w-z)}{c_2-c_1}\big[\big(\frac{c_2}{\alpha}S^2\big)_x
-\big(\frac{c_1}{\alpha}S^2\big)_x\big]
+(w-z)S^2\big[(\frac{a_2}{c_2-c_1})_t+\frac{c_1}{\alpha}(\frac{a_2}{c_2-c_1})_x\big]\\
&=\frac{2a_2(a_2-a_1)S^2}{\alpha}\big(v+\frac{Rw-Sz}{c_2-c_1}\big)
+\frac{w-z}{\alpha(c_2-c_1)}(2a_1a_2+\alpha B_2)S^3-\frac{2a_1a_2(w-z)}{\alpha(c_2-c_1)}R^2S\\
&\quad+\frac{a_2(\alpha\partial_x c_1-c_1\partial_x \alpha)}{\alpha^2(c_2-c_1)}(w-z)S^2+[\frac{\partial_x c_1-\partial_x c_2}{\alpha(c_2-c_1)}+\frac{\partial_x \alpha}{\alpha^2}]a_2 S^2 w+\frac{a_2}{\alpha}S^2z_x\\
&\quad +\frac{c_1\partial_x a_2-a_2\partial_x c_1}{\alpha(c_2-c_1)}(w-z)S^2
+\frac{2a_2c_1b(w-z)}{\alpha(c_2-c_1)}RS-\frac{2a_2}{\alpha}SS_x(w-z).
\end{split}
\end{equation}
The other terms of $J_3^-$ can be estimated similarly. On the one hand, using \eqref{R-S-eqn} and \eqref{wz}  we have
\begin{equation}\label{rtest3}
\begin{split}
&\Big[-\frac{a_1+a_2}{c_2-c_1}(w-z)RS\Big]_t
+\Big[-\frac{c_1}{\alpha}\frac{a_1+a_2}{c_2-c_1}(w-z)RS\Big]_x\\
&=-\frac{a_1+a_2}{c_2-c_1}RS(w_t+\frac{c_1}{\alpha}w_x)
+\frac{a_1+a_2}{c_2-c_1}RS(z_t+\frac{c_2}{\alpha}z_x)
-\frac{a_1+a_2}{\alpha}RSz_x\\
&\quad-\frac{a_1+a_2}{c_2-c_1}(w-z)S\big[R_t+\frac{c_1}{\alpha}R_x\big]
-\frac{a_1+a_2}{c_2-c_1}(w-z)R\big[S_t+\frac{c_2}{\alpha}S_x\big]\\
&\quad+\frac{a_1+a_2}{\alpha}(w-z)RS_x
-(w-z)RS\big[(\frac{a_1+a_2}{c_2-c_1})_t+\big(\frac{c_1}{\alpha}\frac{a_1+a_2}{c_2-c_1}\big)_x\big]\\
&=\frac{2(a_1^2-a_2^2)}{\alpha}RS\big(v+\frac{Rw-Sz}{c_2-c_1}\big)
+\frac{w-z}{\alpha(c_2-c_1)}(2a_2^2-3a_1^2-a_1a_2)RS^2
-\frac{a_1+a_2}{\alpha}RSz_x
\\
&\quad+\frac{w-z}{\alpha(c_2-c_1)}(2a_2^2+a_1^2+3a_1a_2)R^2S
+\frac{a_1(a_1+a_2)}{\alpha(c_2-c_1)}(w-z)R^3-\frac{a_2(a_1+a_2)}{\alpha(c_2-c_1)}(w-z)S^3\\
&\quad+\frac{(a_1+a_2)RS}{\alpha^2(c_2-c_1)}\Big[2(c_1\partial_x \alpha-\alpha\partial_x c_1)w+\big(\alpha(\partial_x c_1+\partial_x c_2)-(c_2+c_1)\partial_x\alpha\big)z\Big]\\
&\quad -\frac{a_1+a_2}{\alpha(c_2-c_1)}(w-z)\big[(a_1+a_2)R^2S+c_1bR^2
+c_2bS^2-(d_1+d_2)RS\big]
\\
&\quad-\frac{c_1(w-z)}{\alpha(c_2-c_1)^2}\big[(c_2-c_1)\partial_x(a_1+a_2)
-(a_1+a_2)\partial_x(c_2-c_1)\big]RS+\frac{a_1+a_2}{\alpha}RS_x(w-z)\\
&\quad -\frac{w-z}{\alpha(c_2-c_1)^2}\big[(c_2-c_1)\partial_u(a_1+a_2)
-(a_1+a_2)\partial_u(c_2-c_1)\big]RS^2.
\end{split}
\end{equation}
On the other hand, by \eqref{R-S-eqn} and \eqref{wz} we obtain
\begin{equation}\label{rtest4}
\begin{split}
&\Big[\frac{c_2b(w-z)}{c_2-c_1}S\Big]_t
+\Big[\frac{c_1}{\alpha}\frac{c_2b(w-z)}{c_2-c_1}S\Big]_x\\
&=\frac{c_2bS}{c_2-c_1}(w_t+\frac{c_1}{\alpha}w_x)
-\frac{c_2bS}{c_2-c_1}(z_t+\frac{c_2}{\alpha}z_x)+\frac{c_2bS}{\alpha}z_x
+\frac{c_2b(w-z)}{c_2-c_1}\big(S_t+\frac{c_2}{\alpha}S_x\big)\\
&\quad-\frac{c_2bS_x(w-z)}{\alpha}
+(w-z)S\big[(\frac{c_2b}{c_2-c_1})_t+\big(\frac{c_1}{\alpha}\frac{c_2b}{c_2-c_1}\big)_x\big]\\
&=\frac{2c_2b(a_2-a_1)S}{\alpha}\big(v+\frac{Rw-Sz}{c_2-c_1}\big)
+\frac{c_2b(w-z)}{\alpha(c_2-c_1)}(4a_1-a_2)S^2
+\frac{c_2b}{\alpha}Sz_x\\
&\quad+\frac{(c_2-c_1)\partial_u(c_2b)-c_2b\partial_u (c_2-c_1)}{\alpha(c_2-c_1)^2}(w-z)S^2
-\frac{2c_2b(a_1+a_2)(w-z)}{\alpha(c_2-c_1)}RS
\\
&\quad -\frac{a_1c_2b(w-z)}{\alpha(c_2-c_1)}R^2+\frac{c_2b(w-z)}{\alpha(c_2-c_1)}[(a_1+a_2)RS+c_1bR-d_2S]
\\
&\quad+\frac{c_2bS}{\alpha^2(c_2-c_1)}\big[2(\alpha\partial_xc_1-c_1\partial_x \alpha)w-\big(\alpha\partial_x(c_1+c_2)-(c_1+c_2)\partial_x \alpha\big)z\big]\\
&\quad+\frac{c_1(w-z)S}{\alpha(c_2-c_1)^2}\big[(c_2-c_1)\partial_x(c_2b)-c_2b\partial_x(c_2-c_1) \big]-\frac{c_2b}{\alpha}S_x(w-z).
\end{split}
\end{equation}
In view of \eqref{rseq}, combining \eqref{rtest1}--\eqref{rtest4} yields
\begin{equation}\label{rtest5}
\begin{split}
&\hat{r}_t+\big(\frac{c_1}{\alpha}\hat{r}\big)_x\\
&=(v+\frac{Rw-SZ}{c_2-c_1})\Big(B_1R^2-(B_1+B_2)RS+B_2S^2+B_3S
-B_4R-\frac{2a_2(a_1-a_2)}{\alpha}S^2\Big)
\\
&\quad+(v+\frac{Rw-Sz}{c_2-c_1})\Big(\frac{2(a_1^2-a_2^2)}{\alpha}RS-\frac{2(a_1-a_2)c_2b}{\alpha}S\Big)
+\frac{c_2b-(a_1+a_2)R}{\alpha}\hat{s}
\\
&\quad+\Big(\frac{(a_1-a_2)S-d_1}{\alpha}+\frac{\alpha\partial_x c_1-c_1\partial_x\alpha}{\alpha^2}\Big)\hat{r}+\frac{a_2}{\alpha}(2S\hat{s}+S^2z_x-\frac{2a_2(w-z)}{c_2-c_1}RS^2)\\
&\quad +\frac{c_2b-(a_1+a_2)R}{\alpha}\big(Sz_x-\frac{2a_2(w-z)}{c_2-c_1}RS\big)
+\frac{R^2S(w-z)}{\alpha(c_2-c_1)}[a_1^2+a_1a_2-\alpha B_1]\\
&\quad+A_1wR^2-\frac{RS^2(w-z)}{\alpha(c_2-c_1)}[a_2^2+3a_1a_2]
+\frac{(a_1+a_2)RS}{\alpha(c_2-c_1)}[w\partial_x c_2-z\partial_x c_1-d_2(w-z)]\\
&\quad+\frac{RS}{\alpha(c_2-c_1)}\big[(c_1z-c_2w)\partial_x(a_1+a_2)
+(\alpha B_4-2a_1c_2b)(w-z)\big]-A_4wR\\
&\quad
+\frac{S}{\alpha(c_2-c_1)}\big[(c_2w-c_1z)\partial_x(c_2b)
+(d_1+d_2)c_2b(w-z)+c_2b(z\partial_x c_1-w\partial_x c_2)\big]\\
&\quad+\frac{S^2}{\alpha(c_2-c_1)}\big[(c_2w-c_1z)\partial_x a_2
+a_2(d_1+c_2b)(w-z)+a_2(z\partial_x c_1-w\partial_x c_2)\big].
\end{split}
\end{equation}
Hence, we can conclude that
\begin{equation*}
\begin{split}
&\frac{d}{dt}\int_\mathbb{R}J_3^-\mathcal{W}^-\,dx = \frac{d}{dt}\int_\mathbb{R}\big|\hat{r}\big|\mathcal{W}^-\,dx \\
&\leq C\int_\mathbb{R} |w|(|R|+|S|+R^2+S^2+|RS|+|R^2S|+|RS^2|)\mathcal{W}^-\,dx\\
&\quad+C\int_\mathbb{R} |z|(|S|+S^2+|RS|+|R^2S|+|RS^2|)\mathcal{W}^+\,dx+C\int_\mathbb{R} |\hat{r}|(1+|S|)\mathcal{W}^-\,dx\\
&\quad+C\int_\mathbb{R} |\hat{s}|(1+|R|)\mathcal{W}^+\,dx+C\int_\mathbb{R} \Big|Sz_x-\frac{2a_2(w-z)}{c_2-c_1}RS\Big|(1+|R|)\mathcal{W}^+\,dx\\
&\quad
+C\int_\mathbb{R} \Big|2S\hat{s}+S^2z_x-\frac{2a_2(w-z)}{c_2-c_1}RS^2\Big|\mathcal{W}^+\,dx
+C\int_\mathbb{R}  \Big|v+\frac{Rw-Sz}{c_2-c_1}\Big|(R^2+|R|)\mathcal{W}^-\,dx\\
&\quad
+C\int_\mathbb{R}  \Big|v+\frac{Rw-Sz}{c_2-c_1}\Big|(S^2+|S|)\mathcal{W}^+\,dx+G_1(t)\int_\mathbb{R} \big|\hat{r}\big|\mathcal{W}^-\,dx
-\frac{2\gamma_1}{\alpha_2}\int_\mathbb{R} \big|\hat{r}\big|S^2\mathcal{W}^-\,dx.
\end{split}
\end{equation*}
Applying the similar idea used above to get
\begin{equation}\label{I3est}
\begin{split}
\frac{d}{dt}I_3
\leq&  C\sum_{k=2,3,5,6}\int_\mathbb{R}\Big((1+|S|)J_k^-\mathcal{W}^-
+(1+|R|)J_k^+\mathcal{W}^+\Big)\,dx\\
&+
C\int_\mathbb{R}\Big((1+S^2)J_1^-\mathcal{W}^-
+(1+R^2)J_1^+\mathcal{W}^+\Big)\,dx\\
&+\int_\mathbb{R} \Big(G_1(t)J_3^-\mathcal{W}^-+G_2(t)J_3^+\mathcal{W}^+\Big)\,dx-\frac{2\gamma_1}{\alpha_2}\int_\mathbb{R} \Big(S^2J_3^-\mathcal{W}^-+R^2J_3^+\mathcal{W}^+\Big)\,dx.
\end{split}
\end{equation}

\bigskip
\paragraph{\bf 4}  Consider $I_4$, we first differentiate $\eqref{wz}_1$ with respect to $x$ to get
\begin{equation}\label{wxt}
\begin{split}
w_{tx}&+\big(\frac{c_1}{\alpha}w_x\big)_x=-\frac{2a_1}{\alpha}(r+R_xw)
+\frac{2a_1}{\alpha}(s+S_xw)-\frac{2a_1}{\alpha}(R-S)w_x
\\
&-\frac{2a_1}{\alpha}(\partial_x c_2-\partial_x c_1)v-2(c_2-c_1)(A_1+B_1\frac{R-S}{c_2-c_1})(v+\frac{R-S}{c_2-c_1}w)\\
&+
\frac{\alpha \partial_xc_1-c_1\partial_x\alpha}{\alpha^2}w_x+\big(\frac{\alpha \partial_xc_1-c_1\partial_x\alpha}{\alpha^2}\big)_xw.
\end{split}\end{equation}
With this help, utilizing the estimates \eqref{R-S-eqn}, \eqref{wz} and \eqref{wxt}, we can derive
\begin{equation}\label{wxtest}
\begin{split}
&\Big[w_x+\frac{2a_1(w-z)}{c_2-c_1}S\Big]_t
+\Big[\frac{c_1}{\alpha}\big(w_x+\frac{2a_1(w-z)}{c_2-c_1}S\big)\Big]_x\\
&=w_{xt}+\big(\frac{c_1}{\alpha}w_x\big)_x+\frac{2a_1S}{c_2-c_1}(w_t+\frac{c_1}{\alpha}w_x)
-\frac{2a_1S}{c_2-c_1}(z_t+\frac{c_2}{\alpha}z_x)+\frac{2a_1S}{\alpha}z_x
\\
&\quad+\frac{2a_1(w-z)}{c_2-c_1}\big(S_t+\frac{c_2}{\alpha}S_x\big)-\frac{2a_1(w-z)}{\alpha}S_x
+2(w-z)S\big[(\frac{a_1}{c_2-c_1})_t+(\frac{c_1}{\alpha}\frac{a_1}{c_2-c_1})_x\big]\\
&=-\frac{2a_1}{\alpha}\hat{r}+\frac{2a_1}{\alpha}\hat{s}
-\frac{2a_1}{\alpha}\big(Rw_x+\frac{2a_1(w-z)}{c_2-c_1}RS\big)+
\frac{2a_1}{\alpha}\big(Sz_x-\frac{2a_2(w-z)}{c_2-c_1}RS\big)\\
&\quad-2\Big[A_1(c_2-c_1)+B_1(R-S)+\frac{2a_1S}{\alpha}(a_1-a_2)
+\frac{a_1(\partial_x c_2-\partial_x c_1)}{\alpha}\Big]\big(v+\frac{Rw-Sz}{c_2-c_1}\big)
\\
&\quad+\big(\frac{2a_1S}{\alpha}+\frac{\alpha\partial_x c_1-c_1\partial_x\alpha}{\alpha^2}\big)\big(w_x+\frac{2a_1(w-z)}{c_2-c_1}S\big)
+\big(\frac{\alpha\partial_x c_1-c_1\partial_x \alpha}{\alpha^2}\big)_xw\\
&\quad +\frac{2a_1(\partial_x c_2-\partial_x c_1)}{\alpha(c_2-c_1)}Rw
 +\frac{2a_1Sz}{\alpha(c_2-c_1)}[2\partial_x(c_1-c_2)
-\frac{c_1-c_2}{\alpha}\partial_x \alpha]\\
&\quad+\frac{2(w-z)S}{\alpha(c_2-c_1)}\big[2a_1a_2S-a_1(a_1+a_2)R+a_1(d_2+c_2b-\partial_x c_2)+\alpha B_1R+c_2\partial_x a_1\big]
.
\end{split}
\end{equation}
This in turn yields the estimate
\begin{equation*}
\begin{split}
&\frac{d}{dt}\int_\mathbb{R}J_4^-\mathcal{W}^-\,dx = \frac{d}{dt}\int_\mathbb{R}\Big|w_x+\frac{2a_1(w-z)}{c_2-c_1}S\Big|\mathcal{W}^-\,dx \\
&\leq C\int_\mathbb{R} |w|(1+|R|+|S|+S^2+|RS|)\mathcal{W}^-\,dx+C\int_\mathbb{R} |z|(|S|+S^2+|RS|)\mathcal{W}^+\,dx\\
&\quad+C\int_\mathbb{R} |\hat{r}|\mathcal{W}^-\,dx+C\int_\mathbb{R} |\hat{s}|\mathcal{W}^+\,dx+C\int_\mathbb{R} \Big|Rw_x+\frac{2a_1(w-z)}{c_2-c_1}RS\Big|\mathcal{W}^-\,dx\\
&\quad+C\int_\mathbb{R} \Big|Sz_x-\frac{2a_2(w-z)}{c_2-c_1}RS\Big|\mathcal{W}^+\,dx+C\int_\mathbb{R} \Big|w_x+\frac{2a_1(w-z)}{c_2-c_1}S\Big|(1+|S|)\mathcal{W}^-\,dx\\
&\quad
+C\int_\mathbb{R}  \Big|v+\frac{Rw-Sz}{c_2-c_1}\Big|(1+|R|)\mathcal{W}^-\,dx+C\int_\mathbb{R}  \Big|v+\frac{Rw-Sz}{c_2-c_1}\Big||S|\mathcal{W}^+\,dx\\
&\quad
+G_1(t)\int_\mathbb{R} \Big|w_x+\frac{2a_1(w-z)}{c_2-c_1}S\Big|\mathcal{W}^-\,dx
-\frac{2\gamma_1}{\alpha_2}\int_\mathbb{R} \Big|w_x+\frac{2a_1(w-z)}{c_2-c_1}S\Big|S^2\mathcal{W}^-\,dx.
\end{split}
\end{equation*}
Thus, we can get
\begin{equation}\label{I4est}
\begin{split}
\frac{d}{dt}I_4
\leq&  C\sum_{k=2,3,4,5}\int_\mathbb{R}\Big((1+|S|)J_k^-\mathcal{W}^-
+(1+|R|)J_k^+\mathcal{W}^+\Big)\,dx\\
&+
C\int_\mathbb{R}\Big((1+S^2)J_1^-\mathcal{W}^-
+(1+R^2)J_1^+\mathcal{W}^+\Big)\,dx\\
&+\int_\mathbb{R} \Big(G_1(t)J_4^-\mathcal{W}^-+G_2(t)J_4^+\mathcal{W}^+\Big)\,dx-\frac{2\gamma_1}{\alpha_2}\int_\mathbb{R} \Big(S^2J_4^-\mathcal{W}^-+R^2J_4^+\mathcal{W}^+\Big)\,dx.
\end{split}
\end{equation}

\bigskip
\paragraph{\bf 5}  Next, we deal with the time derivative of $I_5$. By \eqref{R-S-eqn} and \eqref{wxtest}, it holds that
\begin{equation*}
\begin{split}
&\Big[R\big(w_x+\frac{2a_1(w-z)}{c_2-c_1}S\big)\Big]_t
+\Big[\frac{c_1}{\alpha}R\big(w_x+\frac{2a_1(w-z)}{c_2-c_1}S\big)\Big]_x\\
&=(R_t+\frac{c_1}{\alpha}R_x)\big(w_x+\frac{2a_1(w-z)}{c_2-c_1}S\big)\\
&\quad+R\Big[\big(w_x+\frac{2a_1(w-z)}{c_2-c_1}S\big)_t+\big(\frac{c_1}{\alpha}(w_x+\frac{2a_1(w-z)}{c_2-c_1}S)\big)_x\Big]\\
&=-\frac{a_1}{\alpha}[2R\hat{r}+R^2w_x+\frac{2a_1(w-z)}{c_2-c_1}R^2S]
+\frac{2a_1}{\alpha}R\hat{s}+
\frac{2a_1R}{\alpha}\big(Sz_x-\frac{2a_2(w-z)}{c_2-c_1}RS\big)\\
&\quad+\big(\frac{a_1-a_2}{\alpha}S-\frac{d_1}{\alpha}
+\frac{\alpha\partial_x c_1-c_1\partial_x\alpha}{\alpha^2}\big)\big(Rw_x+\frac{2a_1(w-z)}{c_2-c_1}RS\big)\\
&\quad-2R\Big[A_1(c_2-c_1)+B_1(R-S)+\frac{2a_1S}{\alpha}(a_1-a_2)
+\frac{a_1(\partial_x c_2-\partial_x c_1)}{\alpha}\Big]\big(v+\frac{Rw-Sz}{c_2-c_1}\big)
\\
&\quad+\frac{S}{\alpha}(a_2S+c_2b)\big(w_x+\frac{2a_1(w-z)}{c_2-c_1}S\big)
+\big(\frac{\alpha\partial_x c_1-c_1\partial_x \alpha}{\alpha^2}\big)_xRw\\
&\quad +\frac{2a_1(\partial_x c_2-\partial_x c_1)}{\alpha(c_2-c_1)}R^2w
 +\frac{2a_1RSz}{\alpha(c_2-c_1)}[2\partial_x(c_1-c_2)
-\frac{c_1-c_2}{\alpha}\partial_x \alpha]\\
&\quad+\frac{2(w-z)RS}{\alpha(c_2-c_1)}\big[2a_1a_2S-a_1(a_1+a_2)R+a_1(d_2+c_2b-\partial_x c_2)+\alpha B_1R+c_2\partial_x a_1\big].
\end{split}
\end{equation*}
This immediately gives that
\begin{equation*}
\begin{split}
&\frac{d}{dt}\int_\mathbb{R}J_5^-\mathcal{W}^-\,dx = \frac{d}{dt}\int_\mathbb{R}\Big|Rw_x+\frac{2a_1(w-z)}{c_2-c_1}RS\Big|\mathcal{W}^-\,dx \\
&\leq C\int_\mathbb{R} |w|(|R|+|RS|+|R^2|+|R^2S|+|RS^2|)\mathcal{W}^-\,dx\\
&\quad+C\int_\mathbb{R} |z|(|RS|+|RS^2|+|R^2S|)\mathcal{W}^+\,dx+C\int_\mathbb{R}  \Big|v+\frac{Rw-Sz}{c_2-c_1}\Big|(|R|+R^2+|RS|)\mathcal{W}^-\,dx \\
&\quad+C\int_\mathbb{R} |\hat{s}||R|\mathcal{W}^+\,dx+C\int_\mathbb{R} \Big|2R\hat{r}+R^2w_x+\frac{2a_1(w-z)}{c_2-c_1}R^2S\Big|\mathcal{W}^-\,dx
\\
&\quad+C\int_\mathbb{R} \Big|Sz_x-\frac{2a_2(w-z)}{c_2-c_1}RS\Big||R|\mathcal{W}^+\,dx+C\int_\mathbb{R} \Big|Rw_x+\frac{2a_1(w-z)}{c_2-c_1}RS\Big|(1+|S|)\mathcal{W}^-\,dx\\
&\quad+C\int_\mathbb{R} \Big|w_x+\frac{2a_1(w-z)}{c_2-c_1}S\Big|(|S|+S^2)\mathcal{W}^-\,dx+G_1(t)\int_\mathbb{R} \Big|Rw_x+\frac{2a_1(w-z)}{c_2-c_1}RS\Big|\mathcal{W}^-\,dx\\
&\quad
-\frac{2\gamma_1}{\alpha_2}\int_\mathbb{R} \Big|Rw_x+\frac{2a_1(w-z)}{c_2-c_1}RS\Big|S^2\mathcal{W}^-\,dx.
\end{split}
\end{equation*}
Similarly, we can bound the remaining term of $I_5$. We thus conclude
\begin{equation}\label{I5est}
\begin{split}
\frac{d}{dt}I_5
\leq&  C\sum_{k=2,3,5,6}\int_\mathbb{R}\Big((1+|S|)J_k^-\mathcal{W}^-
+(1+|R|)J_k^+\mathcal{W}^+\Big)\,dx\\
&+
C\sum_{i=1,4}\int_\mathbb{R}\Big((1+S^2)J_i^-\mathcal{W}^-
+(1+R^2)J_i^+\mathcal{W}^+\Big)\,dx\\
&+\int_\mathbb{R} \Big(G_1(t)J_5^-\mathcal{W}^-+G_2(t)J_5^+\mathcal{W}^+\Big)\,dx-\frac{2\gamma_1}{\alpha_2}\int_\mathbb{R} \Big(S^2J_5^-\mathcal{W}^-+R^2J_5^+\mathcal{W}^+\Big)\,dx.
\end{split}
\end{equation}

\bigskip
\paragraph{\bf 6}  Finally, we repeat the same procedure on $I_6$.
Using \eqref{R-S-eqn}, \eqref{balance1}, \eqref{rtest5} and \eqref{wxt}, we achieve
\begin{equation}\label{2Rest1}
\begin{split}
&\big[2R\hat{r}+R^2w_x\big]_t+\big[\frac{c_1}{\alpha}(2R\hat{r}+R^2w_x)\big]_x\\
&=2R\Big(\hat{r}_t+\big(\frac{c_1}{\alpha}\hat{r}\big)_x\Big)
+2\hat{r}(R_t+\frac{c_1}{\alpha}R_x)+R^2[w_{xt}+(\frac{c_1}{\alpha}w_x)_x]
+w_x[(R^2)_{t}+\frac{c_1}{\alpha}\big(R^2\big)_x]\\
&=2(v+\frac{Rw-Sz}{c_2-c_1})\Big[B_2RS^2-B_2R^2S+B_3RS-B_4R^2-A_1(c_2-c_1)R^2
+\frac{2(a_1^2-a_2^2)}{\alpha}R^2S\\
&\quad
-\frac{2a_2(a_1-a_2)}{\alpha}RS^2-\frac{a_1(\partial_x c_2-\partial_x c_1)}{\alpha}R^2+\frac{2c_2b(a_2-a_1)}{\alpha}RS\Big]+\frac{2\hat{r}}{\alpha}(a_2S^2+c_2bS)
\\
&\quad-\frac{2\hat{s}}{\alpha}(a_2R^2-c_2bR)+\big(\frac{c_1\partial_x c_2-c_2\partial_xc_1}{\alpha(c_2-c_1)}-\frac{2a_2S}{\alpha}\big) \big(2R\hat{r}+R^2w_x+\frac{2a_1(w-z)}{c_2-c_1}R^2S\big)
\\
&\quad+\frac{2a_2R}{\alpha} \big(2S\hat{s}+S^2z_x-\frac{2a_2(w-z)}{c_2-c_1}RS^2\big)
+\frac{2c_2bS+2a_2S^2}{\alpha}\big(Rw_x+\frac{2a_1(w-z)}{c_2-c_1}RS\big)\\
&\quad+\frac{2c_2bR-2(a_1+a_2)R^2}{\alpha}\big(Sz_x-\frac{2a_2(w-z)}{c_2-c_1}RS\big)
+\frac{2a_1(w-z)}{\alpha}R^2S_x+\frac{2a_1^2(w-z)}{\alpha(c_2-c_1)}R^4\\
&\quad+\frac{2a_1(\partial_x c_2-\partial_x c_1)}{\alpha(c_2-c_1)}R^3w-\frac{2a_1c_1b(w-z)}{\alpha(c_2-c_1)}R^3
+2A_1wR^3-2A_4wR^2-\frac{4a_1a_2(w-z)}{\alpha(c_2-c_1)}RS^3\\
&\quad+\frac{2RS}{\alpha(c_2-c_1)}\big[(c_2w-c_1z)\partial_x(c_2b)
+c_2b(d_1+d_2)(w-z)+c_2b(z\partial_xc_1-w\partial_xc_2)\big]\\
&\quad +\frac{2RS^2}{\alpha(c_2-c_1)}\big[(c_2w-c_1z)\partial_x a_2+(a_2(d_1+c_2b)-2a_1c_2b)(w-z)+a_2(z\partial_xc_1-w\partial_xc_2)\big]\\
&\quad +\frac{2R^2S}{\alpha(c_2-c_1)}\Big[(c_1z-c_2w)\partial_x (a_1+a_2)-a_1(\partial_xc_2-\partial_xc_1)z+\alpha(B_4+A_1(c_2-c_1))(w-z)\\
&\quad-a_1(c_2b+\frac{c_1\partial_x c_2-c_2\partial_x c_1}{c_2-c_1})(w-z)\Big]
+\frac{2(a_1+a_2)}{\alpha(c_2-c_1)}R^2S[w\partial_xc_2-z\partial_xc_1-d_2(w-z)]\\
&\quad-\frac{2(w-z)}{\alpha(c_2-c_1)}R^2S^2[\alpha B_1+a_2^2]-\frac{2(w-z)}{\alpha(c_2-c_1)}R^3S(a_1^2+a_1a_2)
+\big(\frac{\alpha\partial_x c_1-c_1\partial_x\alpha )}{\alpha^2}\big)_x R^2w.
\end{split}
\end{equation}
Moreover, it follows from \eqref{R-S-eqn}, \eqref{balance1} and \eqref{wz} that
\begin{equation}\label{2Rest2}
\begin{split}
&\Big[\frac{2a_1(w-z)}{c_2-c_1}R^2S\Big]_t
+\Big[\frac{c_1}{\alpha}\frac{2a_1(w-z)}{c_2-c_1}R^2S\Big]_x\\
&=\frac{2a_1R^2S}{c_2-c_1}(w_t+\frac{c_1}{\alpha}w_x)
-\frac{2a_1R^2S}{c_2-c_1}(z_t+\frac{c_2}{\alpha}z_x)+\frac{2a_1R^2S}{\alpha}z_x
-\frac{2a_1(w-z)}{\alpha}R^2S_x\\
&\quad+\frac{2a_1R^2(w-z)}{c_2-c_1}\big(S_t+\frac{c_2}{\alpha}S_x\big)+\frac{2a_1S(w-z)}{c_2-c_1}\big[(R^2)_t+(\frac{c_1}{\alpha}R^2)_x\big]
\\
&\quad+2(w-z)R^2S\big[(\frac{a_1}{c_2-c_1})_t+\frac{c_1}{\alpha}\big(\frac{a_1}{c_2-c_1}\big)_x\big]\\
&=\frac{4a_1(a_2-a_1)R^2S}{\alpha}\big(v+\frac{Rw-Sz}{c_2-c_1}\big)
+\frac{2a_1R^2}{\alpha}\big(Sz_x-\frac{2a_2(w-z)}{c_2-c_1}RS\big)
-\frac{2a_1(w-z)}{\alpha}R^2S_x\\
&\quad+
\frac{2R^2S^2(w-z)}{\alpha(c_2-c_1)}(\alpha B_1-a_1a_2)+\frac{2a_1R^2S}{\alpha(c_2-c_1)}
\big[\frac{\alpha\partial_xc_1-c_1\partial_x\alpha}{\alpha}w
-\frac{\alpha\partial_xc_2-c_2\partial_x\alpha}{\alpha}z\big]\\
&\quad -\frac{2R^2S(w-z)}{\alpha(c_2-c_1)}[a_1d_2+a_1\partial_x c_1-c_1\partial_x a_1]+\frac{4a_1a_2(w-z)}{\alpha(c_2-c_1)}RS^3
+\frac{4a_1c_2b(w-z)}{\alpha(c_2-c_1)}RS^2\\
&\quad
+\frac{2a_1(a_1+a_2)(w-z)}{\alpha(c_2-c_1)}R^3S
+\frac{2a_1c_1b(w-z)}{\alpha(c_2-c_1)}R^3
-\frac{2a_1^2(w-z)}{\alpha(c_2-c_1)}R^4.
\end{split}
\end{equation}
Combining \eqref{2Rest1} and \eqref{2Rest2}, we obtain
\begin{equation*}
\begin{split}
&\big[2R\hat{r}+R^2w_x+\frac{2a_1(w-z)}{c_2-c_1}R^2S\big]_t
+\big[\frac{c_1}{\alpha}\big(2R\hat{r}+R^2w_x+\frac{2a_1(w-z)}{c_2-c_1}R^2S\big)\big]_x\\
&=2(v+\frac{Rw-Sz}{c_2-c_1})\Big[B_2(RS^2-R^2S)+B_3RS-B_4R^2-A_1(c_2-c_1)R^2
-\frac{2(a_1-a_2)^2}{\alpha}R^2S\\
&\quad
-\frac{2a_2(a_1-a_2)}{\alpha}RS^2-\frac{a_1(\partial_x c_2-\partial_x c_1)}{\alpha}R^2+\frac{2c_2b(a_2-a_1)}{\alpha}RS\Big]+\frac{2\hat{r}}{\alpha}(a_2S^2+c_2bS)
\\
&\quad-\frac{2\hat{s}}{\alpha}(a_2R^2-c_2bR)+\big(\frac{c_1\partial_x c_2-c_2\partial_xc_1}{\alpha(c_2-c_1)}-\frac{2a_2S}{\alpha}\big) \big(2R\hat{r}+R^2w_x+\frac{2a_1(w-z)}{c_2-c_1}R^2S\big)
\\
&\quad+\frac{2a_2R}{\alpha} \big(2S\hat{s}+S^2z_x-\frac{2a_2(w-z)}{c_2-c_1}RS^2\big)
+\frac{2c_2bS+2a_2S^2}{\alpha}\big(Rw_x+\frac{2a_1(w-z)}{c_2-c_1}RS\big)\\
&\quad+\frac{2c_2bR-2a_2R^2}{\alpha}\big(Sz_x-\frac{2a_2(w-z)}{c_2-c_1}RS\big)
-\frac{2(w-z)}{\alpha(c_2-c_1)}R^2S^2[a_1a_2+a_2^2]
\\
&\quad +\frac{2R^2S}{\alpha(c_2-c_1)}\Big[(c_1-c_2)z\partial_x a_1+(c_1z-c_2w)\partial_xa_2+(\alpha B_4-a_1c_2b-a_2d_2)(w-z)\Big]\\
&\quad+\frac{2a_1R^2S}{\alpha(c_2-c_1)}\Big[(\partial_xc_1-\partial_xc_2)z
+\frac{(c_2-c_1)\partial_x \alpha}{\alpha}z\Big]+\frac{2a_2R^2S}{\alpha(c_2-c_1)}[w\partial_xc_2-z\partial_xc_1]
\end{split}
\end{equation*}
\begin{equation*}
\begin{split}
&\quad +\frac{2RS^2}{\alpha(c_2-c_1)}\big[(c_2w-c_1z)\partial_x a_2+a_2(d_1+c_2b)(w-z)+a_2(z\partial_xc_1-w\partial_xc_2)\big]\\
&\quad+\frac{2RS}{\alpha(c_2-c_1)}\big[(c_2w-c_1z)\partial_x(c_2b)
+c_2b(d_1+d_2)(w-z)+c_2b(z\partial_xc_1-w\partial_xc_2)\big]\\
&\quad
-\frac{\alpha\partial_u(\alpha\partial_x c_1-c_1\partial_x \alpha)-2(\alpha\partial_x c_1-c_1\partial_x\alpha)\partial_u \alpha}{\alpha^3}R^2Sw-2A_4R^2w\\
&\quad+\frac{\alpha\partial_x(\alpha\partial_x c_1-c_1\partial_x \alpha)-2(\alpha\partial_x c_1-c_1\partial_x\alpha)\partial_x \alpha}{\alpha^3}R^2w.\\
\end{split}
\end{equation*}
Thus, we can conclude
\begin{equation*}
\begin{split}
&\frac{d}{dt}\int_\mathbb{R}J_6^-\mathcal{W}^-\,dx = \frac{d}{dt}\int_\mathbb{R}\Big|2R\hat{r}+R^2w_x+\frac{2a_1(w-z)}{c_2-c_1}R^2S\Big|\mathcal{W}^-\,dx \\
&\leq C\int_\mathbb{R} |w|(|RS|+|R^2|+|R^2S|+|RS^2|+R^2S^2)\mathcal{W}^-\,dx+C\int_\mathbb{R} |\hat{r}|(|S|+S^2)\mathcal{W}^-\,dx\\
&\quad+C\int_\mathbb{R} |z|(|RS|+|RS^2|+|R^2S|+R^2S^2)\mathcal{W}^+\,dx+C\int_\mathbb{R} |\hat{s}|(| R|+R^2)\mathcal{W}^+\,dx\\
&\quad+C\int_\mathbb{R}  \Big|v+\frac{Rw-Sz}{c_2-c_1}\Big|(R^2+|RS|+|R^2S|)\mathcal{W}^-\,dx
+C\int_\mathbb{R}  \Big|v+\frac{Rw-Sz}{c_2-c_1}\Big||RS^2|\mathcal{W}^+\,dx  \\
&\quad+C\int_\mathbb{R} \Big|Rw_x+\frac{2a_1(w-z)}{c_2-c_1}RS\Big|(|S|+S^2)\mathcal{W}^-\,dx\\
&\quad+C\int_\mathbb{R} \Big|Sz_x-\frac{2a_2(w-z)}{c_2-c_1}RS\Big|(| R|+R^2)\mathcal{W}^+\,dx\\
&\quad+C\int_\mathbb{R} \Big|2R\hat{r}+R^2w_x+\frac{2a_1(w-z)}{c_2-c_1}R^2S\Big|(1+|S|)\mathcal{W}^-\,dx
\\
&\quad+C\int_\mathbb{R} \Big|2S\hat{s}+S^2z_x-\frac{2a_2(w-z)}{c_2-c_1}RS^2\Big||R|\mathcal{W}^+\,dx
\\
&\quad+G_1(t)\int_\mathbb{R} \Big|2R\hat{r}+R^2w_x+\frac{2a_1(w-z)}{c_2-c_1}R^2S\Big|\mathcal{W}^-\,dx\\
&\quad
-\frac{2\gamma_1}{\alpha_2}\int_\mathbb{R} \Big|2R\hat{r}+R^2w_x+\frac{2a_1(w-z)}{c_2-c_1}R^2S\Big|S^2\mathcal{W}^-\,dx.
\end{split}
\end{equation*}
In a similar way, we deduce that
\begin{equation}\label{I6est}
\begin{split}
\frac{d}{dt}I_6
\leq&  C\sum_{k=2,6}\int_\mathbb{R}\Big((1+|S|)J_k^-\mathcal{W}^-
+(1+|R|)J_k^+\mathcal{W}^+\Big)\,dx\\
&+
C\sum_{i=1,3,5}\int_\mathbb{R}\Big((1+S^2)J_i^-\mathcal{W}^-
+(1+R^2)J_i^+\mathcal{W}^+\Big)\,dx\\
&+\int_\mathbb{R} \Big(G_1(t)J_6^-\mathcal{W}^-+G_2(t)J_6^+\mathcal{W}^+\Big)\,dx-\frac{2\gamma_1}{\alpha_2}\int_\mathbb{R} \Big(S^2J_6^-\mathcal{W}^-+R^2J_6^+\mathcal{W}^+\Big)\,dx.
\end{split}
\end{equation}
Combining the estimates in \eqref{I0est}, \eqref{I1est}, \eqref{I2est}, \eqref{I3est}, \eqref{I4est}, \eqref{I5est} and \eqref{I6est}, and using \eqref{4.9}, we have
\[\left.\begin{array}{l}
\displaystyle\frac{dI_k}{dt}\leq~
C\sum_{\ell\in {\mathcal F}^l_k}
\left(\int_\mathbb{R}  (1+|S|)\,J_\ell^-\,W^- \,dx+\int_\mathbb{R}
(1+|R|)J_\ell^+\,\,W^+ \,dx
\right) \\[2mm]
\displaystyle \qquad \qquad+C\sum_{\ell\in {\mathcal F}^h_k}
\left(\int_\mathbb{R} (1+ S^2)\,J_\ell^-\,W^- \,dx+\int_\mathbb{R}   (1+R^2)\,J_\ell^+\,W^+\, dx
\right)\\[2mm]
\displaystyle\qquad \qquad+\int_\mathbb{R} \Big(G_1(t)J_k^-\mathcal{W}^-+G_2(t)J_k^+\mathcal{W}^+\Big)\,dx-\frac{\gamma_1}{\alpha_2}\left(\int_\mathbb{R}    S^2\,J_k^-\,W^- \,dx+\int_\mathbb{R}   R^2\,J_k^+\,W^+ \,dx
\right).
\end{array}\right.
\]
Here $\mathcal {F}^l_k,\mathcal {F}^h_k\subset\{0,1,2,\cdots,6\}$ are suitable sets of indices from the estimates  \eqref{I0est}, \eqref{I1est}, \eqref{I2est}, \eqref{I3est}, \eqref{I4est}, \eqref{I5est} and \eqref{I6est}, where a graphical summary of ${\mathcal F}_k^h$ is illustrated in Fig. \ref{f:wa36}.
For example, by \eqref{I6est}, $\mathcal {F}^l_6=\{2,6\}$ and $\mathcal {F}^h_6=\{1,3,5\}$.

\begin{figure}[htbp]
   \centering
\includegraphics[width=0.3\textwidth]{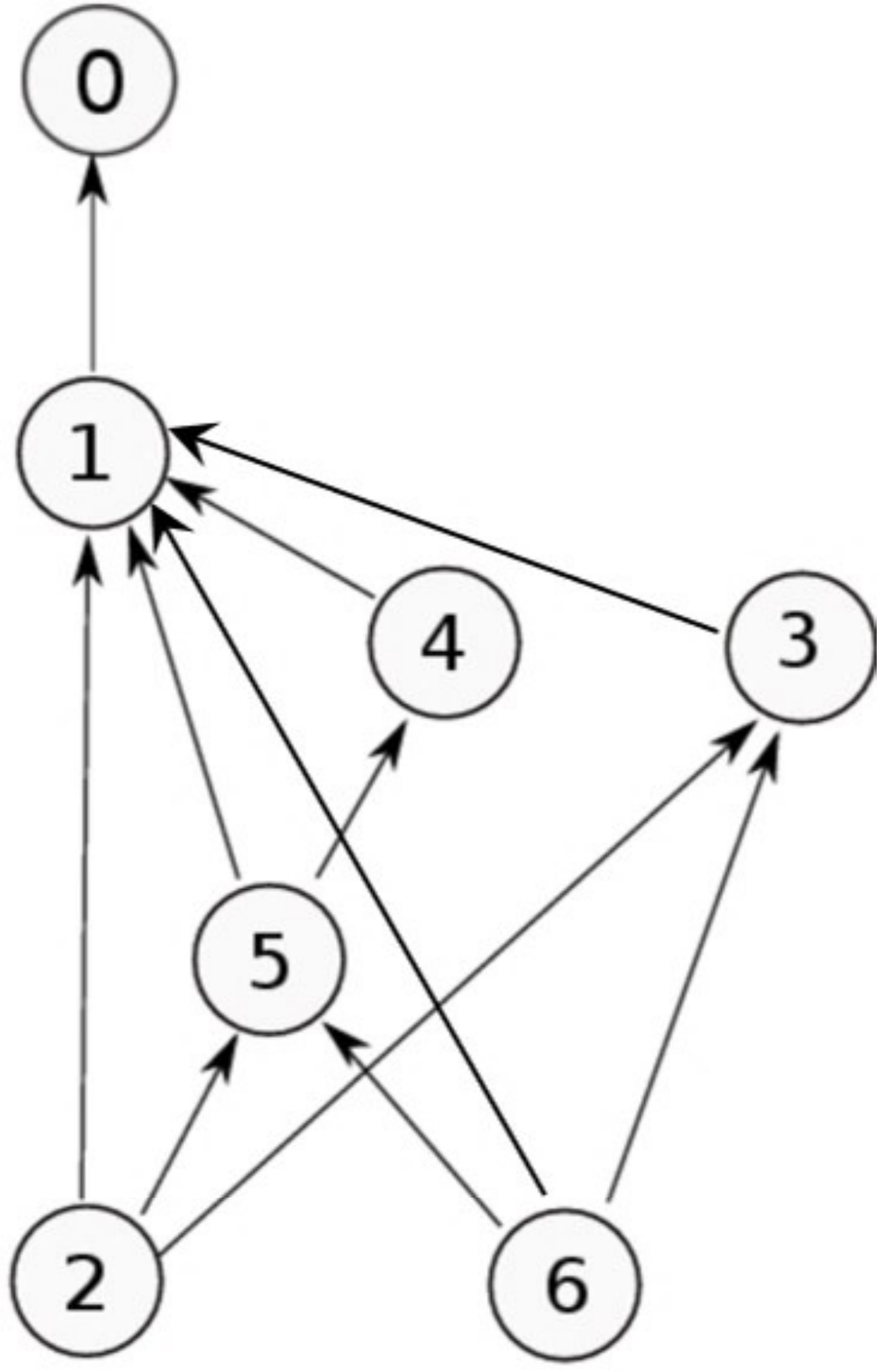}
\caption{$k\rightarrow {\mathcal F}_k^h \subset\{0,1,\cdots 6\}$ has {\bf{no cycle!}} Choose $\kappa_k$ in a certain order
($\kappa_0\gg\kappa_1\gg\kappa_3, \kappa_4\gg \kappa_5\gg \kappa_2,
\kappa_6$) to prove \eqref{est on w and z}. }
   \label{f:wa36}
\end{figure}

Since there is no cycle for the relation tree ${\mathcal F}_k^h$,
we can choose a suitable small constant $\delta>0$, with the weighted norm defined by
\begin{equation}
\|(w,\hat{r},z,\hat{s})\|_{(u,R,S)}:=I_0+\delta I_1+\delta^4 I_2+\delta^2 I_3+\delta^2 I_4+\delta^3 I_5+\delta^4 I_6,
\end{equation}
such that the desired estimate \eqref{est on w and z} holds.
This completes the proof of Lemma \ref{lem_est}.

\end{proof}

\section{Generic regularity of conservative solutions}\label{gen_sec}

For any path $\theta\mapsto u^\theta, \theta\in[0,1]$ of smooth solutions to \eqref{vwl}, Lemma \ref{lem_est} has provided a key estimate on the growth of its weighted length. However, smooth solutions do not always remain smooth for all time. In fact, the gradient of the solution can blow up in finite time, and therefore for any two solutions there may be no regular enough solution path to connect them, as shown in Fig. \ref{homotopy figure} (b).

In another word, we need to first prove the existence of regular enough transport path, so the norm \eqref{Finsler v} can be well defined.
Secondly, even if the norm is well defined, it is still not obvious that the estimate in Lemma \ref{lem_est} holds even only after finitely many singularities.
We will cope with these two issues in this and next sections, respectively.

Aim of this section is to study generic singularities to (\ref{vwl})--(\ref{ID}) and thus prove Theorem \ref{thm_reg} by an application of the Thom's Transversality theorem. Furthermore, for any two generic solutions, we show that there exist a family of regular solutions connecting them.

We divide this section into four parts.
In the first part, we review the semi-linear system used in the construction of conservative solutions to \eqref{vwl}--\eqref{ID} in \cite{H}. In the second part, we define the generic singularities. In the third part, we construct several families of perturbations of a given solution to the semi-linear system. The proof of Theorem \ref{thm_reg} is completed in the last part.

\subsection{The semi-linear system on new coordinates}
As a start, we briefly review the semi-linear system introduced in \cite{H}, which will be used in both this and next sections. Please find detail calculations and derivations in \cite{H}. Consider the equations for the forward and backward characteristics as follows
\begin{equation*}
\begin{cases}
\frac{d}{ds}x^\pm(s;x,t)=\lambda_\pm(x^\pm(s;x,t),u(s;x^\pm(s;x,t))),\\
x^\pm|_{s=t}=x,
\end{cases}
\end{equation*}
where $\lambda_\pm$ are defined in \eqref{lambda}. Then introduce a  new coordinate transformation $(x,t)\to (X,Y)$ as
\begin{equation*}
X~:=~\int_0^{x^-(0;x,t)}[1+R^2(y,0)]\,dy, \quad \text{and }\quad Y~:=~\int^0_{x^+(0;x,t)}[1+S^2(y,0)]\,dy,
\end{equation*}
which implies that
\begin{equation}\label{coor}
\alpha(x,u)X_t+c_1(x,u)X_x=0,\quad \alpha(x,u) Y_t+c_2(x,u)Y_x=0.\end{equation}
Thus, for any smooth function $f$, we have
\begin{equation}\label{2.15}
\begin{cases}
\alpha(x,u) f_t+c_2(x,u)f_x=(\alpha X_t+c_2 X_x)f_X=(c_2-c_1)X_x f_X,\\
\alpha(x,u) f_t+c_1(x,u)f_x=(\alpha Y_t+c_1Y_x)f_Y=(c_1-c_2)Y_x f_Y.
\end{cases}
\end{equation}

For convenience to deal with possibly unbounded values of $R$ and $S$, we introduce a new set of dependent variables
\begin{equation*}\label{2.16}
\begin{split}
&\ell:=\frac{ R}{1+ R^2},\quad   h:=\frac{1}{1+R^2}, \quad p:=\frac{1+R^2}{X_x},\\
&m:=\frac{ S}{1+ S^2}, \quad g:=\frac{1}{1+ S^2}, \quad q:=\frac{1+ S^2}{-Y_x}.
\end{split}\end{equation*}
Making use of  \eqref{R-S-eqn}, \eqref{coor} and the above definitions, one obtains a semi-linear hyperbolic system with smooth coefficients for the variables $\ell,m,h,g,p,q,u,x$ in $(X,Y)$ coordinates, c.f. \cite{H}
\begin{equation}\label{semi1}
\begin{cases}
 \displaystyle\ell_Y=\frac{q(2h-1)}{c_2-c_1}[a_1g+a_2h-(a_1+a_2)(g h+m\ell)+c_2 bhm-d_1g\ell],\quad\quad\quad\quad\quad\quad\quad\quad\\
 \displaystyle m_X =\frac{p(2g-1)}{c_2-c_1}[-a_1g-a_2h+(a_1+a_2)(g h+m\ell)+c_1 bg\ell-d_2hm],\\
 \end{cases}
\end{equation}
\begin{equation}\label{semi2}
\begin{cases} \displaystyle h_Y=-\frac{2q\ell}{c_2-c_1}[a_1g+a_2h-(a_1+a_2)(gh+m\ell)
+c_2bhm-d_1g\ell],\quad\quad\quad\quad\quad\quad\quad\quad\\
\displaystyle
 g_X =-\frac{2pm}{c_2-c_1}[-a_1g-a_2h+(a_1+a_2)(g h+m\ell)+c_1 bg\ell-d_2hm],\\
\end{cases}
\end{equation}
\begin{equation}\label{semi3}
\begin{cases} \displaystyle p_Y=\frac{2pq}{c_2-c_1}[a_2(\ell-m)+(a_1+a_2)(hm-g\ell)+c_2 bm\ell+d_1gh+\frac{c_1\partial_x c_2-c_2\partial_x c_1}{2(c_2-c_1)}g],\\
  \displaystyle q_X=\frac{2pq}{c_2-c_1}[a_1(\ell-m)+(a_1+a_2)(hm-g\ell)+c_1
bm\ell+d_2gh+\frac{c_1\partial_x c_2-c_2\partial_x c_1}{2(c_2-c_1)}h],\\
\end{cases}
\end{equation}
\begin{equation}\label{semi4}
\begin{cases}\displaystyle u_X=\frac{p\ell}{c_2-c_1},\quad({\text or }\quad u_Y=\frac{qm}{c_2-c_1}),\qquad\qquad\qquad\qquad\qquad\qquad\qquad\qquad\qquad\qquad\\
\displaystyle x_X=\frac{c_2}{c_2-c_1}ph,\quad({\text or } \quad x_Y=\frac{c_1}{c_2-c_1}qg).
\end{cases}
\end{equation}
Setting $f=t$ in the \eqref{2.15}, we obtain the equations for $t$,
\begin{equation}\label{2.18}
t_X=\frac{\alpha ph}{c_2-c_1},\quad t_Y=\frac{\alpha qg}{c_2-c_1}.
\end{equation}
The system \eqref{semi1}--\eqref{2.18} must now be supplemented by non-characteristic boundary conditions, corresponding to the initial data \eqref{ID}. Toward this goal,
along the curve $$\gamma_0:=\{(X,Y);~X+Y=0\}\subset\mathbb{R}^2$$ parameterized by $x\mapsto(\bar{X}(x),\bar{Y}(x))~:=~(x,-x)$, we assign the boundary data $(\bar{u},\bar{\ell}, \bar{m},\bar{h}, \bar{g},\bar{p},\bar{q})$ by setting
\begin{equation}\label{2.19}
\begin{split}
&\bar{u}=u_0(x), \quad\bar{h}=\frac{1}{1+R^2(x,0)}, \quad \bar{g}=\frac{1}{1+ S^2(x,0)},\\&\bar{\ell}=R(x,0)\bar{h}, \quad\bar{ m}= S(x,0)\bar{g},
\quad \bar{p}=1+ R^2(x,0),  \quad\bar{q}=1+S^2(x,0),
\end{split}\end{equation}
with
\begin{equation*}
\begin{split}
 R(x,0)=\alpha\big(x,u_0(x)\big) u_1(x)+c_2\big(x,u_0(x)\big)u_{x,0}(x), \\
S(x,0)=\alpha\big(x,u_0(x)\big) u_1(x)+c_1\big(x,u_0(x)\big)u_{x,0}(x).
\end{split}\end{equation*}
Obviously, the coordinate transformation $\mathcal{F}:(X,Y)\mapsto (x,t)$ maps the point $(x,-x)\in\gamma_0$ to the point $(x,0)$, for every $x\in\mathbb{R}$.

For future reference, we state the following result of the above construction in \cite{CCDS, H}.
\begin{Lemma} \label{Lemma 2.1}
Let $(u, \ell, m, h, g, p, q, x,t)$ be a smooth solution to the system \eqref{semi1}--\eqref{2.18} with $p,q>0$. Then the function $u=u(x,t)$ whose graph is
\begin{equation}\label{2.20}
\big\{(x(X,Y),t(X,Y),u(X,Y));~(X,Y)\in \mathbb{R}^2\}
\end{equation}
provide the unique conservative solution to the variational wave equation \eqref{vwl}--\eqref{con}.
\end{Lemma}

As a preliminary, we examine the boundary data should satisfy some compatibility conditions. Instead of \eqref{2.19}, we can assign a more general boundary data for \eqref{semi1}--\eqref{2.18}, along a line $\gamma_\kappa=\{(X,Y); ~X+Y=\kappa\}$, say
\begin{equation}\label{2.21}
  u(s,\kappa-s)=\bar{u}(s), ~~
  \left\{
   \begin{array}{c}
  \ell(s,\kappa-s)=\bar{\ell}(s), \\
  m(s,\kappa-s)=\bar{m}(s),  \\
   \end{array}
  \right.~~
   \left\{
   \begin{array}{c}
 h(s,\kappa-s)=\bar{h}(s), \\
 g(s,\kappa-s)=\bar{g}(s),  \\
   \end{array}
    \right.~~
      \left\{
   \begin{array}{c}
p(s,\kappa-s)=\bar{p}(s), \\
q(s,\kappa-s)=\bar{q}(s),  \\
   \end{array}
  \right.\end{equation}
and
\begin{equation}\label{bcon}
x(s,\kappa-s)=\bar{x}(s),\quad t(s,\kappa-s)=\bar{t}(s).
\end{equation}
 If both equations in $\eqref{semi4}_1$ hold, then the boundary data should satisfy the compatibility condition
 \begin{equation}\label{2.22}
 \begin{split}
 \frac{d}{ds}\bar{u}(s)&=\frac{d}{ds}u(s,\kappa-s)=(u_X-u_Y)(s,\kappa-s)\\
 &=\frac{\bar{p}(s)\bar{\ell}(s)}{\bar{c}_2-\bar{c}_1}
 -\frac{\bar{q}(s)\bar{m}(s)}{\bar{c}_2-\bar{c}_1}.
\end{split}\end{equation}
 Moreover, according to $\eqref{semi4}_2$ and \eqref{2.18}, the following compatibility conditions is also be required
 \begin{equation}\label{2.28}
 \frac{d}{ds}\bar{x}(s)=\frac{d}{ds}x(s,\kappa-s)=\frac{\bar{c}_2\bar{p}(s)\bar{h}(s)
 -\bar{c}_1\bar{q}(s)\bar{g}(s)}
 {\bar{c}_2-\bar{c}_1},
 \end{equation}
  \begin{equation}\label{2.29}
 \frac{d}{ds}\bar{t}(s)=\frac{d}{ds}t(s,\kappa-s)=\frac{\bar{p}(s)\bar{h}(s)
 -\bar{q}(s)\bar{g}(s)}{\bar{c}_2-\bar{c}_1}\alpha\big(\bar{x}(s),\bar{u}(s)\big).
 \end{equation}
here, we have denoted
\begin{equation}\label{barc}
\bar{c}_1:=c_1\big(\bar{x}(s),\bar{u}(s)\big) \quad{\rm and} \quad \bar{c}_2:=c_2\big(\bar{x}(s),\bar{u}(s)\big).
\end{equation}

We take the following lemma as the starting point for our analysis.

 \begin{Lemma}\label{Lemma 2.2}
{\rm(i)} Let $(u, \ell,m, h, g, p, q)(X,Y)$ be smooth solutions of the system \eqref{semi1}--\eqref{semi4} with the boundary conditions \eqref{2.21} along the line $\gamma=\{(X,Y); ~X+Y=\kappa\}$. Assume that the compatibility condition \eqref{2.22} is satisfied.  Then, for any $(X,Y)\in\mathbb{R}^2$, it holds that \begin{equation}\label{uY}
u_Y=\frac{qm}{c_2-c_1}.
\end{equation}
if and only if
\begin{equation}\label{uX}
u_X=\frac{p\ell}{c_2-c_1},
\end{equation}

{\rm (ii)} Let $(u, \ell,m, h, g, p, q)(X,Y)$ be smooth solutions of the system \eqref{semi1}--\eqref{semi4}. Then there exists a solution $(t,x)(X,Y)$ of $\eqref{semi4}_2$--\eqref{2.18} with the boundary data \eqref{bcon} if and only if the compatibility conditions \eqref{2.28}--\eqref{2.29} are satisfied.
 \end{Lemma}

 \begin{proof}

 (i). By a direct calculation, we observe that
\begin{equation*}
\begin{split}
\big(\frac{qm}{c_2-c_1}\big)_X&=\frac{pq}{(c_2-c_1)^2}\Big\{a_1(2m\ell-g)
-(a_1+a_2)(m\ell-gh)+a_2(h-2gh)\\
&\quad-\frac{\partial_u (c_2-c_1)}{c_2-c_1}m\ell+c_1bg\ell+(d_2-\partial_x c_2)mh\Big\}\\
&=\frac{pq}{(c_2-c_1)^2}\Big\{a_1g(h-1)
+a_2h(1-g)-\frac{\partial_u [\alpha(c_2-c_1)]}{2\alpha(c_2-c_1)}m\ell+b(c_1g\ell+c_2mh)\Big\},
\end{split}
\end{equation*}
and
\begin{equation*}
\begin{split}
\big(\frac{p\ell}{c_2-c_1}\big)_Y&=\frac{pq}{(c_2-c_1)^2}\Big\{a_1(2gh-g)
+(a_1+a_2)(m\ell-gh)+a_2(h-2m\ell)\\
&\quad-\frac{\partial_u (c_2-c_1)}{c_2-c_1}m\ell+c_2bmh+(d_1-\partial_x c_1)g\ell\Big\}\\
&=\frac{pq}{(c_2-c_1)^2}\Big\{a_1g(h-1)
+a_2h(1-g)-\frac{\partial_u [\alpha(c_2-c_1)]}{2\alpha(c_2-c_1)}m\ell+b(c_1g\ell+c_2mh)\Big\},
\end{split}
\end{equation*}
which leads to
\begin{equation}\label{uXY}
\big(\frac{qm}{c_2-c_1}\big)_X=\big(\frac{p\ell}{c_2-c_1}\big)_Y.
\end{equation}
Assume that \eqref{uY} holds, it follows from \eqref{uXY} that
\begin{equation*}
u_{YX}=\big(\frac{qm}{c_2-c_1}\big)_X=\big(\frac{p\ell}{c_2-c_1}\big)_Y.
\end{equation*}
This together with the boundary condition\eqref{2.21}, compatibility condition \eqref{2.22} and the assumption \eqref{uY} gives that
\begin{equation*}
\begin{split}
u_{X}(X,Y)&=u_X(X,\kappa-X)+\int_{\kappa-X}^Y\big(\frac{p\ell}{c_2-c_1}\big)_Y(X,s)\,ds\\
&=[u_X-u_Y](X,\kappa-X)+u_Y(X,\kappa-X)+\frac{p\ell}{c_2-c_1}(X,Y)-\frac{p\ell}{c_2-c_1}(X,\kappa-X)\\
&=\frac{p\ell}{c_2-c_1}(X,Y),
\end{split}\end{equation*}
which is indeed the desired identity \eqref{uX}. Similar arguments yield the converse implication.

(ii). We omit the proof here for brevity, since a similar approach of this result can be found in \cite{BC}.
\end{proof}
\subsection{Three types of generic singularities}
We observe that, for smooth initial data, the solution of the semilinear system \eqref{semi1}--\eqref{2.18} remains smooth on the $X$-$Y$
plane. However, the solution $u(x,t)$ of system \eqref{vwl}--\eqref{con} can have singularities. This happens precisely at points where the Jacobian matrix $D\mathcal{F}$ is not invertible. In fact, the determinant of its Jacobian matrix is calculated as
\begin{equation}
det\left(
  \begin{array}{cc}
   x_X & x_Y \\
    t_X & t_Y \\
  \end{array}
\right)   =
\frac{\alpha pqgh}{c_2-c_1}.
\end{equation}
We recall that $p,q$ remain uniformly positive
and uniformly bounded on compact subsets of the $X$-$Y$ plane. Hence, at a point $(X_0, Y_0)$ where $g(X_0, Y_0)\neq 0$ and $h(X_0, Y_0)\neq 0$, this coordinate transformation is invertible, having a strictly positive determinant. We thus can conclude that the solution $u(x,t)$ of system \eqref{vwl}--\eqref{con} is smooth on a neighborhood of
the point $\big(x(X_0, Y_0),t(X_0, Y_0)\big)$. To analyze the set of points $(x,t)$ where $u$ is singular, we thus need to study in
more details of the points where $g=0$ or $h=0$.
It is natural to distinguish three generic types of singularities:
\begin{itemize}
\item[i.] Points where $h=0$ but $\ell_X\neq0$ and $g\neq 0$ (or else, where $g=0$ but $m_Y\neq0$ and $h\neq 0$), their images under the map $\mathcal{F}:(X,Y)\mapsto \big(x(X,Y),t(X,Y)\big)$ yield a family of characteristic curves in the $x$-$t$ plane where solution $u(x,t)$ is singular (Fig. \ref{sinpoints}, black curves, inner points of singular curves).

\item[ii.] Points where $h=0$ and $\ell_X=0$ but $\ell_{XX}\neq 0$ (or else, $g=0$ and $m_Y=0$ but $m_{YY}\neq 0$), their images in the $x$-$t$ plane are point where singular curves start or end: \quad (Fig. \ref{sinpoints}, red dots, initial and terminal points of singular curves).

\item[iii.] Points where $h=0$ and $g=0$, their images in the $x$-$t$ plane are points where two singular curves cross:\quad{(Fig. \ref{sinpoints}, blue dot, intersection of singular curves in two directions)}.
\end{itemize}
\bigskip
\begin{figure}[htbp]
   \centering
\includegraphics[width=0.7\textwidth]{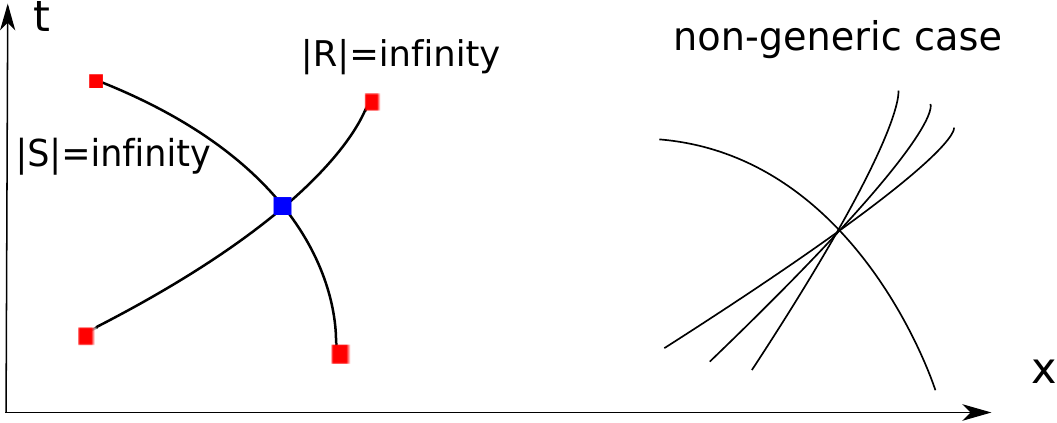}
\caption{The singular point in a solution $u(t,x)$.}
\label{sinpoints}
\end{figure}



More precisely, we give the following definition.

\begin{Definition}\label{def_gensin}
We say that a solution $u=u(x,t)$ of \eqref{vwl} has only {\bf generic singularities} for $t\in[0,T]$ if it admits a representation of the form \eqref{2.20}, where

{\rm (i)} the functions $(u, \ell, m, h, g, p, q, x,t)(X,Y)$ are $\mathcal{C}^\infty$,

{\rm (ii)} the following generic conditions
\begin{equation}\label{generic_con}
\begin{cases}
 h=0, \ell_X=0 \Longrightarrow \ell_Y\neq 0,\ell_{XX}\neq 0,\\
g=0, m_Y=0 \Longrightarrow m_X\neq 0,m_{YY}\neq 0,\\
 h=0, g=0 \Longrightarrow \ell_X\neq0,  m_Y\neq 0,\\
\end{cases}\end{equation}
hold for $t(X,Y)\in[0,T]$.
\end{Definition}

\subsection{Families of perturbed solutions}
Now we construct families of smooth solutions to the semi-linear system of \eqref{semi1}--\eqref{semi4}, depending on parameters.
Let a point $(X_0,Y_0)$ be given, and consider the line
\begin{equation*}\label{3.1}
\gamma_\kappa=\{(X,Y);~X+Y=\kappa\}, \quad \kappa~\:=~X_0+Y_0.
\end{equation*}
The main result of this part reads
\begin{Lemma}\label{Lemma 3.1}
 Assume the generic condition \eqref{gencon} holds. Let a point $(X_0,Y_0)\in \mathbb{R}^2$ be given, and $(u, \ell, m, h, g, p, q,x)$ be a smooth solution of the semi-linear system \eqref{semi1}--\eqref{semi4}.

{\bf(1)} If $(h, \ell_X,\ell_{XX})(X_0,Y_0)=(0, 0, 0)$, then there exists a 3-parameter family of smooth solutions $(u^\vartheta,\ell^\vartheta,  m^\vartheta, h^\vartheta, g^\vartheta, p^\vartheta, q^\vartheta,x^\vartheta)$ of \eqref{semi1}--\eqref{semi4}, depending smoothly on $\vartheta\in \mathbb{R}^{3}$, such that the following holds.

{\rm (i)} When $\vartheta= 0\in\mathbb R^{3}$, one recovers the original solution, namely $(u^0,\ell^0, m^0, h^0, g^0, p^0, q^0,x^0)$ $=( u, \ell, m, h, g, p, q,x)$.

{\rm(ii)} At a point $(X_0,Y_0)$, when $\vartheta=0$ one has
\begin{equation}\label{3.2}
\text{rank }D_\vartheta(h^\vartheta, \ell_X^\vartheta, \ell^\vartheta_{XX})~=~3.
\end{equation}

{\bf(2)} If $(h, g, \ell_X)(X_0,Y_0)=(0,0, 0)$, then there exists a 3-parameter family of smooth solutions $(u^\vartheta,\ell^\vartheta,  m^\vartheta, h^\vartheta, g^\vartheta, p^\vartheta, q^\vartheta,x^\vartheta)$ satisfying {\rm(i)--(ii)} as above, with \eqref{3.2} replaced by
\begin{equation}\label{3.3}
\text{rank }D_\vartheta(h^\vartheta, g^\vartheta,\ell^\vartheta_X)~=~3.
\end{equation}

{\bf(3)} If $\big(h, \ell_X,\partial_u\lambda_-(x,u)\big)(X_0,Y_0)=(0,0, 0)$, then there exists a 3-parameter family of smooth solutions $(u^\vartheta,\ell^\vartheta,  m^\vartheta, h^\vartheta, g^\vartheta, p^\vartheta, q^\vartheta,x^\vartheta)$ satisfying {\rm(i)--(ii)} as above, with \eqref{3.2} replaced by
\begin{equation}\label{3.4}
\text{rank }D_\vartheta\big(h^\vartheta,\ell^\vartheta_X,\partial_u\lambda_-(x^\vartheta,u^\vartheta)\big)~=~3.
\end{equation}
\end{Lemma}
\begin{Remark}
Three groups of functions considered in \eqref{3.2}-\eqref{3.4} are corresponding to conditions on three types of generic singularities in Definition \ref{def_gensin}.
\end{Remark}
\begin{proof}[\bf Proof]
Let $(u, \ell, m, h, g, p, q,x)$ be a smooth solution to the semi-linear system \eqref{semi1}--\eqref{semi4}, and $(\bar{u},\bar{\ell},\bar{ m},\bar{h},\bar{g},\bar{p},\bar{q},\bar{x})(s)$ be the values along a line $\gamma_\kappa$ as in \eqref{2.21}.
The main goal of this lemma is to consider families solution $(\bar{u}^\vartheta,\bar{\ell}^\vartheta,\bar{ m}^\vartheta,\bar{h}^\vartheta,\bar{g}^\vartheta,\bar{p}^\vartheta,
\bar{q}^\vartheta,\bar{x}^\vartheta)$ of \eqref{semi1}--\eqref{semi4} with perturbations on the data \eqref{2.21} along the curve $\gamma_\kappa$, so that the matrices in \eqref{3.2}--\eqref{3.4}, computed at $\vartheta=0$, have full rank at the given point $(X_0,Y_0)$. These perturbations will have the form
\begin{equation*}
\left\{
\begin{array}{ll}
\displaystyle \bar{\ell}^\vartheta(s)=\bar{\ell}(s)+\sum_{j=1}^3\vartheta_j L_{j}(s),\\
\displaystyle \bar{m}^\vartheta(s)=\bar{m}(s)+\sum_{j=1}^3\vartheta_j M_{j}(s),\\
\displaystyle \bar{h}^\vartheta(s)=\bar{h}(s)+\sum_{j=1}^3\vartheta_j H_{j}(s),
\end{array}
\right.
\quad
\left\{
\begin{array}{ll}
\displaystyle \bar{g}^\vartheta(s)=\bar{g}(s)+\sum_{j=1}^3\vartheta_j G_{\j}(s),\\
\displaystyle \bar{p}^\vartheta(s)=\bar{p}(s)+\sum_{j=1}^3\vartheta_j P_{j}(s),\\
\displaystyle \bar{q}^\vartheta(s)=\bar{q}(s)+\sum_{j=1}^3\vartheta_j Q_{j}(s),\\
\end{array}
\right.
\end{equation*}
for some suitable functions $L_{j},M_{j}, H_{j}, G_{j}, P_j, Q_j\in \mathcal{C}^\infty_c(\mathbb{R})$. Moreover, at point $s=X_0$, we set
\begin{equation*}
\displaystyle \bar{u}^\vartheta(X_0)=\bar{u}(X_0)+\sum_{j=1}^3\vartheta_j U_{j}(X_0),\quad \bar{x}^\vartheta(X_0)=\bar{x}(X_0)+\sum_{j=1}^3\vartheta_j \mathcal{X}_{j}(X_0).
\end{equation*}
Notice that, with the above definitions and the compatibility conditions \eqref{2.22} and \eqref{2.28},  we can obtain the values $\bar{u}^\vartheta(s)$ and $\bar{x}^\vartheta(s)$ for all $s\in \mathbb{R}$. In addition, we can derive a unique solution of the semi-linear system \eqref{semi1}--\eqref{semi4} for each $\vartheta\in\mathbb{R}^3$.

To prove our results, we proceed with the values of $\ell_X$ and $\ell_{XX}$ at the point $(X_0,Y_0)$. To this end, we first observe that,
\begin{equation*}
\bar{z}'(s)=\frac{d}{ds}z(s,\kappa-s)=(z_X -z_Y)(s,\kappa-s),
\end{equation*}
 at any point $(s,\kappa-s)\in\gamma_\kappa$, for $z=\ell,m,h,g,q$. Here and in the rest of this manuscript, unless specified, we will use a prime to denote the derivative with respect to the parameter $s$ along the line $\gamma_\kappa$. Hence, it follows from \eqref{semi1}--\eqref{semi3} that,
\begin{equation}\label{lmeqn}
\ell_X(X_0,Y_0)
=\bar{\ell}'+\frac{\bar{q}(2\bar{h}-1)}{\bar{c}_2-\bar{c}_1}f_1,\quad
m_Y(X_0,Y_0)
=-\bar{m}'+\frac{\bar{p}(2\bar{g}-1)}{\bar{c}_2-\bar{c}_1}f_2,
\end{equation}
\begin{equation}\label{hgeqn}
h_X(X_0,Y_0)
=\bar{h}'-\frac{2\bar{q}\bar{\ell}}{\bar{c}_2-\bar{c}_1}f_1,\qquad
g_Y(X_0,Y_0)
=-\bar{g}'-\frac{2\bar{p}\bar{m}}{\bar{c}_2-\bar{c}_1}f_2,\quad\quad
\end{equation}
\begin{equation}\label{qyeqn}
q_Y(X_0,Y_0)
=-\bar{q}'+\frac{2\bar{p}\bar{q}}{\bar{c}_2-\bar{c}_1}f_3,
\quad\quad\quad\quad\quad\quad\quad\quad\qquad\qquad\qquad\qquad
\end{equation}
where the right hand sides of \eqref{lmeqn}--\eqref{qyeqn} are evaluated at $s=X_0$ and we have denoted
\begin{equation*}
f_1~:=~\bar{a}_1\bar{g}+\bar{a}_2\bar{h}-(\bar{a}_1+\bar{a}_2)(\bar{g} \bar{h}+\bar{m}\bar{\ell})+\bar{c}_2 \bar{b}\bar{h}\bar{m}-\bar{d}_1\bar{g}\bar{\ell},\quad\quad\quad\quad\quad\quad\quad\quad\quad
\end{equation*}
\begin{equation*}
f_2~:=~-\bar{a}_1\bar{g}-\bar{a}_2\bar{h}+(\bar{a}_1+\bar{a}_2)(\bar{g} \bar{h}+\bar{m}\bar{\ell})+\bar{c}_1 \bar{b}\bar{g}\bar{\ell}-\bar{d}_2\bar{h}\bar{m},\quad\quad\quad\quad\quad\quad\quad\quad
\end{equation*}
\begin{equation*}
f_3~:=~\bar{a}_1(\bar{\ell}-\bar{m})+(\bar{a}_1+\bar{a}_2)(\bar{h}\bar{m}
-\bar{g}\bar{\ell})+\bar{c}_1
\bar{b}\bar{m}\bar{\ell}+\bar{d}_2\bar{g}\bar{h}+\frac{\bar{c}_1
\partial_x \bar{c}_2-\bar{c}_2\partial_x \bar{c}_2}{2(\bar{c}_2-\bar{c}_1)}\bar{h},
\end{equation*}
with $\bar{a}_i=a_i\big(\bar{x}(s),\bar{u}(s)\big),$ $\bar{b}=b\big(\bar{x}(s),\bar{u}(s)\big)$,
$\bar{d}_i=d_i\big(\bar{x}(s),\bar{u}(s)\big)$ and $\bar{c}_i$ denoted in \eqref{barc}, for $i=1,2$.

On the other hand, a straightforward computation now shows
\begin{equation}\label{3.6}
\frac{d^2}{ds^2}\bar{\ell}(s)=\frac{d}{ds}[\ell_X (s,\kappa-s)-\ell_Y (s,\kappa-s)]=(\ell_{XX}+\ell_{YY}-2\ell_{XY})(s,\kappa-s).
\end{equation}
 By \eqref{semi1}--\eqref{semi4} and \eqref{hgeqn}--\eqref{qyeqn}, further manipulation leads to the following estimates for $\ell_{YY}$ and $\ell_{YX}$:
\begin{equation}\label{3.7}
\begin{split}
\ell_{YY}(X_0,Y_0)=&-\frac{\bar{q}(2\bar{h}-1)}{(\bar{c}_2-\bar{c}_1)^2}
\Big[\partial_u(\bar{c}_2-\bar{c}_1)\bar{u}_Y+
\partial_x(\bar{c}_2-\bar{c}_1)\bar{x}_Y\Big]f_1\\
&+\frac{(2\bar{h}-1)f_1}{\bar{c}_2-\bar{c}_1}\bar{q}_Y
+\frac{2\bar{q} f_1}{\bar{c}_2-\bar{c}_1}\bar{h}_Y+\frac{\bar{q}(2\bar{h}-1)}{\bar{c}_2-\bar{c}_1}\partial_Y f_1\\
=&\frac{\bar{q}(2\bar{h}-1)}{(\bar{c}_2-\bar{c}_1)^2}
\big[\frac{\partial_u\bar{c}_1-\partial_u\bar{c}_2}{\bar{c}_2-\bar{c}_1}\bar{q}\bar{m}
+\frac{\partial_x\bar{c}_1-\partial_x\bar{c}_2}{\bar{c}_2-\bar{c}_1}\bar{c}_1\bar{q}\bar{m}\big]f_1\\
&+\frac{(2\bar{h}-1)f_1}{\bar{c}_2-\bar{c}_1}[-\bar{q}'+\frac{2\bar{p}\bar{q}}{\bar{c}_2-\bar{c}_1}f_3]
-\frac{4\bar{q}^2\bar{\ell} f_1^2}{(\bar{c}_2-\bar{c}_1)^2}+\frac{\bar{q}(2\bar{h}-1)}{\bar{c}_2-\bar{c}_1}\partial_Y f_1=:F_1,
\end{split}
\end{equation}
and
\begin{equation}\label{3.8}
\begin{split}
\ell_{YX}(X_0,Y_0)=&-\frac{\bar{q}(2\bar{h}-1)}{(\bar{c}_2-\bar{c}_1)^2}
\Big[\partial_u(\bar{c}_2-\bar{c}_1)\bar{u}_X+
\partial_x(\bar{c}_2-\bar{c}_1)\bar{x}_X\Big]f_1\\
&+\frac{(2\bar{h}-1)f_1}{\bar{c}_2-\bar{c}_1}\bar{q}_X
+\frac{2\bar{q} f_1}{\bar{c}_2-\bar{c}_1}\bar{h}_X+\frac{\bar{q}(2\bar{h}-1)}{\bar{c}_2-\bar{c}_1}\partial_X f_1\\
=&\frac{\bar{q}(2\bar{h}-1)}{(\bar{c}_2-\bar{c}_1)^2}
\big[\frac{\partial_u\bar{c}_1-\partial_u\bar{c}_2}{\bar{c}_2-\bar{c}_1}\bar{p}\bar{\ell}
+\frac{\partial_x\bar{c}_1-\partial_x\bar{c}_2}{\bar{c}_2-\bar{c}_1}\bar{c}_2\bar{p}\bar{h}\big]f_1\\
&+\frac{2(2\bar{h}-1)\bar{p}\bar{q}f_1f_3}{(\bar{c}_2-\bar{c}_1)^2}
+\frac{2\bar{q} f_1}{\bar{c}_2-\bar{c}_1}[\bar{h}'-\frac{2\bar{q}\bar{\ell}}{\bar{c}_2-\bar{c}_1}f_1]+\frac{\bar{q}(2\bar{h}-1)}{\bar{c}_2-\bar{c}_1}\partial_X f_1=:F_2,
\end{split}
\end{equation}
with
$
\partial_r f_1=\partial_r\big[\bar{a}_1\bar{g}+\bar{a}_2\bar{h}-(\bar{a}_1+\bar{a}_2)(\bar{g} \bar{h}+\bar{m}\bar{\ell})+\bar{c}_2 \bar{b}\bar{h}\bar{m}-\bar{d}_1\bar{g}\bar{\ell}\big]
$, for $r=X$ or $Y$.
The equations \eqref{3.6}--\eqref{3.8} in turn yield
\begin{equation}\label{3.9}
\ell_{XX}=\bar{\ell}''-F_{1}+2F_{2}.
\end{equation}

Now, we are ready to construct families of perturbed solutions satisfying \eqref{3.2}--\eqref{3.4}.

{\bf (1).} We choose suitable perturbations $(\bar{u}^\vartheta,\bar{\ell }^\vartheta,\bar{ m}^\vartheta,\bar{h}^\vartheta,\bar{g}^\vartheta,\bar{p}^\vartheta,\bar{q}^\vartheta,\bar{x}^\vartheta)$, such that, at the point $s=X_0$ and $\vartheta= 0$, it holds that
\begin{equation*}
 D_\vartheta\left(
\begin{array}{c}
   \bar{u}\\
      \bar{h}\\
 \bar{g}\\
    \bar{\ell}\\
    \bar{m}\\
      \bar{p}\\
 \bar{q}\\
 \bar{x}\\
 \bar{g}'\\
 \bar{h}'\\
    \bar{m}'\\
      \bar{\ell}'\\
 \bar{\ell}''\\
 \bar{q}'\\
  \end{array}
\right) =\left(
\begin{array}{ccc}
  0&0&0\\
  1&0&0\\
 0&0&0\\
 0&0&0\\
   0&0&0\\
 0&0&0\\
 0&0&0\\
0&0&0\\
 0&0&0\\
0&0&0\\
 0&0&0\\
 0&1&0\\
  0&0&1\\
    0&0&0\\
  \end{array}
\right) .
\end{equation*}
Hence, by using \eqref{lmeqn} and \eqref{3.9}, we obtain the desired Jacobian matrix at the point $(X_0,Y_0)$,
 \begin{equation*}
 D_\vartheta\left(
\begin{array}{c}
h\\
     \ell_X\\
        \ell_{XX}\\
  \end{array}
\right) =\left( \begin{array}{ccccccc}
   1&0&0\\
  *&1&0\\
  *&*&1\\
  \end{array}
\right).
\end{equation*}
This in turn yields \eqref{3.2}.

{\bf (2).} We choose suitable perturbations $(\bar{u}^\vartheta,\bar{\ell}^\vartheta,\bar{ m}^\vartheta,\bar{h}^\vartheta,\bar{g}^\vartheta,\bar{p}^\vartheta,\bar{q}^\vartheta,\bar{x}^\vartheta)$, such that, at the point $s=X_0$ and $\vartheta= 0$, the Jacobian matrix of first order derivatives with respect to  $\vartheta$ is given by
 \begin{equation*}
 D_\vartheta\left(
\begin{array}{c}
   \bar{u}\\
 \bar{h}\\
    \bar{g}\\
      \bar{\ell}\\
 \bar{m}\\
    \bar{p}\\
      \bar{q}\\
      \bar{x}\\
      \bar{\ell}'\\
  \end{array}
\right) =\left(
\begin{array}{ccc}
  0& 0&0\\
  1& 0&0\\
  0& 1&0\\
  0& 0&0\\
  0&0&0\\
  0& 0&0\\
  0& 0&0\\
  0& 0&0\\
  0& 0&1\\
  \end{array}
\right) .
\end{equation*}
This together with \eqref{lmeqn} implies that at the point $(X_0,Y_0)$,
 \begin{equation*}
 D_\vartheta\left(
\begin{array}{c}
h\\
g\\
    \ell_X \\
  \end{array}
\right) =\left( \begin{array}{ccc}
   1&0&0\\
  0&1&0\\
  *&*&1\\
  \end{array}
\right).
\end{equation*}
We thus conclude this matrix has full rank, that is, \eqref{3.3} holds.

{\bf (3).} If $(h,\partial_u\lambda_-, \ell_X)(X_0,Y_0)=(0,0,0)$ is satisfied, the generic condition \eqref{gencon} gives that
\begin{equation}\label{3.10}
\partial_{uu}\lambda_-(X_0,Y_0)\neq0 \quad{\rm or}\quad \partial_{ux}\lambda_-(X_0,Y_0)\neq0.
\end{equation}
By choosing suitable perturbations $(\bar{u}^\vartheta,\bar{ \ell}^\vartheta,\bar{ m}^\vartheta,\bar{h}^\vartheta,\bar{g}^\vartheta,\bar{p}^\vartheta,
 \bar{q}^\vartheta,\bar{x}^\vartheta)$, such that, at the point $s=X_0$ and $\vartheta=0$, it holds that
 \begin{equation*}
 D_\vartheta\left(
\begin{array}{c}
   \bar{u}\\
 \bar{h}\\
    \bar{g}\\
      \bar{\ell}\\
 \bar{m}\\
    \bar{p}\\
      \bar{q}\\
      \bar{x}\\
      \bar{\ell}'\\
  \end{array}
\right) =\left(
\begin{array}{ccc}
   0& 1&0\\
  1& 0&0\\
  0& 0&0\\
  0& 0&0\\
  0& 0&0\\
  0&0&0\\
  0& 0&0\\
  0& 0&0\\
  0& 0&1\\
  \end{array}
\right)
\quad {\rm or}\quad
 D_\vartheta\left(
\begin{array}{c}
   \bar{u}\\
 \bar{h}\\
    \bar{g}\\
      \bar{\ell}\\
 \bar{m}\\
    \bar{p}\\
      \bar{q}\\
      \bar{x}\\
      \bar{\ell}'\\
  \end{array}
\right) =\left(
\begin{array}{ccc}
   0& 0&0\\
  1& 0&0\\
  0& 0&0\\
  0& 0&0\\
  0& 0&0\\
  0&0&0\\
  0& 0&0\\
  0& 1&0\\
  0& 0&1\\
  \end{array}
\right) .
\end{equation*}
Here the first matrix corresponds to the assumption that $\partial_{uu}\lambda_-(X_0,Y_0)\neq0$,
while the second one corresponds to $\partial_{ux}\lambda_-(X_0,Y_0)\neq0$.
In terms of this construction and \eqref{lmeqn}, one has
 \begin{equation*}
 D_\vartheta\left(
\begin{array}{c}
h\\
  \partial_u \lambda_-\\
    \ell_X\\
  \end{array}
\right) =\left( \begin{array}{ccccc}
 1&0&0\\
  *&\partial_{uu}\lambda_-&0\\
   *&*&1\\
  \end{array}
\right)
\quad {\rm or}\quad
 D_\vartheta\left(
\begin{array}{c}
h\\
  \partial_u \lambda_-\\
    \ell_X\\
  \end{array}
\right) =\left( \begin{array}{ccccc}
 1&0&0\\
  *&\partial_{ux}\lambda_-&0\\
   *&*&1\\
  \end{array}
\right),
\end{equation*}
at the point $(X_0,Y_0)$, which, in combination with \eqref{3.10} achieves \eqref{3.4}. Here the first matrix corresponds to the assumption that $\partial_{uu}\lambda_-(X_0,Y_0)\neq0$,
while the second one corresponds to $\partial_{ux}\lambda_-(X_0,Y_0)\neq0$.
This completes the proof  of Lemma \ref{Lemma 3.1}.
\end{proof}

\subsection{Proof of Theorem \ref{thm_reg}}

To recover the singularities of the solution $u=u(x,t)$ of \eqref{vwl} in the original $(x,t)$ plane, we will use Lemma \ref{Lemma 3.1} together with transversality argument (c.f.\cite{Bloom, BC, GG}) to study the smooth solutions to the semi-linear \eqref{semi1}--\eqref{2.18}, and hence determine the generic structure of the level sets $\{(X,Y);~ h(X,Y)=0\}$ and $\{(X,Y); ~g(X,Y)=0\}$.
One can prove the following lemma in a very similar method as in \cite{BC}, we omit it here for brevity.

\begin{Lemma}\label{Lemma 4.1}
Assume the generic condition \eqref{gencon} holds. Consider a compact domain of the form $$\Omega:=~\{(X,Y); ~|X|\leq M,\quad |Y|\leq M\},$$
and denote $\mathcal{S}$ be the family of all $\mathcal{C}^2$ solutions $(u, \ell, m, h, g, p, q,x)$ to the semi-linear system \eqref{semi1}--\eqref{semi4}, with $p,q>0$ for all $(X,Y)\in\mathbb{R}^2$. Moreover, denote $\mathcal{S}'\subset \mathcal{S}$ be the subfamily of all solutions $(u, \ell, m, h, g, p, q,x)$, such that for $(X ,Y)\in\Omega$, none of the following values is attained:
\begin{equation}\label{4.1}
\left\{
\begin{array}{ll}
(h, \ell_X,\ell_{XX})=(0, 0,0),\\
(g, m_Y, m_{YY})=(0, 0, 0),
\end{array}
\right.\quad
\left\{
\begin{array}{ll}
(h,g, \ell_X)=(0,0, 0),\\
(h,g,  m_Y)=(0,0, 0),
\end{array}
\right.\quad
\left\{
\begin{array}{ll}
(h,\partial_{u}\lambda_-,\ell_X)=(0,0,0),\\
(g, \partial_{u}\lambda_+,m_Y)=(0,0,0).
\end{array}
\right.\\
\end{equation}
Then $\mathcal{S}'$ is a relatively open and dense subset of  $\mathcal{S}$, in the topology induced by $\mathcal{C}^2(\Omega)$.
\end{Lemma}

For future use, we introduce a new space
\begin{equation}\label{Ndef}\mathcal{N}~:=~\Big(\mathcal{C}^3(\mathbb{R})\cap H^1(\mathbb{R})\Big)\times\Big(\mathcal{C}^2(\mathbb{R})\cap L^2(\mathbb{R})\Big),\end{equation} equipped with the norm $$\|(u_{0},u_{1})\|_{\mathcal{N}}~:=~\|u_{0}\|_{\mathcal{C}^3}+\|u_{0}\|_{H^1}+\|u_{1}\|_{\mathcal{C}^2}+\|u_{1}\|_{L^2}.$$

Applying a standard comparison argument, we deduce that, if the initial data $(u_{0},u_{1})\in\mathcal{N}$, then the corresponding solution remains smooth for all $|x|$ sufficiently large. The proof of this lemma
is similar to \cite{BC}, and we omit it here for brevity.
\begin{Lemma}\label{lem_out}
Assume $(u_{0},u_{1})\in\mathcal{N}$ and let $T>0$ be given. Then there exists $r>0$ sufficiently large so that the solution $u=u(x,t)$ of (\ref{vwl})--(\ref{ID}) remains $\mathcal{C}^2$ on the domain $\{(x,t);~t\in[0,T], ~|x|\geq r\}$.
\end{Lemma}

With the help of Lemma \ref{Lemma 4.1}, we can now prove the generic regularity of conservative solutions to \eqref{vwl}--\eqref{con} of the Theorem \ref{thm_reg}.
\begin{proof}[\bf Proof of Theorem \ref{thm_reg}]
Let the initial data $(\hat{u}_{0},\hat{u}_{1})\in\mathcal{N}$ be given and set the open ball
\[B_\delta~:=~\{(u_{0},u_{1})\in\mathcal{N};~\|(u_{0},u_{1})-(\hat{u}_{0},\hat{u}_{1})\|_{\mathcal{N}}<\delta\}.\]
To prove our main theorem, it suffices to prove that, for any $(\hat{u}_{0},\hat{u}_{1})\in\mathcal{N}$, there exists an open dense subset $\widehat{\mathcal{M}}\subset B_\delta$, such that, for every initial data $(u_0,u_1)\in\widehat{\mathcal{M}}$, the conservative solution $u=u(x,t)$ of  \eqref{vwl}--\eqref{con} is twice continuously differentiable in the complement of finitely many characteristic curves, within the domain $\mathbb{R}\times[0,T]$.
We prove this result by two steps.

{\bf (1).(Construction of an open dense set $\widehat{\mathcal{M}}$)}  Since $(\hat{u}_{0},\hat{u}_{1})\in\mathcal{N}$, in view of Lemma \ref{lem_out}, we can choose $r>0$ large enough so that the corresponding functions $\hat{R}, \hat{S}$ in \eqref{R-S} being uniformly bounded on the domain of the form $\{(x,t); t\in[0,T], |x|\geq r\}$. In particular, we can choose $\delta>0$, such that, for initial data
 $(u_{0},u_{1})\in B_\delta$, the corresponding solution $u=u(x,t)$ of \eqref{vwl} being twice continuously differentiable
on the outer domain $\{(x,t); ~t\in[0,T], |x|\geq \varrho\}$,  for some $\varrho>0$ sufficiently large. This means the singularities of $u(x,t)$ in the set $\mathbb{R}\times[0,T]$ only appear on the compact set $$\mathcal{U}~:=~[-\varrho, \varrho]\times [0,T].$$
Denote $\mathcal{F}$ be the map of $(X,Y)\mapsto \mathcal{F}(X,Y)~:=~(x(X,Y),t(X,Y))$. Then, we can easily obtain the inclusion $\mathcal{U}\subset \mathcal{F}(\Omega)$ by choosing $M$ large enough and by possibly shrinking the radius $\delta$, where $\Omega$ is a domain defined in Lemma \ref{Lemma 4.1}.

Now, we defined the subset $\widehat{\mathcal{M}}\subset B_\delta$ as follows: $(u_{0},u_{1})\in \widehat{\mathcal{M}}$ if the following items are satisfied

(I). $(u_{0},u_{1})\in B_\delta$;

(II). for any $(X,Y)$ such that $(x(X,Y),t(X,Y))\in\mathcal{U},$  the values \eqref{4.1} are never attained, here $(u, \ell,  m, h, g, p, q)$ is
the corresponding solution of \eqref{semi1}--\eqref{semi4} with boundary data \eqref{2.19}.

We claim the set $\widehat{\mathcal{M}}$ is open and dense in $B_\delta$, we omit the detailed proof here for brevity, since a similar procedure of this result can be found in \cite{BC, CCD}.

{\bf (2). ($u$ is piecewise smooth)} Now, it remains to verify that for every initial data $(u_{0},u_{1})\in\widehat{\mathcal{M}}$, the corresponding solution $u(x,t)$ of \eqref{vwl} is piecewise $\mathcal{C}^2$ on the domain $[0,T]\times\mathbb{R}$. Toward this goal, we recall that $u(x,t)$ is $\mathcal{C}^2$ on the outer domain $\{(x,t);~ t\in [0,T ], |x|\geq \varrho\}$, so in the following we just need to consider the singularities of solution $u(x,t)$ on the inner domain $\mathcal{U}$. Recall the inclusion $\mathcal{U}\subset \mathcal{F}(\Omega)$, we know that, for every point $(X_0,Y_0)\in\Omega$, there are two cases:

{\it CASE 1.} If $h(X_0,Y_0)\neq 0$ and $g(X_0,Y_0)\neq 0$, we can obtain the determinant of the Jacobian matrix
 \begin{equation*}
\text{ det } \left(
\begin{array}{cc}
  x_X& x_Y\\
 t_X& t_Y\\
  \end{array}
\right) =\frac{\alpha pq hg}{c_2-c_1}>0,\end{equation*}
by $\eqref{semi4}_2$ and \eqref{2.18}. This implies that the map $(X,Y)\mapsto (x,t)$ is locally invertible in a neighborhood of $(X_0,Y_0)$. Clearly, the solution $u(x,t)$ is $\mathcal{C}^2$ in a neighborhood of $(x(X_0,Y_0),$ $ t(X_0,Y_0))$.

{\it CASE 2.} If $h(X_0,Y_0)=0$, we can obtain $\ell=0$, immediately. In this case, we claim that either $\ell_X \neq  0$ or $\ell_Y\neq 0$. In fact, by the equation \eqref{semi1} and the definition of $a_1$ in \eqref{R-S-eqn}, at the point $(X_0,Y_0)$, we have
\begin{equation*}
\ell_Y(X_0,Y_0)=\displaystyle-\frac{q}{c_2-c_1}a_1g=\frac{\alpha q g}{2(c_2-c_1)^2}\partial_u \lambda_-.
\end{equation*}
This together with the construction of $\widehat{\mathcal{M}}$ that the values $(h,g, \ell_X)=(0,0,0) $ and
$(h,\partial_u \lambda_-,\ell_X)=(0,0,0)$ are never attained in $\Omega$, it is easy to see that $\ell_X\neq 0$ or $\ell_Y\neq  0$.

By continuity, we can choose $\eta>0,$ so that in the open neighborhood $$\Omega' :=~\{(X,Y); ~|X|<M+\eta,|Y|< M+\eta\}, $$
the values listed in \eqref{4.1} are never attained. Applying the implicit function theorem, we derive that the sets
 \[\chi^h:=~\{(X,Y)\in\Omega';~h(X,Y)=0\},\quad
 \chi^g:=\{(X,Y)\in\Omega'; ~g(X,Y)=0\}\]
 are 1--dimensional embedded manifold of class $\mathcal{C}^2$.
In particular, the set $\chi^h\cap\Omega$ has finite connected components. Indeed, assume on the contrary that there exists a sequence of points $P_1, P_2,\cdots\in\chi^h\cap \Omega$  belonging to distinct components. Then we can choose a subsequence, denote still by $P_i$, such that $P_i\to \bar{P}$ for some $\bar{P}\in \chi^h\cap\Omega$. Since $h(\bar{P})=0,$ then $(\ell_X,\ell_Y)(\bar{P})\neq (0,0)$, which together with the implicit function theorem implies that there is a neighborhood $\Gamma$ of $\bar{P}$ such that $\chi^h \cap\Gamma$ is a connected $\mathcal{C}^2$ curve. Thus, $P_i\in\chi^h \cap\Gamma$ for all $i$ large enough, providing a contradiction on the assumption
that $P_i$ belongs to distinct components.

To complete the proof, we need to study more details on the image of the singular sets $\chi^h$ and $\chi^g$, since the set for the singular points $(t,x)$ of $u$ coincides with the image of the two sets $\chi^h, \chi^g$ under the $C^2$ map $(X,Y)\mapsto \mathcal{F}(X,Y)= (x(X,Y),t(X,Y))$.

By the previous argument, there are only finite many points $P_i=(X_i,Y_i), i=1,\cdots, m$, inside set $\Omega'$, where $h=0, \ell= 0,$ and $ \ell_X= 0$. Moreover, by \eqref{4.1}, at a point $(X_0,Y_0)\in \chi^h\cap \chi^g$, we have $\ell_X\neq  0, \ell_Y= 0,  m_X= 0,  m_Y\neq  0$. Thus, the two curves ${h=0}$ and ${g=0}$ intersect perpendicularly. Therefore, there are only finitely many such intersection points $Q_\jmath=(X_\jmath',Y_\jmath'), \jmath=1,\cdots, n$, inside the compact set $\Omega$.

Moreover, the set $\chi^h \verb|\| \{P_1,\cdots,P_m,Q_1,\cdots,Q_n\}$ has finitely many connected components which intersect $\Omega$. Consider any one of these components, which is a connected curve, say $\gamma_j$, such that $h=0, \ell=0$ and $\ell_X\neq 0$ for any $(X,Y)\in \gamma_j$. Thus, for a suitable function $\varphi_j$, this curve can be expressed as $$\gamma_j=\{(X,Y):~X=\varphi_j(Y), a_j<Y<b_j\},$$
We claim that the image $\Lambda(\gamma_j)$ is a $\mathcal{C}^2$ curve in the $x$-$t$ plane. Indeed, on the open interval $(a_j,b_j)$, the differential of the map $Y\mapsto \big(x(\varphi_j(Y),Y), t(\varphi_j(Y),Y)\big)$ does not vanish. This is true, because by \eqref{2.18}, we have
\[\frac{d}{dY}t(\varphi_j(Y),Y))=t_X\varphi'_j+t_Y=0\cdot\varphi'_j+\frac{\alpha qg}{c_2-c_1}>0,\]
since $g, c_2-c_1, q>0$. As a consequence, the singular set $\mathcal{F}(\chi^h)$ is the union of the finitely points $p_i=\mathcal{F}(P_i), i=1,\cdots, m,$ $ q_\jmath=\mathcal{F}(Q_\jmath), \jmath=1,\cdots,n,$ together with finitely many $\mathcal{C}^2$-curve $\mathcal{F}(\gamma_j)$.   Obviously, the same representation is valid for the image $\mathcal{F}(\chi^g)$. This completes the proof of Theorem \ref{thm_reg}.
\end{proof}

To this end, we introduce suitable regularity condition that allows us to define the
tangent vectors in the norm \eqref{Finsler v} between any two generic solutions and hence compute its weighted length.
In what follows, we are interested not in a single generic solution, but a path of generic solutions $\theta\mapsto u^\theta,\theta\in[0,1]$.

\begin{Definition}\label{def_repath}
We say that a path of initial data $\gamma^0:\theta\mapsto (u_0^\theta,u_1^\theta)$, $\theta\in[0,1]$ is a {\bf piecewise regular path} if the following conditions hold.

{\rm (i)} There exists a continuous map $(X,Y,\theta)\mapsto (u, \ell, m, h, g, p, q, x,t)$ such that the semilinear system \eqref{semi1}--\eqref{2.18} holds for $\theta\in[0,1]$, and the function $u^\theta(x,t)$ whose graph is
\begin{equation*}
\text{ Graph }(u^\theta)=\{(x,t,u)(X,Y,\theta);~(X,Y)\in \mathbb{R}^2\}
\end{equation*}
provides the conservation solution of \eqref{vwl} with initial data $u^\theta(x,0)=u^\theta_0(x),u_t^\theta(x,0)=u_1^\theta(x)$.

{\rm (ii)} There exist finitely many values $0=\theta_0<\theta_1<\cdots<\theta_N=1$ such that the map $(X,Y,\theta)\mapsto (u, \ell, m, h, g, p, q, x,t)$ is $\mathcal{C}^\infty$ for $\theta\in(\theta_{i-1},\theta_i), i=1,\cdots, N$, and the solution $u^\theta=u^\theta(x,t)$ has only generic singularities at time $t=0$.

In addition, if for all $\theta\in[0,1]\backslash\{\theta_1,\cdots,\theta_N\}$, the solution $u^\theta$ has only generic singularities for $t\in [0,T]$, then we say that the path of solution $\gamma^t: \theta\mapsto (u^\theta,u^\theta_t)$ is {\bf piecewise regular} for $t\in[0,T]$.
\end{Definition}

Towards our goal, we state the following result, which is an application of Theorem \ref{thm_reg}.
\begin{Theorem}\label{thm_repath}
Assume the generic condition \eqref{gencon} holds.
For any fixed $T>0,$ let $\theta\mapsto( u^\theta,\ell^\theta,m^\theta,h^\theta, g^\theta,$ $p^\theta, q^\theta, x^\theta, t^\theta ), \theta\in[0,1],$ be a smooth path of solutions to the system \eqref{semi1}--\eqref{2.18}. Then there exists a sequence of paths of solutions $\theta\mapsto( u^\theta_n,\ell^\theta_n,m^\theta_n,h^\theta_n, g^\theta_n, p^\theta_n, q^\theta_n,$ $ x^\theta_n, t^\theta_n ), $ such that

{\rm (i)} For each $n\geq 1$, the path of the corresponding solution of \eqref{vwl} $\theta\mapsto u_n^\theta$ is regular for $t\in[0,T]$ in the sense of Definition \ref{def_repath}.

{\rm (ii)} For any bounded domain $\Sigma$ in the $(X$,$Y)$ space, the functions $( u^\theta_n,\ell^\theta_n,m^\theta_n,h^\theta_n, g^\theta_n, p^\theta_n, q^\theta_n,$ $  x^\theta_n, t^\theta_n )$ converge to $( u^\theta,\ell^\theta,m^\theta,h^\theta, g^\theta,p^\theta, q^\theta, x^\theta, t^\theta )$ uniformly in $\mathcal{C}^k([0,1]\times\Sigma)$, for every $k\geq 1$, as $n\to \infty.$
\end{Theorem}
 The proof of this lemma
is similar to \cite{BC}, and we omit it here for brevity.

\section{Metric for piecewise smooth solutions}\label{sec_piecewise}

In this section, we extend the Lipschitz metric for smooth solutions in Section \ref{sec_smooth} to piecewise smooth solutions with only generic singularities.

\subsection{Tangent vectors in transformed coordinates}
To begin with, we express the norm of tangent
vectors \eqref{norm1} in transformed coordinates $X$-$Y$.

Let $u(x,t)$ be a reference solution of \eqref{vwl} and $u^\varepsilon(x,t)$ be a family of perturbed solutions. In the $(X$,$Y)$ plane, denote $(u, \ell, m, h, g, p, q, x,t)$ and $( u^\varepsilon,\ell^\varepsilon,m^\varepsilon,h^\varepsilon, g^\varepsilon, p^\varepsilon, q^\varepsilon, x^\varepsilon, t^\varepsilon)$ be the corresponding smooth solutions of \eqref{semi1}--\eqref{2.18}, and moreover assume the perturbed solutions take the form
\begin{equation*}
( u^\varepsilon,\ell^\varepsilon,m^\varepsilon,h^\varepsilon, g^\varepsilon, p^\varepsilon, q^\varepsilon, x^\varepsilon, t^\varepsilon)=(u, \ell, m, h, g, p, q, x,t)+\varepsilon(U, L, M,H, G, P, Q,\mathcal{X},\mathcal{T})+o(\varepsilon).
\end{equation*}
Here we denote the curve in ($X,Y)$ plane by
\begin{equation}\label{curve}
\Gamma_\tau=\{(X,Y)\,|\,t(X,Y)=\tau\}=\{(X,Y(\tau, X));X\in\mathbb{R}\}=\{(X(\tau, Y),Y);Y\in\mathbb{R}\}\end{equation}
and the perturbed curve as
 $$ \Gamma_\tau^\varepsilon=\{(X,Y)\,|\,t^\varepsilon(X,Y)=\tau\}=\{(X,Y^\varepsilon(\tau, X));X\in\mathbb{R}\}=\{(X^\varepsilon(\tau, Y),Y);Y\in\mathbb{R}\}.$$
Notice that the coefficients of system \eqref{semi1}--\eqref{2.18} are smooth, it thus follows that the first order perturbations satisfy a linearized system and are well defined for $(X,Y)\in\mathbb{R}^2$.

Now, we are ready to derive an expression for $I_0$--$I_6$ of \eqref{norm1} in terms of $(U, L, M,H, G, P, Q,$ $\mathcal{X},\mathcal{T})$. First, we observe that
$$t^\varepsilon\big(X,Y^\varepsilon(\tau,X)\big)
=t^\varepsilon\big(X^\varepsilon(\tau,Y),Y\big)=\tau.$$
By the implicit function theorem, at $\varepsilon=0$, it holds that
\begin{equation}\label{xvar0}
\frac{\partial X^\varepsilon}{\partial\varepsilon}\Big|_{\varepsilon=0}
=-\mathcal{T}\frac{c_2-c_1}{\alpha hp},
\quad {\rm and}\quad
\frac{\partial Y^\varepsilon}{\partial\varepsilon}\Big|_{\varepsilon=0}=-\mathcal{T}\frac{c_2-c_1}{\alpha gq}.
\end{equation}

(1). The change in $x$ is
\begin{equation}\label{wXchange}
\begin{split}
w&=\displaystyle\lim_{\varepsilon\to 0}\frac{x^\varepsilon\big(X,Y^\varepsilon(\tau,X)\big)
-x\big(X,Y(\tau,X)\big)}{\varepsilon}\\
&= \mathcal{X}\big(X,Y(\tau,X)\big)+x_Y\cdot \frac{\partial Y^\varepsilon}{\partial \varepsilon}\Big|_{\varepsilon=0}
=\big(\mathcal{X}-\frac{c_1}{\alpha}\mathcal{T}\big)(X,Y(\tau,X)).
\end{split}
\end{equation}
Similarly, we obtain
\begin{equation}\label{zYchange}
\begin{split}
z&=\displaystyle\lim_{\varepsilon\to 0}\frac{x^\varepsilon\big(X^\varepsilon(\tau,Y),Y\big)
-x\big(X(\tau,Y),Y\big)}{\varepsilon}\\
&= \mathcal{X}\big(X(\tau,Y),Y\big)+x_X\cdot \frac{\partial X^\varepsilon}{\partial \varepsilon}\Big|_{\varepsilon=0}
=\big(\mathcal{X}-\frac{c_2}{\alpha}\mathcal{T}\big)(X(\tau,Y),Y).
\end{split}
\end{equation}

(2). For the change in $u$, we first observe from \eqref{semi4} and \eqref{xvar0} that,
\begin{equation*}
\begin{split}
v+u_x w&=\displaystyle\lim_{\varepsilon\to 0}\frac{u^\varepsilon\big(X,Y^\varepsilon(\tau,X)\big)
-u\big(X,Y(\tau,X)\big)}{\varepsilon}\\
&= U\big(X,Y(\tau,X)\big)+u_Y\cdot \frac{\partial Y^\varepsilon}{\partial \varepsilon}\Big|_{\varepsilon=0}
=\big(U-\frac{\mathcal{T} m}{\alpha g}\big)(X,Y(\tau,X)).
\end{split}
\end{equation*}
This together with \eqref{R-S-eqn} gives
\begin{equation}\label{uXchange}
v+\frac{Rw-sz}{c_2-c_1}=v+u_xw+\frac{w-z}{c_2-c_1}S=U(X,Y(\tau,X)).
\end{equation}

(3). In addition, we derive an expression for the terms $J_3^\pm$ in \eqref{norm1}. Using \eqref{semi1}, \eqref{semi2}  and  \eqref{xvar0}, a direct computation gives rise to
\begin{equation*}
\begin{split}
r+wR_x&=\frac{d}{d\varepsilon}\frac{\ell^\varepsilon}{h^\varepsilon}\big(X,Y^\varepsilon(\tau,X)\big)\Big|_{\varepsilon=0}
=\lim_{\varepsilon\to 0}\frac{\frac{\ell^\varepsilon}{h^\varepsilon}\big(X,Y^\varepsilon(\tau,X)\big)
-\frac{\ell}{h}\big(X,Y(\tau,X)\big)}{\varepsilon}\\
&= \frac{1}{h}\big(L+\ell_Y\cdot \frac{\partial Y^\varepsilon}{\partial \varepsilon}\Big|_{\varepsilon=0}\big)-\frac{\ell}{h^2}\big(H+h_Y\cdot \frac{\partial Y^\varepsilon}{\partial \varepsilon}\Big|_{\varepsilon=0}\big)\\
&=\frac{L}{h}-\frac{\ell H}{h^2}-\frac{\mathcal{T}}{\alpha hg}[a_1 g+a_2 h-(a_1+a_2)(gh+m\ell)+c_2bhm-d_1g\ell],
\end{split}
\end{equation*}
and
\begin{equation*}
\begin{split}
s+zS_x&=\frac{d}{d\varepsilon}\frac{m^\varepsilon}{g^\varepsilon}\big(X^\varepsilon(\tau,Y),Y\big)\Big|_{\varepsilon=0}
\\
&= \frac{1}{g}\big(M+m_X\cdot \frac{\partial X^\varepsilon}{\partial \varepsilon}\Big|_{\varepsilon=0}\big)-\frac{m}{g^2}\big(G+g_X\cdot \frac{\partial X^\varepsilon}{\partial \varepsilon}\Big|_{\varepsilon=0}\big)\\
&=\frac{M}{g}-\frac{mG}{g^2}-\frac{\mathcal{T}}{\alpha hg}[-a_1 g-a_2 h+(a_1+a_2)(gh+m\ell)+c_1bg\ell-d_2hm].
\end{split}
\end{equation*}
Thus, we obtain from \eqref{rseq} that
\begin{equation}\label{rXchange}
\hat{r}
=\frac{L}{h}-\frac{\ell H}{h^2}-\frac{\mathcal{T}}{\alpha h}(a_1-a_1h-d_1\ell),
\end{equation}
and
\begin{equation}\label{rYchange}
\hat{s}=\frac{M}{g}-\frac{mG}{g^2}-\frac{\mathcal{T}}{\alpha g}(-a_2 +a_2g-d_2m).
\end{equation}

(4). To continue, we first use the same procedure to derive the change in the base measure with density $1+R^2$  as
\begin{equation}\label{pXchange}
\frac{d}{d\varepsilon}p^\varepsilon\big(X,Y^\varepsilon(\tau,X)\big)\Big|_{\varepsilon=0}
= P+p_Y\cdot \frac{\partial Y^\varepsilon}{\partial \varepsilon}\Big|_{\varepsilon=0}=P-\frac{c_2-c_1 }{\alpha gq}\mathcal{T}p_Y.
\end{equation}
Moreover, the change in base measure with density $R^2$ can be calculated as
\begin{equation}\label{R2change}
\begin{split}
&\frac{d}{d\varepsilon}\Big(\big(p^\varepsilon(1-h^\varepsilon)\big)\big(X,Y^\varepsilon(\tau,X)\big)\Big)\Big|_{\varepsilon=0}\\
&= \big(P+p_Y\cdot \frac{\partial Y^\varepsilon}{\partial \varepsilon}\Big|_{\varepsilon=0}\big)(1-h)-p\big(H+h_Y\cdot \frac{\partial Y^\varepsilon}{\partial \varepsilon}\Big|_{\varepsilon=0}\big)\\
&=\big(P-\frac{c_2-c_1 }{\alpha gq}\mathcal{T}p_Y\big)(1-h)-p\big(H-\frac{c_2-c_1 }{\alpha gq}\mathcal{T}h_Y\big).
\end{split}\end{equation}
Finally, we can achieve the change in base measure with density 1 by subtracting \eqref{pXchange} from \eqref{R2change}
\begin{equation}\label{1change}
\begin{split}
h\big(P-\frac{c_2-c_1 }{\alpha gq}\mathcal{T}p_Y\big)+p\big(H-\frac{c_2-c_1 }{\alpha gq}\mathcal{T}h_Y\big).
\end{split}\end{equation}
Notice that
$$(1+R^2)dx=pdX,\quad (1+S^2)dx=-qdY.$$
Hence, according to \eqref{wXchange}--\eqref{1change}, the weighted norm \eqref{norm1} can be rewritten as a line integral over the line $\Gamma_\tau$ defined in \eqref{curve}, More specifically, we have
\begin{equation}\label{normXY}
\|(v,w,\hat{r},z,\hat{s})\|_{(u,R,S)}=\sum_{k=0}^6\kappa_k\int_{\Gamma_\tau}
\Big(|J_k|\mathcal{W}^-dX+|H_k|\mathcal{W}^+dY\Big),
\end{equation}
where
\begin{equation*}
\begin{split}
J_0&=\big(\mathcal{X}-\frac{c_1}{\alpha}\mathcal{T}\big)ph,\\
J_1&=\big(\mathcal{X}-\frac{c_1}{\alpha}\mathcal{T}\big)p,\\
J_2&=Up,\\
J_3&=Lp-\frac{\ell H}{h}p-\frac{\mathcal{T}p}{\alpha }(a_1-a_1h-d_1\ell),\\
J_4&=hP+pH+\frac{2p\mathcal{T}}{\alpha}\big(a_1\ell+
\frac{c_1\partial_x\alpha-\alpha\partial_x c_1}{2\alpha}h\big),\\
J_5&=\ell P+\frac{\ell H}{h}p+\frac{2p\mathcal{T}\ell }{\alpha h}\big(a_1\ell+
\frac{c_1\partial_x\alpha-\alpha\partial_x c_1}{2\alpha}h\big),\\
J_6&=(1-h)P-pH-\frac{c_1\partial_x c_2-c_2\partial_x c_1}{\alpha(c_2-c_1)}\mathcal{T}p(1-h),\\
\end{split}
\end{equation*}
and
\begin{equation*}
\begin{split}
H_0&=\big(\mathcal{X}-\frac{c_2}{\alpha}\mathcal{T}\big)qg,\\
H_1&=\big(\mathcal{X}-\frac{c_2}{\alpha}\mathcal{T}\big)q,\\
H_2&=Uq,\\
H_3&=Mq-\frac{mG}{g}q-\frac{\mathcal{T}q}{\alpha }(-a_2 +a_2g-d_2m),\\
H_4&=gQ+qG+\frac{2q\mathcal{T}}{\alpha}\big(-a_2m+
\frac{c_2\partial_x\alpha-\alpha\partial_x c_2}{2\alpha}g\big),\\
H_5&=mQ+\frac{mG}{g}q+\frac{2q\mathcal{T}m}{\alpha g}\big(-a_2m+
\frac{c_2\partial_x\alpha-\alpha\partial_x c_2}{2\alpha}g\big),\\
H_6&=(1-g)Q-qG-\frac{c_1\partial_xc_2-c_2\partial_x c_1}{\alpha(c_2-c_1)}\mathcal{T}q(1-g).\\
\end{split}
\end{equation*}
It is easy to verify that each integrands $J_k, H_k$ are smooth, for $k=0,1,2,4,6$. On the other hand, for the term $\frac{\ell H}{h}p$ in $J_3$ and $J_5$, we first observe that, $$(\ell^\varepsilon)^2+(h^\varepsilon)^2=h^\varepsilon.$$
Differentiating this equation with respect to $\varepsilon$, at $\varepsilon=0$, it holds that
$$2\ell L+2hH=H.$$
With this help, we achieve
\begin{equation*}
\frac{\ell H}{h}p=\frac{2\ell^2 L+2\ell hH}{h}p=\frac{2(h-h^2) L+2\ell hH}{h}p=2p[(1-h)L+\ell H],
\end{equation*}
here we have used the fact that $\ell^2+h^2=h.$ Therefore, $J_3$ and $J_5$ are also smooth. In a similar way, we can get the smoothness of $H_3$ and $H_5$.

\subsection{Length of piecewise regular paths}\label{sub_piecewise}
In this part, we define the length of a piecewise regular path $\gamma^t: \theta\mapsto \big(u^\theta(t),u^\theta_t(t)\big)$, and examine the appearance of the generic singularity will not impact the Lipschitz property of this metric.
\begin{Definition}\label{def_piece}
The length $\|\gamma^t\|$ of the piecewise regular path $\gamma^t: \theta\mapsto \big(u^\theta(t),u^\theta_t(t)\big)$ is defined as
\begin{equation}\label{piecedef}
    \|\gamma^t\|=\inf_{\gamma^t}\int_0^1\Big\{\sum_{k=0}^6\kappa_k\int_{\Gamma_t^\theta}
\Big(|J_k^\theta|\mathcal{W}^-dX+|H_k^\theta|\mathcal{W}^+dY\Big)\Big\}\,d\theta,
\end{equation}
where the infimum is taken over all piecewise smooth relabelings of the $X$-$Y$ coordinates and $\Gamma_\tau^\theta:=\{(X,Y);t^\theta(X,Y)=\tau\}$
\end{Definition}

\begin{Remark}
In general, there are many distinct solutions to the system \eqref{semi1}--\eqref{2.18} which yields the same solution $u=u(x,t)$ of \eqref{vwl}. In fact, suppose $\varphi,\psi: \mathbb{R}\mapsto\mathbb{R}$ be two $\mathcal{C}^2$ bijections, with $\varphi',\psi'>0$. Consider a particular solution $(u, \ell, m, h, g, p, q, x,t)$ to the system \eqref{semi1}--\eqref{2.18}, and let the new independent and dependent variables $(\widetilde{X}, \widetilde{Y})$ and  $(\tilde{u}, \tilde{\ell}, \tilde{m}, \tilde{h}, \tilde{g}, \tilde{p}, \tilde{q}, \tilde{x},\tilde{t})$ be defined by
\begin{equation}\label{variable}
X=\varphi(\widetilde{X}),\quad Y=\psi(\widetilde{Y}),
\end{equation}
\begin{equation*}
\begin{cases}
(\tilde{u}, \tilde{\ell}, \tilde{m}, \tilde{h}, \tilde{g}, \tilde{x},\tilde{t})(\widetilde{X},\widetilde{Y})=(u, \ell, m, h, g, x,t)(X,Y),\\
\tilde{p}(\widetilde{X},\widetilde{Y})=p(X,Y)\cdot\varphi'(\widetilde{X}),\\
\tilde{q}(\widetilde{X},\widetilde{Y})=q(X,Y)\cdot\psi'(\widetilde{Y}).
\end{cases}
\end{equation*}
It is easy to see that $(\tilde{u}, \tilde{\ell}, \tilde{m}, \tilde{h}, \tilde{g}, \tilde{p},\tilde{q}, \tilde{x},\tilde{t})(\widetilde{X},\widetilde{Y})$ is also the solution to the same system \eqref{semi1}--\eqref{2.18}, and the set
\begin{equation}\label{setre}
\{\big(\tilde{x}(\widetilde{X},\widetilde{Y}),
\tilde{t}(\widetilde{X},\widetilde{Y}),
\tilde{u}(\widetilde{X},\widetilde{Y})\big);~~~~
(\widetilde{X},\widetilde{Y})\in\mathbb{R}^2\}
\end{equation}
concides with the set \eqref{2.20}. We thus derive that the set \eqref{setre} is another graph of the same solution $u(x,t)$ of \eqref{vwl}. One can regard the variable transformation \eqref{variable} simply as a relabeling of forward and backward characteristics, in the solution $u(x,t)$. We refer the readers to \cite{GHR} for more details on the relabeling symmetries, in connection with the Camassa-Holm equation.
\end{Remark}

Our main result in this section is stated  as follows, which extends the Lipschitz property in Lemma \ref{lem_est} to piecewise smooth solutions with generic singularities.
\begin{Theorem}\label{thm_length}
 Let $T>0$ be given, consider a path of solutions $\theta\mapsto \big(u^\theta(t),u^\theta_t(t)\big)$ of \eqref{vwl}, which is piecewise regular for $t\in[0,T]$. Moreover, the total energy is less than some constants $E>0$. Then there exists constants $\kappa_0,\kappa_1,\cdots,\kappa_6$ in \eqref{piecedef} and $C>0$, such that the length satisfies
\begin{equation}\label{pieceest}
\|\gamma^t\|\leq C\|\gamma^0\|,
\end{equation}
where the constant $C$ depends only on $T$ and $E$.
\end{Theorem}

\begin{proof}
Let the piecewise regular path $\theta\mapsto \big(u^\theta(t),u^\theta_t(t)\big)$ be given. From Definition \ref{def_repath}, for every $\theta\in[0,1]\backslash\{\theta_1,\cdots,\theta_N\}$,
 the solution $u^\theta(t)$ has generic regularities for $t\in[0,T]$. More specifically, $u^\theta$ is smooth in the $X$-$Y$ coordinates and piecewise smooth in the $x$-$t$ coordinates, hence the tangent vector is well-defined for all $\theta\in [0,1], t\in[0,T]$.

To prove \eqref{pieceest}, it suffices to show that
\begin{equation}\label{6.2}
\|\big(v^\theta,r^\theta,s^\theta\big)(t)\|_{(u^\theta,R^\theta,S^\theta)(t)}\leq C_1\|\big(v^\theta,r^\theta,s^\theta\big)(0)\|_{(u^\theta,R^\theta,S^\theta)(0)},
\end{equation}
 for $\theta\in[0,1]\backslash\{\theta_1,\cdots,\theta_N\}$,  here $C_1>$ is a constant depending only on $T$ and the upper bound of the total energy.
Indeed, according to Definition \ref{def_piece}, fix $\epsilon>0$ and choose a relabeling of the variables $X,Y$, such that, at time $t=0$, it holds that
\begin{equation*}
\int_0^1\Big\{\sum_{k=0}^6\kappa_k\int_{\Gamma_0^\theta}
\Big(|J_k^\theta|\mathcal{W}^-dX+|H_k^\theta|\mathcal{W}^+dY\Big)\Big\}\,d\theta
\leq \|\gamma^0\|+\epsilon.
\end{equation*}
Integrating \eqref{6.2} over $\theta\in[0,1]$, we have
\begin{equation*}
\|\gamma^t\|\leq C(\|\gamma^0\|+\epsilon),
\end{equation*}
which yields the desired estimate \eqref{pieceest} immediately, since $\epsilon>0$ is arbitrary.

To complete the proof we need to achieve the estimate \eqref{6.2}. Two cases can occur.

{\it CASE 1:} If $u^\theta$ is smooth in the $x$-$t$ coordinates, \eqref{6.2} follows directly from \eqref{normest}.

{\it CASE 2:} If $u^\theta$ is piecewise smooth with generic singularities. In this case, we claim that the appearance of the generic singularities will not affect the estimate \eqref{normest}. Toward this goal, we first observe that there exist at most finitely many points $W_j=(X_j,Y_j), j=1,\cdots,N$ such that the generic conditions \eqref{generic_con} hold when $t\in[0,T]$.
Moreover, for each time $t_j=t(X_j,Y_j)$ corresponding to the point $W_j$, the map
\begin{equation*}
t\mapsto\int_0^1\Big\{\sum_{k=0}^6\kappa_k\int_{\Gamma_t^\theta}
\Big(|J_k^\theta|\mathcal{W}^-dX+|H_k^\theta|\mathcal{W}^+dY\Big)\Big\}\,d\theta
\end{equation*}
is continuous. Hence the metric will not be impacted at (at most finitely) time $t=t_j$ when there exist singularities such that the generic conditions \eqref{generic_con} hold.

On the other hand, at time $t\neq t_j$,  to obtain the estimate \eqref{normest} it suffices to show that the time derivative
\begin{equation*}
\frac{d}{dt}\sum_{k=0}^6\kappa_k\int_{\Gamma_t^\theta}
\Big(|J_k^\theta|\mathcal{W}^-dX+|H_k^\theta|\mathcal{W}^+dY\Big)\end{equation*}
will not be affected by the presence of singularity. Indeed, assume that the solution has the generic singularities along a backward characteristic.
For a fixed time $\tau$ and denote $\Gamma_\tau:=\{(X,Y);~t^\theta(X,Y)$ $=\tau\}$. Let the point $(X_\varepsilon,Y_\varepsilon)$ be the intersection of the curve $\Gamma_{\tau-\varepsilon}=\{(X,Y);~ t^\theta(X,Y)=\tau-\varepsilon\}$ and the singular curve $\{(X,Y); ~h^\theta (X,Y)=0\}$, and the point $(X'_\varepsilon,Y'_\varepsilon)$ be the intersection of the curve $\Gamma_{\tau+\varepsilon}=\{(X,Y);~ t^\theta(X,Y)=\tau+\varepsilon\}$ and the singular curve $\{(X,Y); ~h^\theta (X,Y)=0\}$. In addition, define the curves
\begin{equation*}
\begin{cases}
\Lambda_\varepsilon^+:=\Gamma_{\tau+\varepsilon}\cap\{(X,Y);X\in [X'_\varepsilon,X_\varepsilon]\},\\
\Lambda_\varepsilon^-:=\Gamma_{\tau-\varepsilon}\cap\{(X,Y);X\in [X'_\varepsilon,X_\varepsilon]\},
\end{cases}
\quad
\begin{cases}
\chi_\varepsilon^+:=\Gamma_{\tau+\varepsilon}\cap\{(X,Y);Y\in [Y_\varepsilon,Y'_\varepsilon]\},\\
\chi_\varepsilon^-:=\Gamma_{\tau-\varepsilon}\cap\{(X,Y);Y\in [Y_\varepsilon,Y'_\varepsilon]\}.
\end{cases}
\end{equation*}
Then, it follows that
\begin{equation*}
\lim_{\varepsilon\to 0}\frac{1}{\varepsilon}\Big(\int_{\Lambda_\varepsilon^+}-\int_{\Lambda_\varepsilon^-}\Big)\sum_{k=0}^6 |J_k^\theta|\mathcal{W}^-\,dX=0,
\end{equation*}
\begin{equation*}
\lim_{\varepsilon\to 0}\frac{1}{\varepsilon}\Big(\int_{\chi_\varepsilon^+}-\int_{\chi_\varepsilon^-}\Big)\sum_{k=0}^6 |H_k^\theta|\mathcal{W}^+\,dY=0.
\end{equation*}
The first limit holds since each integrand is continuous  and $|X_\varepsilon -X'_\varepsilon|=O(\varepsilon)$. The second limit holds since each integrand is continuous  and $|Y'_\varepsilon -Y_\varepsilon|=O(\varepsilon)$. Consequently, \eqref{normest} follows even in the presence of singular curve where $h=0$. Similarity, we can obtain the same result in the presence of singular curve where $g=0$. This completes the proof of Theorem \ref{thm_length}.
\end{proof}

\section{Metric for general weak solutions}\label{sec_weak}
Aim of this section is to prove the main Theorem \ref{thm_metric}, by extending the Lipschitz metric to general weak solutions using the generic regularity result in section \ref{gen_sec}. Then we compare our metric with some familiar distance determined by various norms.

\subsection{Construction of geodesic distance}
In this part, we construct a geodesic distance $d(\cdot,\cdot)$ on the space $H^1(\mathbb{R})\times L^2(\mathbb{R})$ and prove the Lipschitz property.
For the sake of convenience, fix any constant $E>0$, we denote a set
\begin{equation*}
\Omega_E:=\{(u,u_t)\in H^1(\mathbb{R})\times L^{2}(\mathbb{R}); ~\mathcal{E}(u,u_t):=\int_\mathbb{R}[\alpha^2u_t^2+\gamma^2u_x^2]\,dx
\leq E\}.
\end{equation*}
Recall the generic regularity theorem \ref{thm_reg}, that is, there exists an open dense set of initial data $\mathcal{M}\subset \Big(\mathcal{C}^3(\mathbb{R})\cap H^1(\mathbb{R})\big)\times \big(\mathcal{C}^2(\mathbb{R})\cap L^2(\mathbb{R})\big)$, such that, for $(u_0,u_1)\in\mathcal{M}$, the conservative solution of \eqref{vwl} has only generic singularities. For future reference, we denote a set $$\mathcal{M}^\infty:=\mathcal{C}_0^\infty\cap\mathcal{M},$$ on which we define a geodesic distance by optimizing over all piecewise regular paths connecting two solutions of \eqref{vwl}. Then by the semilinear system \eqref{semi1}--\eqref{2.18} and Theorem \ref{thm_length}, we can extend this distance from space $\mathcal{M}^\infty$ to a larger space.

\begin{Definition}\label{def_piecepath}
For solutions with initial data in $\mathcal{M}^\infty\cap \Omega_E$, we define the geodesic distance $d\big((u,u_t), (\hat{u},\hat{u}_t)\big)$ as the infimum among the weighted lengths of all piecewise regular paths $\theta\mapsto (u^\theta,u_t^\theta)$, which connect $(u,u_t)$ with $(\hat{u},\hat{u}_t)$, that is, for any time $t$,
\begin{equation*}
\begin{split}
d\big((u,u_t),(\hat{u},\hat{u}_t)\big):=\inf \{ \|\gamma^t\|: &\gamma^t \text{ is a piecewise regular path}, \gamma^t(0)=(u,u_t),\\
& \gamma^t(1)=(\hat{u},\hat{u}_t), \mathcal{E}(u^\theta,u_t^\theta)\leq E, \text{ for all } \theta\in[0,1]\}.
\end{split}\end{equation*}
\end{Definition}

The definition $d(\cdot,\cdot)$ is indeed a distance because  after a suitable re-parameterization, the concatenation of two piecewise regular paths is still a piecewise regular path. Now, we can define the metric for the general weak solutions.

\begin{Definition}\label{def_weak}
Let $(u_0,u_1)$ and $(\hat{u}_0,\hat{u}_1)$ in $H^1(R)\times L^{2}(R)$ be two initial data as required in the existence and uniqueness Theorem \ref{thm_ec}. Denote $u$ and $\hat{u}$ to be the corresponding global weak solutions, then for any time $t$, we define,
\begin{equation*}
d\big((u,u_t),(\hat{u},\hat{u}_t)\big):=\lim_{n\rightarrow \infty}d\big((u^n,u_t^n),(\hat{u}^n,\hat{u}_t^n)\big),
\end{equation*}
for any two sequences of solutions $(u^n,u_t^n)$ and $(\hat{u}^n,\hat{u}_t^n)$ with the corresponding initial data in $\mathcal{M}^\infty\cap \Omega_E$, moreover
\[
\|(u^n_0-u_0, \hat{u}^n_0-\hat{u}_0)\|_{H^1}\rightarrow 0,\quad\hbox{and}\quad
\|(u^n_1-u_1, \hat{u}^n_1-\hat{u}_1)\|_{L^2}\rightarrow 0.
\]
\end{Definition}

We claim that the definition of this metric is well-defined. Indeed, the limit in the definition is independent on the selection of sequences because the solution with initial data in $\mathcal{M}^\infty\cap \Omega_E$ are Lipschitz continuous.

On the other hand, when
\[\|u^n_0-u_0\|_{H^1}\rightarrow 0, \quad \|u^n_1-u_1 \|_{L^2}\rightarrow 0,
\]
by the semi-linear equations \eqref{semi1}--\eqref{2.18},  we can get the corresponding solutions satisfy, for any $t>0$,
\[\|u^n-u\|_{H^1}\rightarrow 0,\quad \|u^n_t-u_t \|_{L^2}\rightarrow 0\]
Thus the Lipschitz property in Theorem \ref{thm_length} can be extended to the general solutions, this in turn yields the main theorem: Theorem \ref{thm_metric}.

\subsection{Comparison with other metrics}
Finally, by some calculations, we study the relations among our distance $d(\cdot,\cdot)$ and other types of metrics.

\begin{Proposition}[Comparison with the Sobolev metric]\label{Prop_sob}
For any two finite energy initial data $(u_0,u_1)$ and $(\hat{u}_0, \hat{u}_1) \in \mathcal M^\infty$, there exists some constant $C$ depends only on the initial energy, such that,
\begin{equation*}
d\big((u_0,u_1),(\hat{u}_0, \hat{u}_1)\big)\leq C\Big(\|u_0-\hat{u}_0\|_{H^1}+\|u_0-\hat{u}_0\|_{W^{1,1}}+
\|u_1-\hat{u}_1\|_{L^1}+\|u_{1}-\hat{u}_{1}\|_{L^2}\Big).
\end{equation*}
\end{Proposition}

\begin{proof}
To find an upper bound of this optimal transport metric, we only have to consider one path  $(u^\theta_0,u^\theta_1)$ connecting  $(u_0,u_1)$ and $(\hat u_0, \hat u_1)$, satisfying the  following conditions
\beq
R^\theta=\theta R+(1-\theta)\hat R,\qquad
S^\theta=\theta S+(1-\theta)\hat S.
\eeq
In fact, it is easy to use above equations to recover a unique path $(u^\theta, u^\theta_t)$; see \eqref{u_rec}. It is easy to check that the energy $\int (\tilde{R}^\theta)^2+(\tilde{S}^\theta)^2\,dx$ is bounded by the energies of $(u_0,u_1)$ and $(\hat u_0, \hat u_1)$.

Then we choose $w=z=0$, so the norm becomes
\begin{equation}
\label{norm_rec}
\begin{split}
\|(v^\theta, w^\theta, \hat r^\theta,z^\theta, \hat{s}^\theta)\|_{(u^\theta,R^\theta,S^\theta)}
&=\kappa_2\int_\mathbb{R}  \Big|v^\theta\Big|\big[(1+(R^\theta)^2)\, (\mathcal{W}^-)^\theta+(1+(S^\theta)^2)\, (\mathcal{W}^+)^\theta\big]\,dx\\
&\quad+\kappa_3\int_\mathbb{R}  \big[|r^\theta|\, (\mathcal{W}^-)^\theta
+|s^\theta|\, (\mathcal{W}^+)^\theta\big]\,dx \\
&\quad+\kappa_6\int_\mathbb{R} \Big[ \Big|2R^\theta r^\theta\Big|\, (\mathcal{W}^-)^\theta+
\Big|2S^\theta s^\theta\Big|\, (\mathcal{W}^+)^\theta\Big]\,dx.
\end{split}
\end{equation}

Now we come to estimate terms in the above equation. It is easy to see that
\beq\label{r_rec}
r^\theta=\frac{d}{d\theta} R^\theta=R-\hat R
\eeq
and
\beq\label{s_rec}
s^\theta=\frac{d}{d\theta} S^\theta=S-\hat S.
\eeq
Finally, we estimate $v^\theta$. First, by \eqref{R-S},
\beq\label{u_rec}
u^\theta_x=\frac{R^\theta-S^\theta}{(c_2-c_1)(x,u^\theta)}.
\eeq
Since the right hand side is Lipschitz on $u^\theta$ and $u^\theta$ has compact support, one can easily prove the existence and uniqueness of $u^\theta(x)$. So $v^\theta=\frac{d}{d\theta} u^\theta$ satisfies
\[
v^\theta_x=\frac{r^\theta-s^\theta}{(c_2-c_1)(x,u^\theta)}
-v^\theta\frac{R^\theta-S^\theta}{(c_2-c_1)^2(x,u^\theta)}(\partial_{u} c_2- \partial_{u} c_1)(x,u^\theta),
\]
then, using \eqref{r_rec} and \eqref{s_rec}, it is easy to see that
\beq\label{v_rec}
|v^\theta|\leq K\, (\| S-\hat S\|_{L^1}+\| R-\hat R\|_{L^1})
\eeq
for some constant $K$.

Using \eqref{norm_rec} and \eqref{r_rec}--\eqref{v_rec}, it is easy to prove this Proposition.
\end{proof}

Using the Lipschitz continuous dependence under Finsler norm, i.e. Theorem \ref{thm_metric}, this proposition tells that
\begin{equation*}
d\big((u,u_t)(t),(\hat{u}, \hat{u})\big)(t)\leq C\Big(\|u_0-\hat{u}_0\|_{H^1}+\|u_0-\hat{u}_0\|_{W^{1,1}}+
\|u_1-\hat{u}_1\|_{L^1}+\|u_{1}-\hat{u}_{1}\|_{L^2}\Big).
\end{equation*}
for any $t\geq0$.
The path used in the proof of Proposition \ref{Prop_sob} is totally different from the one used before in \cite{BC2015}, because in the general case we lose the special structure that variational wave equation holds. The following propositions can be proved in a similar way as in \cite{BC2015}. We add the proofs to make this paper self-contained.

\begin{Proposition}[Comparison with $L^1$ metric]
\label{Prop_L1}
For any solutions $u(t), \hat{u}(t)$ of system (\ref{vwl}) with initial data $u_0,\hat{u}_0\in H^1(\mathbb{R})\cap L^1(\mathbb{R})$ and $u_1,\hat{u}_1\in L^2(\mathbb{R})$, there exists some constant $C$ depends only on the upper bound for the total energy, such that,
\begin{equation}\label{L1}
\|u-\hat{u}\|_{L^1}\leq C\cdot d\Big((u,u_t)(t),(\hat{u},\hat{u}_t)(t)\Big).
\end{equation}
\end{Proposition}

\begin{proof}  Let $\gamma^t:\theta\mapsto \big(u^\theta(t),u^\theta_t(t)\big)$ be a regular path connecting $u(t)$ with $\hat{u}(t)$. By direct calculations, we obtain
\begin{equation*}
\begin{split}
|v|&=\Big|v+\frac{Rw-Sz}{c_2-c_1}-\frac{Rw-Sz}{c_2-c_1}\Big|\leq \Big|v+\frac{Rw-Sz}{c_2-c_1}\Big|+\Big|\frac{Rw}{c_2-c_1}\Big|
+\Big|\frac{-Sz}{c_2-c_1}\Big|\\
&\leq \Big|v+\frac{Rw-Sz}{c_2-c_1}\Big|+\frac{|w|(1+R^2)}{2(c_2-c_1)}+\frac{|z|(1+S^2)}{2(c_2-c_1)},
\end{split}\end{equation*}
which, in combination with the definition \ref{def_weak} and \eqref{norm1}, shows that
\begin{equation*}
\begin{split}
d\Big((u,u_t)(t),(\hat{u},\hat{u}_t)(t)\Big)&\geq C\inf_{\gamma^t}\int_0^1\int_{\mathbb{R}} |v^\theta|\,dx\,d\theta\\
&=C\inf_{\gamma^t}\int_0^1\int_{\mathbb{R}} \big|\frac{d u^\theta}{d\theta}\big|\,dx\,d\theta\geq C\|u-\hat{u}\|_{L^1}.
\end{split}
\end{equation*}
The estimate \eqref{L1} follows from the above argument and we thus complete the proof of Proposition \ref{Prop_L1}.
\end{proof}

\begin{Proposition}[Comparison with the Kantorovich-Rubinstein metric]\label{Prop_KR}
For any solutions $u(t),$ $\hat{u}(t)$ of system (\ref{vwl}) with initial data $u_0,\hat{u}_0\in H^1(\mathbb{R})$ and $u_1,\hat{u}_1\in L^2(\mathbb{R})$, there exists some constant $C$ depends only on the upper bound for the total energy, such that,
\begin{equation}\label{KR}
\sup_{\|f\|_{\mathcal{C}^1}\leq 1}\left|\int f \,d\mu-\int f\,d\hat{\mu}\right|\leq C \cdot d\Big((u,u_t)(t),(\hat{u},\hat{u}_t)(t)\Big),
\end{equation}
where $\mu,\hat{\mu}$ are the measures with densities $\alpha^2(x,u) u_t^2+\gamma^2(x,u) u_x^2$ and $\alpha^2(\hat{x},\hat{u}) \hat{u}_t^2+\gamma^2(\hat{x},\hat{u}) \hat{u}_x^2$ with respect to the Lebesgue measure.
The metric (\ref{KR}) is usually called a Kantorovich-Rubinstein distance, which is equivalent to a Wasserstein distance by a duality theorem \cite{V}.
\end{Proposition}

\begin{proof}
Assume that $\gamma^t:\theta\mapsto \big(u^\theta(t),u^\theta_t(t)\big)$ is a regular path connecting $u(t)$ with $\hat{u}(t)$, and denote  $\mu^\theta$  be the measure with density $\alpha^2(x^\theta,u^\theta) (u^\theta_t)^2+\gamma^2(x^\theta,u^\theta) (u^\theta_x)^2=(\tilde{R}^\theta)^2+(\tilde{S}^\theta)^2$ with respect to the Lebesgue measure, then for any function $f$ with $\|f\|_{\mathcal{C}^1}\leq 1$, it holds that
\begin{equation*}
\begin{split}
&\left|\int_0^1\frac{d}{d\theta}\int f\,d\mu^\theta\,d\theta \right| \leq C\int_0^1\int_{\mathbb{R}} \Big(|w^\theta|(1+(R^\theta)^2)+|z^\theta|(1+(S^\theta)^2)\Big)\,dx\,d\theta\\
&\quad+C\int_0^1\int_{\mathbb{R}} |f|\cdot \big|v^\theta+\frac{R^\theta w^\theta-S^\theta z^\theta}{c_2-c_1}\big|\big[(1+(R^\theta)^2)+(1+(S^\theta)^2)\big]\,dx\,d\theta
\\
&\quad+\int_0^1\int_{\mathbb{R}}|f|\cdot\Big|\frac{-c_1}{c_2-c_1}
\big(2R^\theta(r^\theta+R^\theta_xw^\theta)+(R^\theta)^2w_x^\theta\big)+(R^\theta)^2S^\theta\frac{z^\theta-w^\theta}{c_2-c_1}\partial_u\big(\frac{-c_1}{c_2-c_1}\big)\\
&\quad +\frac{c_2}{c_2-c_1}
\big(2S^\theta(s^\theta+S^\theta_xz^\theta)+(S^\theta)^2z_x^\theta\big) +R^\theta(S^\theta)^2\frac{z^\theta-w^\theta}{c_2-c_1}\partial_u\big(\frac{c_2}{c_2-c_1}\big)\Big|\,dx\,d\theta\\
 &\leq C\int_0^1\int_{\mathbb{R}} \Big\{ \big|v^\theta+\frac{R^\theta w^\theta-S^\theta z^\theta}{c_2-c_1}\big|\big[(1+(R^\theta)^2)+(1+(S^\theta)^2)\big]
 +|w^\theta|(1+(R^\theta)^2)\\
&\quad\qquad\qquad+|z^\theta|(1+(S^\theta)^2)+\big|2R^\theta\hat{r}^\theta +(R^\theta)^2w^\theta_x+\frac{2a_1(w^\theta-z^\theta)}{c_2-c_1}(R^\theta)^2S^\theta\big|
\\
&\quad\qquad\qquad+\big|2S^\theta\hat{s}^\theta+(S^\theta)^2z^\theta_x-\frac{2a_2(w^\theta-z^\theta)}{c_2-c_1}R^\theta (S^\theta)^2\big|\Big\}\,dx\,d\theta,
\end{split}\end{equation*}
where we have used the fact that $\partial_u\big(\frac{-c_1}{c_2-c_1}\big)=\frac{2a_1c_2-2a_2c_1}{c_2-c_1}$ and $\partial_u\big(\frac{c_2}{c_2-c_1}\big)=\frac{2a_2c_1-2a_1c_2}{c_2-c_1}$. Hence, we get estimate \eqref{KR} immediately. This completes the proof of Proposition \ref{Prop_KR}.
\end{proof}

\section*{Acknowledgments}
The first author is partially supported by the National Natural Science Foundation of China (No. 11801295), and the Shandong Provincial Natural Science Foundation, China (No. ZR2018BA008). The second author is partially
supported by NSF with grants DMS-1715012 and DMS-2008504.


\end{document}